\documentclass[11pt]{article}

\usepackage{amssymb}   
\usepackage{amsthm}    
\usepackage{amsmath}   
\usepackage{stmaryrd}  
\usepackage{titletoc}  
\usepackage{mathrsfs}  
\usepackage{graphicx}
\usepackage{xcolor}
\usepackage[hidelinks,bookmarks=false]{hyperref}

\newlength{\defbaselineskip}
\newcommand{\setlinespacing}[1]%
           {\setlength{\baselineskip}{#1 \defbaselineskip}}

\theoremstyle{plain}
\newtheorem{thm}{Theorem}[section]

\newtheorem{lem}[thm]{Lemma}
\newtheorem{prop}[thm]{Proposition}

\theoremstyle{definition}
\newtheorem{defn}{Definition}[section]

\newtheorem{rmk}{Remark}[section]

\newcommand{\eps}{\varepsilon}

\DeclareMathOperator*{\esssup}{esssup}
\DeclareMathOperator*{\essinf}{essinf}

\newcommand{\cO}{\mathcal{O}}

\newcommand{\cL}{\mathcal{L}}
\newcommand{\cT}{\mathcal{T}}

\newcommand{\cB}{\mathcal{B}}
\newcommand{\cA}{\mathcal{A}}
\newcommand{\cS}{\mathcal{S}}

\newcommand{\cG}{\mathcal{G}}

\newcommand{\cU}{\mathcal{U}}

\newcommand{\bH}{\mathbb{H}}
\newcommand{\bP}{\mathbb{P}}
\newcommand{\bR}{\mathbb{R}}

\newcommand{\sF}{\mathscr{F}}
\newcommand{\sP}{\mathscr{P}}

\makeatletter\@addtoreset{equation}{section} \makeatother
 \allowdisplaybreaks
\begin{document}

\title{Existence and Uniqueness of Viscosity Solutions for Stochastic Hamilton--Jacobi--Bellman Equations\footnotemark[1] 
}

\author{Jinniao Qiu\footnotemark[2]  }
\footnotetext[1]{This work was partially supported by the National Science and Engineering Research Council of Canada. }
\footnotetext[2]{Department of Mathematics and Statistics, University of Calgary, 2500 University Drive NW, Calgary, AB T2N 1N4, Canada. \textit{E-mail}: \texttt{jinniao.qiu@ucalgary.ca}. }
 
\maketitle

\begin{abstract}
We study stochastic Hamilton--Jacobi--Bellman equations arising from stochastic optimal control problems with random coefficients and possible degenerateness. 
The viscosity solution is reformulated by using adapted smooth random test fields that touch on stopped space--time neighborhoods   almost surely.
 The state coefficients are predictable, with no continuity assumption
on their dependence on the underlying Brownian path.
Under linear growth of locally Lipschitz drift and diffusion coefficients, 
polynomial growth of the running and terminal cost coefficients, 
and a nonempty separable metric space-valued control set, 
we show that the value function is the unique viscosity solution
in this formulation among adapted continuous fields with
square-integrable random polynomial growth bounds.
The comparison proof constructs adapted fields from controlled
trajectories and spatially shaken conditional values, using a weak
semimartingale decomposition and regularization of its positive
finite-variation part.
In particular, the uniqueness allows for diffusion coefficients being random, degenerate, and depending on time, state, and control, 
whereas earlier uniqueness results of Qiu [SIAM J. Control Optim. 56 (2018), pp. 3708--3730] and Qiu and Zhang [Ann. Appl. Probab. 33 (2023), pp. 889--930]
were obtained under additional structural restrictions on the diffusion coefficient, 
such as determinism combined with either superparabolicity or independence of the state and controls.

\end{abstract}



\section{Introduction}

Stochastic Hamilton--Jacobi--Bellman (HJB) equations arise through
dynamic programming for optimal control problems with random
coefficients, which was first introduced by Peng \cite{Peng_92}. In this pioneering paper, Peng also established
the existence and uniqueness of solutions for a semilinear class of such stochastic HJB equations with
control-independent diffusion coefficients. The general case, however, is fully nonlinear. 
He subsequently identified the development of a generalized notion of viscosity solution as an open problem
\cite[p.~268]{peng1999open}, and revisited the question of general
existence and uniqueness in his 2010 International Congress of
Mathematicians plenary lecture \cite[p.~399]{peng2011backward}.
Building on the framework introduced in \cite{Qiu2018viscosity}, 
the present paper develops a comprehensive theory of existence and uniqueness 
for viscosity solutions of fully nonlinear stochastic HJB equations.


Let $(\Omega,\sF,\bP)$ be a complete probability space carrying an
$m$-dimensional Wiener process $W$, and let $(\sF_t)_{t\geq0}$ be
its augmented natural filtration. We write $E_{\sF_t}$ for
conditional expectation. Fix a terminal time $T\in(0,\infty)$ and
a nonempty separable metric control space $U$, equipped with its
Borel sigma-algebra. The admissible class $\cU$ consists of all
$U$-valued progressively measurable controls. We consider the
optimal control problem
\begin{align}
\inf_{\theta\in\cU}E\left[\int_0^T f(s,X_s,\theta_s)\,ds+G(X_T)\right]
\label{Control-probm}
\end{align}
subject to
\begin{equation}\label{state-proces-contrl}
\left\{
\begin{aligned}
dX_t&=\beta(t,X_t,\theta_t)\,dt+\sigma(t,X_t,\theta_t)\,dW_t,
       &&t\in[0,T],\\
X_0&=x.
\end{aligned}
\right.
\end{equation}
The coefficients $\beta,\sigma,f$ and the terminal cost $G$ may
depend on $\omega\in\Omega$; this dependence is suppressed in the
notation. For $t\ge r$, let $X_t^{r,x;\theta}$ denote the state
starting from $x$ at time $r$ under the control $\theta$. The
dynamic cost and value function are
\begin{align}
J(t,x;\theta)
 &=E_{\sF_t}\left[\int_t^T
 f(s,X_s^{t,x;\theta},\theta_s)\,ds+G(X_T^{t,x;\theta})\right],
 \label{eq-cost-funct}\\
V(t,x)&=\essinf_{\theta\in\cU}J(t,x;\theta),
       \qquad (t,x)\in[0,T]\times\bR^d.
\label{eq-value-func}
\end{align}
For background on stochastic control and dynamic programming, see
\cite{Peng-DPP-1997,yong-zhou1999stochastic}.

The associated dynamic programming equation is the stochastic
HJB equation, as introduced by Peng \cite{Peng_92}, of the following form:
\begin{equation}\label{SHJB}
\left\{
\begin{aligned}
-du(t,x)
 &=\mathbb H(t,x,Du,D^2u,D\psi)\,dt-\psi(t,x)\,dW_t,
 &&(t,x)\in Q:=[0,T)\times\bR^d,\\
u(T,x)&=G(x),&&x\in\bR^d,
\end{aligned}
\right.
\end{equation}
where
\begin{align*}
\mathbb H(t,x,p,A,B)
=\essinf_{v\in U}\bigg\{
\operatorname{tr}\left(\frac12\sigma\sigma'(t,x,v)A
                         +\sigma(t,x,v)B\right)
 +\beta'(t,x,v)p+f(t,x,v)\bigg\},
\end{align*}
for $(p,A,B)\in\bR^d\times\bR^{d\times d}\times\bR^{m\times d}$, and both $u$ and $\psi$
are unknown adapted random fields. 
Examples in applications include utility maximization with habit formation
\cite{EnglezosKaratzas09}, swing option pricing
\cite{bender2016-first-order}, and mean field games
\cite{cardaliaguet2015master}. 

The central difficulty in developing a viscosity theory for \eqref{SHJB} is the presence of the martingale term $\psi\,dW_t$ whose integrand
 $\psi$ is itself unknown. Moreover, 
  $dW_t$ is understood only as notation for integration with respect to Brownian motion in the It\^o sense; 
  it does not denote an ordinary differential or derivative. Thus, to overcome this difficulty, an equivalent formulation 
  was introduced in \cite{Qiu2018viscosity}. Intuitively, one may first 
  write the unknown random filed $u$ in the form of an It\^o process:
\begin{align}
u(t,x)=u(T,x)-\int_t^T\mathfrak d_su(s,x)\,ds
                  -\int_t^T\mathfrak d_\omega u(s,x)\,dW_s.
\label{SDE-u}
\end{align}
Then the uniqueness of the semimartingale decomposition determines its drift and
martingale coefficients up to $d\bP\otimes dt$-null sets, which makes sense of the two linear operators $\mathfrak d_t$ and $\mathfrak d_\omega$. 
Related weak functional calculus  may be found in \cite{cont2013functional,Leao-etal-2018}.
An immediate comparison with \eqref{SHJB} yields that $\psi=\mathfrak d_\omega u$, and that solving \eqref{SHJB} for the pair $(u,\psi)$ is equivalent to 
solving for $u$ the following equation: 
\begin{equation}\label{SHJB-eqiv}
\left\{\begin{aligned}
-\mathfrak d_tu(t,x)
-\mathbb H(t,x,Du(t,x),D^2u(t,x),D\mathfrak d_\omega u(t,x))&=0,
                 &&(t,x)\in Q,\\
u(T,x)&=G(x),&&x\in\bR^d.
\end{aligned}\right.
\end{equation}

The equivalence between \eqref{SHJB} and \eqref{SHJB-eqiv}, 
in which the martingale term is absorbed into the operator $\mathfrak d_\omega$,
 provides a natural foundation for the viscosity formulation. 
 Its development nevertheless raises two challenges: 
 the unknown fields must satisfy nonanticipativity constraints, 
 and the coefficients are only measurable in $\omega$, 
 with no topology imposed on $(\Omega,\sF)$. Consequently, 
 These challenges 
prevent us from defining the viscosity solutions 
in a point-wise manner. 
In \cite{Qiu2018viscosity}, a class of random fields of form \eqref{SDE-u}, with sufficient spacial regularity,
 are used as test functions. 
and the contact classes of test functions are parameterized
by  a stopping time $\tau$, an $\bR^d$-valued $\sF_{\tau}$-measurable variable $\xi$, 
and an event $\Omega_{\tau}\in\sF_{\tau}$. Under this framework, 
 the value function is shown to be a viscosity solution 
 for bounded Lipschitz data, and the maximality is established
  under additional regularity assumptions. 
  The uniqueness
theorem in that work assumes, among other conditions, a
deterministic and superparabolic diffusion coefficient. 
In a subsequent work 
on stochastic differential games \cite{QiuZhang2023games}  
a uniqueness result is proved, allowing for degeneracy when the diffusion is
deterministic and independent of the state and controls;   deterministic superparabolic diffusions are also  addressed
 under
further structural restrictions. More recently,  \cite{liang2026viscosity}  extended this framework
 to 
fully nonlinear stochastic HJB equations for controlled
 jump-diffusion processes, 
 while the uniqueness result (\cite[Theorem 5.14]{liang2026viscosity})
assumes that the diffusion coefficient is deterministic.

The requirement that the diffusion coefficient be deterministic excludes a range of potentially 
relevant applications. One example is portfolio optimization 
under rough volatility, a prominent modeling paradigm in finance (see, e.g., \cite{bayer2023rough}). 
The controlled wealth process may take the form 
(see, e.g., \cite{bauerle2020portfolio})
 \begin{align*}
dX_t
=
X_t(r_t+\theta_t\mu_t)\,dt
+
X_t \theta_t\sqrt{\mathbb V_t}\,dW_t, \text{ for } t\in[0,T]; \quad X_0=x_0,
\end{align*}
where $d=m=1$, $r_t$ is the interest rate, $\mu_t$ denotes the excess return, 
$\theta_t $ is the proportion of wealth invested in the risky asset, 
and $\mathbb V_t$ is the stochastic variance process. 
In typical rough-volatility models, 
$\mathbb V_t$ is driven by fractional Brownian motions with Hurst parameter $H\in(0,1/2)$ and
 is typically neither a semimartingale nor a Markov process. 
Consequently, the portfolio optimization problem is a typical non-Markovian control problem and features a 
random, potentially degenerate controlled diffusion coefficient. It therefore falls outside the scope of existing uniqueness results under deterministic-diffusion assumptions. 
Hence, it becomes interesting and significant to establish a uniqueness result for stochastic HJB equations with random, 
possibly degenerate diffusion coefficients that may depend on the state and control.      

 In this work, we establish both the existence
and uniqueness for stochastic HJB equations with random, possibly
degenerate diffusion coefficients depending on both the state and the
control. The state coefficients are predictable, with no continuity assumption
on their dependence on the underlying Brownian path.
Under $(\cA1)$, the state coefficients may be locally Lipschitz
with linear growth, the costs may have polynomial growth, and the control
space need only be separable. In the viscosity formulation of Definition~\ref{defn-viscosity},
the value function is the unique solution in
$\cS^2_{\mathrm{pol}}$, which allows square-integrable random
polynomial growth bounds
(Theorems~\ref{thm-existence} and~\ref{thm-uniqueness-A1}).
The comparison proof constructs adapted upper fields along controlled
trajectories and lower fields from spatially shaken values, using a weak
semimartingale decomposition and regularization of its positive
finite-variation part. Localization and polynomial barriers constructed
through backward stochastic differential equations (BSDEs) extend comparison from bounded auxiliary problems to the
 polynomial-growth setting. The resulting comparison argument is carried out directly with adapted
random fields; for approaches based on Markovian approximations, refer to
\cite{liang2026viscosity,Qiu2018viscosity,QiuZhang2023games}.

We would also note that the concerned stochastic HJB equation may also be seen
as a class of fully nonlinear backward stochastic partial differential equations (BSPDEs). 
The study of linear BSPDEs goes back to Pardoux's work on filtering
\cite{Pardoux1979}. Regularity and well-posedness for linear and
semilinear equations have been developed in
\cite{DuQiuTang10,DuTangZhang-2013,Hu_Ma_Yong02,Tang-Wei-2013}.
Their representations by forward--backward stochastic differential
equations and their connections with stochastic control are
studied in \cite{Du_Zhang_DegSemilin2012,ma2012non,Tang_98,Zhou_93}.
Related developments include quasilinear equations
\cite{QiuTangMPBSPDE11}, reflected equations
\cite{QiuWei-RBSPDE-2013}, and semilinear stochastic HJB equations in optimal controls and mean-field games
\cite{bensoussan2013mean,cardaliaguet2025mean,meng2023optimal,ye2025fractional}. A Sobolev
weak-solution theory for a special class of fully nonlinear
stochastic HJB equations was established in \cite{Qiu2014weak} and for duality method for

 For the deterministic theory of viscosity solutions,
we refer to
\cite{crandall2000lp,crandall1992user,juutinen2001definition,wang1992-II}.
Stochastic viscosity solutions for forward SPDEs are studied in
\cite{buckdahn2015pathwise,buckdahn2007pathwise,LionsSouganidis1998b}.
A fundamental distinction between these forward equations and backward
stochastic HJB equations such as \eqref{SHJB} concerns the martingale
term. In the former, the noise coefficient is prescribed as part of the
equation, whereas in the latter, the martingale integrand is part of the unknown
solution and is determined jointly with the solution field through
martingale representation.
Path-dependent PDEs provide another framework for non-Markovian control
problems when the coefficients $\beta,\sigma,f$ and $G$ are represented as
deterministic nonanticipative functionals of the state and driving Brownian motions, with the appropriate time and control arguments. For the viscosity
theory of nonlinear path-dependent PDEs, we refer to
\cite{ekren2014viscosity,ekren2016viscosity-1,peng2011note,tang2012path,zhou2023viscosity}.
In their path-dependent formulations, the uniqueness results in
\cite[Theorem~6.4]{ekren2016pseudo} and
\cite[Theorem~7.3]{zhou2023viscosity} allow possibly degenerate
diffusion depending on both the state and the control.
These results impose continuity or Lipschitz conditions on the
coefficients as functionals of the underlying paths.
The present formulation works directly with adapted random fields and
permits merely measurable dependence on $\omega\in\Omega$,
subject to the predictability and spatial regularity assumptions
in $(\cA1)$, with no topology imposed on $\Omega$.
 More recently, in \cite{bank2026duality},
 a duality method for stochastic control with non-Markovian random coefficients is developed, providing
 an alternative approach to characterizing the random value function.

The remainder of the paper is organized as follows. Section~2 introduces
the notation and standing assumptions. Section~3 establishes estimates
for the controlled state process, approximation and regularity properties
of the value function, and the dynamic programming principle.
Section~4 defines test function spaces and the viscosity solution, explore some properties,  and  proves
 the existence of the viscosity solution. 
 Section~5 constructs
upper and lower comparison fields and proves maximality and uniqueness
under $(\cA1)$ in the polynomial-growth class.
Appendix~A recalls the measurable selection theorem used in the proofs,
and Appendix~B gives the proof of Proposition \ref{prop-classical-viscosity}.

\section{Preliminaries}

{We denote by $\sP$ the $\sigma$-algebra of predictable
sets on $\Omega\times[0,T]$ associated with $(\sF_t)_{t\geq0}$.}

 We use $|\cdot|$ for the Euclidean norm of vectors and the Frobenius norm of matrices. For matrices $A,B\in\bR^{d\times d}$, we write $A:B=\operatorname{tr}(AB')$. 
 Define the parabolic distance in $\bR^{1+d}$ as follows:
$$\delta(X,Y):=\max\{ |t-s|^{1/2},|x-y| \},$$
 for $X:=(t,x)$ and $Y:=(s,y)\in \bR^{1+d}$. Denote by $Q^+_r(X)$ the hemisphere of radius $r>0$ and center $X:=(t,x)\in \bR^{1+d}$ with $x\in \bR^d$:
\begin{equation*}
  \begin{aligned}
    Q^+_r(X):=\,  [t,t+r^2)\times B_r(x), \quad
    B_r(x):=\, \{ y\in\bR^d:|y-x|<r \},
  \end{aligned}
\end{equation*}
and by $|Q^+_r(X)|$ the volume. Throughout this paper, we write $(s,y)\rightarrow (t^+,x)$, meaning that $s \downarrow t$ and $y\rightarrow x$.

Let $\mathbb B $ be a Banach space equipped with norm $\|\cdot\|_{\mathbb B }$. Banach-valued $L^p$ spaces are understood in the strongly measurable (Bochner) sense. For each $t\in[0,T]$, denote by $L^0(\Omega,\sF_t;\mathbb B)$ the space of $\mathbb B$-valued $\sF_t$-measurable random variables. For $p\in[1,\infty]$, $\cS ^p ({\mathbb B })$ is the set of all the ${\mathbb B }$-valued,
 $\sP$-measurable continuous processes $\{\mathcal X_{t}\}_{t\in [0,T]}$ such
 that
{\small $$\|\mathcal X\|_{\cS ^p({\mathbb B })}:= \left\|\sup_{t\in [0,T]} \|\mathcal X_t\|_{\mathbb B }\right\|_{L^p(\Omega,\sF,\bP)}< \infty.$$
}
 For $1\leq p<\infty$, denote by $\mathcal{L}^p({\mathbb B })$ the totality of all  the ${\mathbb B }$-valued,
  $\sP$-measurable processes $\{\mathcal X_{t}\}_{t\in [0,T]}$ such
 that
 {\small
 $$
 \|\mathcal X\|_{\mathcal{L}^p({\mathbb B })}:=\left\| \bigg(\int_0^T \|\mathcal X_t\|_{\mathbb B }^p\,dt\bigg)^{1/p} \right\|_{L^p(\Omega,\sF,\bP)}< \infty.
 $$
 }
For $p=\infty$, define $\mathcal L^\infty(\mathbb B)$ using the norm
$$\|\mathcal X\|_{\mathcal L^\infty(\mathbb B)}
:=\mathop{\mathrm{ess\,sup}}_{(\omega,t)\in\Omega\times[0,T]}
\|\mathcal X_t(\omega)\|_{\mathbb B},$$
where the essential supremum is with respect to $\bP\otimes dt$.
Obviously, $(\cS^p({\mathbb B }),\,\|\cdot\|_{\cS^p({\mathbb B })})$ and $(\mathcal{L}^p({\mathbb B }),\|\cdot\|_{\mathcal{L}^p({\mathbb B })})$
are Banach spaces.

For each $(k,q)\in \mathbb{N}_0\times [1,\infty]$ we define  the $k$-th Sobolev space $(H^{k,q},\|\cdot\|_{k,q})$ as usual, and for each domain $\cO\subset \bR^d$, denote by $C^k(\cO)$ the space of functions with the up to $k$-th order derivatives being bounded and continuous on $\cO$, $C^k_0(\cO) $ being the subspace of $C^k(\cO)$ vanishing on the boundary $\partial \cO$. On $\bR^d$, $C^k_0(\bR^d)$ denotes the functions in $C^k(\bR^d)$ whose derivatives of orders zero through $k$ vanish at spatial infinity. When $k=0$, write $C_0(\cO) $ and $C(\cO)$ simply. Through this paper, we define 
$$
\cS^p(C_{loc}(\bR^d))=\cap_{N>0} \cS^{p}(C(B_N(0))),\quad \text{ for } p\in[1,\infty].
$$

By convention, we treat elements of spaces like $\cS^p(H^{k,q})$ and $\cL^p(H^{k,q})$ as functions rather than distributions or classes of equivalent functions, and if a function of such class admits a version with better properties, we always denote this version by itself. For example, if $u\in \cL^p(H^{k,q})$ and $u$ admits a version lying in $\mathcal{S}^p(H^{k,q})$, we always adopt the modification $u\in \cL^p(H^{k,q})\cap \cS^p(H^{k,q})$.
\medskip

{
{Throughout the paper, we write $\varpi(x)=(1+|x|^2)^{1/2}$ for the spatial weight, with the same convention for scalar arguments.} For $R>0$, let $B_R=B_R(0)$ denote the open ball of radius $R$ centered at the origin, and define the radial projection $\pi_R:\bR^d\to\overline{B_R}$ by
\[
\pi_R(x)=
\begin{cases}
x, & |x|\le R,\\
R\dfrac{x}{|x|}, & |x|>R.
\end{cases}
\]

The standing assumption is
the following.

\medskip
\noindent$(\cA1)$
(i) The control space $U$ is a nonempty separable metric space with
its Borel sigma-algebra. The coefficients $\beta,\sigma,f$ are
$\sP\otimes\cB(\bR^d)\otimes\cB(U)$-measurable, and $G$ is
$\sF_T\otimes\cB(\bR^d)$-measurable. For almost every $(\omega,t)$,
$\beta,\sigma,f$ are continuous in $(x,v)$; $G$ is continuous in $x$
  a.s..

(ii) There is a deterministic $K<\infty$ such that
\[
 |\beta(t,x,v)|+{|\sigma(t,x,v)|}
 \le K{\varpi(x)}.
\]
For each $R<\infty$ there is a deterministic $L_R<\infty$ such that
\[
 |\beta(t,x,v)-\beta(t,y,v)|+
 {|\sigma(t,x,v)-\sigma(t,y,v)|}
 \le L_R|x-y|,\qquad x,y\in\overline B_R, v\in U.
\]

(iii) There are deterministic $C<\infty$ and $p_0\ge0$ such that
\[
 |f(t,x,v)|\le C{\varpi(x)}^{p_0},\qquad
 |G(x)|\le C{\varpi(x)}^{p_0}.
\]
For every $R<\infty$, almost everywhere in $(\omega,t)$,
\begin{equation}\label{eq:A1-cost-modulus}
 \lim_{a\downarrow0}\sup_{v\in U}
 \sup_{\substack{x,y\in\overline B_R\\|x-y|\le a}}
 |f(t,x,v)-f(t,y,v)|=0.
\end{equation}
All bounds hold outside common null sets, simultaneously in the spatial
and control variables. There is no restriction on the growth of $L_R$
as $R\to\infty$. Local spatial Lipschitz continuity of $f$ with constants
independent of $v$ implies \eqref{eq:A1-cost-modulus}; a deterministic
modulus is not required. For compact $U$, joint continuity already
implies \eqref{eq:A1-cost-modulus}. Separability and continuity in $v$
allow all control infima to be taken over a fixed countable dense set;
no attaining minimizer is assumed.

We use the following stronger hypothesis only for bounded auxiliary
problems and for results that explicitly invoke it.
\medskip

\noindent$(\cA1_{\mathrm b})$ Assumption $(\cA1)$ holds, and there
is a deterministic $L<\infty$ such that
\begin{align*}
 |G(x)|+\sup_{v\in U}
 \bigl(|f(t,x,v)|+|\beta(t,x,v)|+{|\sigma(t,x,v)|}\bigr)
 &\le L,\\
 |G(x)-G(y)|+\sup_{v\in U}\bigl(
 |f(t,x,v)-f(t,y,v)|+|\beta(t,x,v)-\beta(t,y,v)|
 &\\[-6pt]
 {}+{|\sigma(t,x,v)-\sigma(t,y,v)|}\bigr)
 &\le L|x-y|.
\end{align*}
Compactness of $U$ is not part of $(\cA1_{\mathrm b})$.

The polynomial-growth class used for uniqueness is
\begin{equation}\label{eq:polynomial-class}
 \cS^2_{\mathrm{pol}}
 :=\bigcup_{r\ge0}\left\{u\in\cS^2(C_{\mathrm{loc}}(\bR^d)):
 E\left[\sup_{0\le t\le T,\,x\in\bR^d}
 \frac{|u(t,x)|^2}{{\varpi(x)}^{2r}}\right]<\infty\right\}.
\end{equation}
This class permits a random square-integrable polynomial growth bound.
The degree $r$ of a competing solution need not equal $p_0$.
}

Finally, we introduce an auxiliary result on the Snell envelope of a continuous adapted process. 
As it follows straightforwardly from \cite[Propositions~2.3, 3.5 and~5.1]
{El_Karoui-reflec-1997}, we omit the proof here. 
\begin{lem}[Scalar Snell envelope]\label{lem-snell-envelope}
Let $0\le a<T$ be deterministic, and let $Y$ be a real-valued
continuous adapted process on $[a,T]$ such that
$E(Y^*)^2<\infty$, where $Y^*=\sup_{a\le s\le T}|Y_s|$.
Under the standing augmented Brownian filtration, the following hold.
\begin{enumerate}
\item[(i)] The Snell envelope $Z$ of $Y$ is a continuous
supermartingale. For every $[a,T]$-valued stopping time $\tau$,
\[
 Z_\tau=\esssup_{\rho:\,\tau\le\rho\le T}
                 E_{\sF_\tau}Y_\rho,
\]
where the essential supremum is taken over stopping times $\rho$.
In particular, $Z\ge Y$ and $Z_T=Y_T$.

\item[(ii)] The envelope has the Doob--Meyer decomposition
\[
 Z_s=Z_a+M_s-K_s,\qquad
 M_s=\int_a^s\zeta_r\,dW_r,\qquad a\le s\le T,
\]
where $\zeta$ is predictable and square-integrable, and $K$ is
continuous, adapted and increasing, with $K_a=0$. Moreover,
there is a constant $C$, depending only on $T$, such that
\begin{equation}\label{eq:contact-snell-bounds}
 E\sup_{a\le s\le T}|Z_s|^2
 +E\int_a^T|\zeta_s|^2\,ds+EK_T^2
 \le C\,E(Y^*)^2.
\end{equation}

\item[(iii)] For every $[a,T]$-valued stopping time $\tau$, the
first contact time
\[
 \rho_\tau^*:=\inf\{s\in[\tau,T]:Z_s=Y_s\}
\]
is optimal: $Z_\tau=E_{\sF_\tau}Y_{\rho_\tau^*}$.
\end{enumerate}
\end{lem}


\section{Some properties of the value function and dynamic programming principle}

We first present some estimates for the state process and the standard proof is omitted.

\begin{lem}\label{lem-SDE}
Let Assumption $(\cA1)$ hold. Then the following assertions hold.
\begin{enumerate}
\item[(i)] For each $\theta\in\cU$, the controlled SDE
\eqref{state-proces-contrl} has a unique nonexplosive strong solution,
even when it is started at a stopping time $\tau$ from a finite
$\sF_\tau$-measurable state $\xi$.

\item[(ii)] For every $p>0$, the solution in {\rm (i)} satisfies
\begin{align}
 E_{\sF_\tau}\sup_{\tau\le s\le T}
       {\varpi(X_s^{\tau,\xi;\theta})}^p
 &\le C_p{\varpi(\xi)}^p,\label{eq-poly-state-moment}\\
 E_{\sF_\tau}\sup_{\tau\le s\le(\tau+h)\wedge T}
       |X_s^{\tau,\xi;\theta}-\xi|^p
 &\le C_p{\varpi(\xi)}^p h^{p/2},\qquad 0<h\le T,
 \label{eq-poly-state-increment}
\end{align}
where $C_p$ depends only on $p,T$ and the linear-growth constant
in $(\cA1)$. 
Here, we note that for merely finite $\xi$, the nonnegative conditional
expectations are understood in the extended sense, or equivalently
by $\sF_\tau$-localization. 
 More precisely, for any nonnegative random variable $Z$,
we define
\[
 E_{\sF_\tau}Z
 :=\lim_{k\to\infty}E_{\sF_\tau}[Z\wedge k]\in[0,\infty],
\]
where the limit is increasing. For localization, put
$A_n=\{|\xi|\le n\}\in\sF_\tau$ and $\xi_n=\xi\mathbf1_{A_n}$.
Let $Z(\xi)$ denote either of the nonnegative suprema on the
left-hand sides of \eqref{eq-poly-state-moment} and
\eqref{eq-poly-state-increment}, with $\tau,\theta$ and, when
applicable, $h$ fixed. The  uniqueness of the strong solution in (i) implies that
$Z(\xi)=Z(\xi_n)$ on $A_n$. Hence
\[
 E_{\sF_\tau}[\mathbf1_{A_n}Z(\xi)]
 =\mathbf1_{A_n}E_{\sF_\tau}Z(\xi_n).
\]
The right-hand side is an ordinary integrable conditional
expectation. The estimates for bounded initial states bound it by
$C_p\mathbf1_{A_n}{\varpi(\xi)}^p$ for
\eqref{eq-poly-state-moment}, and by
$C_p\mathbf1_{A_n}{\varpi(\xi)}^p h^{p/2}$ for
\eqref{eq-poly-state-increment}.
Since $A_n\uparrow\Omega$   a.s., conditional monotone
convergence yields the required estimates.
The displayed bounds make them finite
  a.s.; the corresponding suprema are integrable whenever
$E{\varpi(\xi)}^p<\infty$.

\item[(iii)] If $X=X^{\tau,\xi;\theta}$ and
$\widehat X=X^{\tau,\hat\xi;\theta}$ start from
$\sF_\tau$-measurable $\overline B_R$-valued states $\xi,\hat\xi$ and
use the same control, define
\[
 \rho_R=\inf\{s\in[\tau,T]:|X_s|\vee|\widehat X_s|\ge R\}\wedge T,
 \qquad \inf\varnothing=\infty.
\]
Then, for every $p>0$,
\begin{equation}\label{eq-poly-state-stability}
 E_{\sF_\tau}\sup_{\tau\le s\le\rho_R}
 |X_s^{\tau,\xi;\theta}-X_s^{\tau,\hat\xi;\theta}|^p
 \le C_{p,R}|\xi-\hat\xi|^p,
\end{equation}
where $C_{p,R}$ depends only on $p,T$ and the local Lipschitz
constant $L_R$. If either initial state lies on $\partial B_R$,
then $\rho_R=\tau$ and the estimate is immediate.
Under $(\cA1_{\mathrm b})$, the estimate holds up to $T$, with
a constant independent of $R$.

\item[(iv)] The solution has the following flow property: for each fixed
$\theta,\tau,\xi$ and stopping time $\rho$ with $\tau\le\rho\le T$,
\[
 X_s^{\tau,\xi;\theta}
 =X_s^{\rho,X_\rho^{\tau,\xi;\theta};\theta},
 \qquad \rho\le s\le T,\quad\text{a.s.}
\] 
\end{enumerate}
\end{lem}

Next is a bounded-data approximation result that is used in various applications.

\begin{lem}[Bounded-data approximation]\label{lem-poly-approximation}
{Let Assumption $(\cA1)$ hold.}
Let $\pi_N$ denote the radial projection onto $\overline B_N$,
$N\ge1$, and define
\begin{align*}
 \beta_N(t,x,v)&=\beta(t,\pi_Nx,v),&
 \sigma_N(t,x,v)&=\sigma(t,\pi_Nx,v),\\
 f_N(t,x,v)&=f(t,\pi_Nx,v),&G_N(x)&=G(\pi_Nx).
\end{align*}
For a nonnegative smooth spatial mollifier $\rho_h$ supported in
$B_h$, $0<h\le1$, set
$f_{N,h}=f_N*\rho_h$ and $G_{N,h}=G_N*\rho_h$.
Write $J_N,V_N$ and
$J_{N,h},V_{N,h}$ for the corresponding costs and values. Set
\begin{align}
 a_{N,h}(t)&=\sup_{x\in\bR^d,\,v\in U}
       |f_{N,h}(t,x,v)-f_N(t,x,v)|,\nonumber\\
 b_{N,h}&=\sup_{x\in\bR^d}|G_{N,h}(x)-G_N(x)|,\nonumber\\
 \eta_{N,h}(t)&=E_{\sF_t}
       \left[b_{N,h}+\int_t^T a_{N,h}(s)\,ds\right].
 \label{eq-poly-cost-correction}
\end{align}
{Then the following assertions hold.}
\begin{enumerate}
\item[{(i)}] The data
$(\beta_N,\sigma_N,f_{N,h},G_{N,h})$ satisfy
$(\cA1_{\mathrm b})$ for each fixed $N,h$.

\item[{(ii)}] {The quantities
$a_{N,h},b_{N,h}$ and $\eta_{N,h}$ are measurable,} and
\begin{equation}\label{eq-poly-eta-convergence}
 -d\eta_{N,h}=a_{N,h}\,dt-z_{N,h}\,dW,
 \qquad \eta_{N,h}\ge0,
 \qquad \|\eta_{N,h}\|_{\cS^2(\bR)}\longrightarrow0
 \quad(h\downarrow0).
\end{equation}
\item[{(iii)}] For every fixed control, the costs and the values have jointly
 (w.r.t. $(t,x)$) continuous versions.
{The exceptional null set for a fixed-control field
may depend on that control.}
For these versions,
\begin{equation}\label{eq-poly-cost-mollification}
\begin{aligned}
 |J_{N,h}(t,x;\theta)-J_N(t,x;\theta)|
   &\le\eta_{N,h}(t),\\
 |V_{N,h}(t,x)-V_N(t,x)|&\le\eta_{N,h}(t),
\end{aligned}
\end{equation}
\begin{equation}\label{eq-poly-exit-cost}
\begin{aligned}
 |J(t,x;\theta)-J_N(t,x;\theta)|
   &\le C_a{\varpi(x)}^{p_0+a}N^{-a},\\
 |V(t,x)-V_N(t,x)|
   &\le C_a{\varpi(x)}^{p_0+a}N^{-a}
\end{aligned}
\end{equation}
for every $a>0$ and $|x|<N$, with $C_a$ independent of $N$ and
$\theta$. {Both estimates hold at stopping times $\tau$
and bounded $\sF_\tau$-measurable initial states $\xi$, with
$|\xi|<N$ a.s. for \eqref{eq-poly-exit-cost}.}
\item[{(iv)}] All these costs and values satisfy 
\begin{equation}\label{eq-poly-value-bound}
 |V(t,x)|+|J(t,x;\theta)|\le C_0{\varpi(x)}^{p_0},
\end{equation}
and the same estimate, with a uniform constant, holds for the
projected and mollified problems.

\item[{(v)}] In particular,
$V_{N,h}\to V_N$ in $\cS^2(C(\overline B_R))$ as $h\downarrow0$,
and $V_N\to V$ in $\cS^2(C(\overline B_R))$ as $N\uparrow\infty$, for every $R$.
\end{enumerate}
\end{lem}

\begin{proof}
{\noindent\textbf{(i).}} The projection is $1$-Lipschitz. Thus $\beta_N,\sigma_N$ are bounded
and globally Lipschitz, and retain a common linear-growth constant.
The projected costs are bounded by $C{\varpi(N)}^{p_0}$.
Spatial mollification gives
deterministic global Lipschitz bounds for the mollified costs, while
preserving continuity in the control. {Projection and convolution also
preserve the required measurability. Hence these data satisfy
$(\cA1_{\mathrm b})$.}

{\noindent\textbf{(ii).}} The local spatial equicontinuity of $f$, uniformly in $v$, implies
uniform spatial continuity of $f_N$ on all of $\bR^d$, uniformly in $v$, a.s. The same assertion for $G_N$ follows from
continuity on the compact ball. Consequently $a_{N,h}\to0$
 $\bP\otimes dt$-a.e., and $b_{N,h}\to0$ a.s. Their bounds
$2C{\varpi(N)}^{p_0}$ permit dominated convergence. By continuity, suprema
in \eqref{eq-poly-cost-correction} can be taken over fixed countable
dense sets in $\bR^d$ and $U$, respectively. This proves
measurability.
As dominated convergence gives
$\|b_{N,h}\|_{L^2(\Omega)}+\|a_{N,h}\|_{\cL^2}\to0$,
the standard BSDE estimate for the linear driver $a_{N,h}$ (see \cite{Karoui_Peng_Quenez} for instance) yields
\[
\begin{aligned}
 \|\eta_{N,h}\|_{\cS^2}
 &\le C\left(\|b_{N,h}\|_{L^2(\Omega)}
             +\,\|a_{N,h}\|_{\cL^2}\right)
 \longrightarrow0,
\end{aligned}
\]
proving the remaining assertion of \eqref{eq-poly-eta-convergence}.

{\noindent\textbf{(iii).} Under $(\cA1_{\mathrm b})$,
the argument of \cite[Proposition 3.3 and Appendix B]{Qiu2018viscosity}
{gives the deterministic-time conditional-cost and value
identities, joint continuity, spatial Lipschitz estimates, and
uniform bounds for the bounded cost and value fields}
for the present separable metric control space, as it just uses uniform bounds
and spatial Lipschitz estimates, Brownian martingale continuity, and
measurable pasting, without requiring compactness of the control space $U$. 
At deterministic
times, substitution of random initial states follows by simple-state
approximation, SDE stability, and spatial Lipschitz continuity
of the costs; countable pasting on $\sF_t$-events gives the
essential-infimum identity. These arguments use neither a
stopping-time DPP nor the polynomial-growth assertions of this lemma.}

Here is the stopping-time extension needed for those bounded problems.
For bounded $\xi\in L^0(\sF_\tau)$, choose finite-valued
$\tau_n\downarrow\tau$ with $\tau_n-\tau\le2^{-n}T$.
Boundedness and global Lipschitz continuity give $L^1$ convergence of
the full costs started from $(\tau_n,\xi)$ to those started from
$(\tau,\xi)$, uniformly in the control, with error at most
$C2^{-n/2}$.
{Write
\[
 H(r,\zeta;\theta)
 =\int_r^T f(s,X_s^{r,\zeta;\theta},\theta_s)\,ds
       +G(X_T^{r,\zeta;\theta})
\]
for the full cost before taking conditional expectations.
For each fixed $\theta$, put $H_n=H(\tau_n,\xi;\theta)$
and $H=H(\tau,\xi;\theta)$. Right-continuity of the filtration
gives $\sF_{\tau_n}\downarrow\sF_\tau$, so continuity of the martingale and dominated convergence yield
\[
\begin{aligned}
 \|E_{\sF_{\tau_n}}H_n-E_{\sF_\tau}H\|_{L^1}
 &\le \|H_n-H\|_{L^1}
       +\|E_{\sF_{\tau_n}}H-E_{\sF_\tau}H\|_{L^1}\\
 &\longrightarrow0.
\end{aligned}
\]
The finite-valued stopping-time identity gives
$J(\tau_n,\xi;\theta)=E_{\sF_{\tau_n}}H_n$.
Continuity and boundedness of the cost field also give
$J(\tau_n,\xi;\theta)\to J(\tau,\xi;\theta)$ in $L^1$,
which identifies $J(\tau,\xi;\theta)=E_{\sF_\tau}H$.}

At a deterministic time, downward
directedness of conditional costs gives a decreasing sequence tending
to their essential infimum. Selecting the first index whose cost is
within a prescribed positive constant of that infimum gives an
admissible control by countable pasting; refer to \cite[Proposition 3.3]{Qiu2018viscosity}.
 This also applies at each
finite-valued $\tau_n$.
 Set
\[
 Q=\essinf_{\theta\in\cU}E_{\sF_\tau}H(\tau,\xi;\theta).
\]
For each fixed $\theta$, passage to the limit in
$V(\tau_n,\xi)\le J(\tau_n,\xi;\theta)$ gives
$V(\tau,\xi)\le J(\tau,\xi;\theta)$, and hence
$V(\tau,\xi)\le Q$.
Choose $\theta_n\in\cU$ with
$J(\tau_n,\xi;\theta_n)\le V(\tau_n,\xi)+1/n$.
Past an arbitrary control on $[\tau,\tau_n)$ and denote
the resulting control by $\bar\theta_n$. The uniform full-cost
estimate above gives
\[
\begin{aligned}
 EQ
 \le E H(\tau,\xi;\bar\theta_n)
 &\le E H(\tau_n,\xi;\bar\theta_n)+C2^{-n/2}\\
 &=EJ(\tau_n,\xi;\theta_n)+C2^{-n/2}\\
 &\le EV(\tau_n,\xi)+1/n+C2^{-n/2}.
\end{aligned}
\]
Boundedness and continuity of $V$ imply $EQ\le EV(\tau,\xi)$.
Together with $V(\tau,\xi)\le Q$, this proves $Q=V(\tau,\xi)$
  a.s., i.e., 
\[
 V(\tau,\xi)=\essinf_{\theta\in\cU}
 E_{\sF_\tau}\left[\int_\tau^T
 f(s,X_s^{\tau,\xi;\theta},\theta_s)\,ds
 +G(X_T^{\tau,\xi;\theta})\right]
\]
for the bounded problem. Locality argument extends this identity to every
finite $\xi$. Th establishes the bounded stopping-time
identity independently of the general results below.

The state dynamics are unchanged by cost mollification, so comparison
of the full costs gives \eqref{eq-poly-cost-mollification} first at
deterministic points, uniformly in the control. Taking essential
infima gives its value-function version.
{For fixed $N,h,h'$, first take a countable dense
set of deterministic $(t,x)$. 
Then, at all these points, it holds that
\[
 |V_{N,h}(t,x)-V_{N,h'}(t,x)|
 \le\eta_{N,h}(t)+\eta_{N,h'}(t),\quad \text{ a.s.}
\]
 Continuity of the two bounded-data
value fields and of $\eta_{N,h},\eta_{N,h'}$ extends this
inequality to every $(t,x)$. Consequently,
\[
 \|V_{N,h}-V_{N,h'}\|_{\cS^2(C(\overline B_R))}
 \le\|\eta_{N,h}\|_{\cS^2(\bR)}+\|\eta_{N,h'}\|_{\cS^2(\bR)}
 \longrightarrow0,\quad \text{as }h,h'\downarrow0.
\]
  This  gives a continuous limit in $\cS^2(C(\overline B_R))$.
At each deterministic point, \eqref{eq-poly-cost-mollification}
identifies that limit with $V_N$. A countable exhaustion by
balls gives a jointly continuous version on all of
$[0,T]\times\bR^d$. The same argument applies to each
fixed-control cost $J_N$.
Passing to the limit also gives \eqref{eq-poly-cost-mollification}
  for these versions.}
{To identify these versions at a stopping time $\tau$
and a bounded $\sF_\tau$-measurable state $\xi$, let
$H_N(\tau,\xi;\theta)$ denote the full cost for the projected
problem and set
\[
 C_N^\theta=E_{\sF_\tau}H_N(\tau,\xi;\theta),\qquad
 Q_N=\essinf_{\theta\in\cU}C_N^\theta.
\]
Optional sampling gives
\[
 \eta_{N,h}(\tau)=E_{\sF_\tau}\left[
 b_{N,h}+\int_\tau^T a_{N,h}(s)\,ds\right].
\]
Thus the full-cost comparison and the bounded stopping-time
identities give
\[
 |J_{N,h}(\tau,\xi;\theta)-C_N^\theta|\le\eta_{N,h}(\tau),
 \qquad
 |V_{N,h}(\tau,\xi)-Q_N|\le\eta_{N,h}(\tau).
\]
Combining these with \eqref{eq-poly-cost-mollification}, which
holds for all $(t,x)$ a.s., yields
\[
 |J_N(\tau,\xi;\theta)-C_N^\theta|\le2\eta_{N,h}(\tau),
 \qquad
 |V_N(\tau,\xi)-Q_N|\le2\eta_{N,h}(\tau).
\]
Since $\|\eta_{N,h}(\tau)\|_{L^2}\le
\|\eta_{N,h}\|_{\cS^2(\bR)}\to0$ as $h\downarrow0$, these versions coincide with the
conditional costs and their essential infima at $(\tau,\xi)$.}

{Fix a stopping time $\tau$, a bounded
$\sF_\tau$-measurable state $\xi$ with $|\xi|<N$, and a control
$\theta$. Couple the original state $X$ and the projected state
$X^N$ using the same Brownian motion and control, both starting
from $(\tau,\xi)$. They agree up to their common first hitting
time of radius $N$. Hence, 
\[
 E_N=\left\{\sup_{\tau\le s\le T}|X_s|\ge N\right\}
    =\left\{\sup_{\tau\le s\le T}|X_s^N|\ge N\right\}, \quad \text{a.s.}
\]
  Put
\[
 M=\sup_{\tau\le s\le T}{\varpi(X_s)},\qquad
 M_N=\sup_{\tau\le s\le T}{\varpi(X_s^N)},
\]
and let $H,H_N$ denote the full costs for the original and
projected data. These costs agree on $E_N^c$. Their growth
bounds and $M,M_N\ge N$ on $E_N$ give, for every $a>0$,
\[
\begin{aligned}
 |H-H_N|
 &\le C(M^{p_0}+M_N^{p_0})\mathbf1_{E_N}\\
 &\le CN^{-a}(M^{p_0+a}+M_N^{p_0+a}).
\end{aligned}
\]
The original and projected dynamics have a common linear-growth
constant. Lemma~\ref{lem-SDE} yields
\[
 E_{\sF_\tau}|H-H_N|
 \le C_aN^{-a}{\varpi(\xi)}^{p_0+a},
\]
uniformly in $N$ and $\theta$. At deterministic initial points
this gives \eqref{eq-poly-exit-cost} for costs. Taking essential
infima gives the same estimate for the values.}
{For fixed $R,a>0$ and $N,n>R$, comparison through
$V$ at a countable dense set of deterministic points gives
\[
 \sup_{\substack{0\le t\le T\\|x|\le R}}
 |V_N(t,x)-V_n(t,x)|
 \le C_a(1+R^2)^{(p_0+a)/2}(N^{-a}+n^{-a}),
\]
where continuity of $V_N,V_n$ extends the estimate to every point.
Taking integers $N,n>R$ and a countable exhaustion by balls constructs
a locally uniform continuous (adapted) limit. The deterministic-point
estimate \eqref{eq-poly-exit-cost} identifies this limit with $V$.
The same argument applies to each fixed-control cost $J$.
Continuity then extends \eqref{eq-poly-exit-cost} to all $(t,x)$
with $|x|<N$ for these versions.

For a stopping time $\tau$ and a bounded $\sF_\tau$-measurable
state $\xi$, set $C^\theta=E_{\sF_\tau}H$ and
$Q=\essinf_{\theta\in\cU}C^\theta$.
For $N>\|\xi\|_{L^\infty}$, the conditional coupling estimate gives
\[
 |C_N^\theta-C^\theta|\le C_aN^{-a}{\varpi(\xi)}^{p_0+a},
 \qquad
 |Q_N-Q|\le C_aN^{-a}{\varpi(\xi)}^{p_0+a}.
\]
Since $C_N^\theta=J_N(\tau,\xi;\theta)$ and
$Q_N=V_N(\tau,\xi)$, local uniform convergence identifies
$J(\tau,\xi;\theta)=C^\theta$ and $V(\tau,\xi)=Q$.}

{\noindent\textbf{(iv).}
Since $|\pi_Nx|\le |x|$ and
${\varpi(x-y)}\le2{\varpi(x)}$ for $|y|\le h\le1$,
the projected and mollified costs satisfy
\[
 |f_{N,h}(t,x,v)|+|G_{N,h}(x)|
 \le C'{\varpi(x)}^{p_0},
\]
with $C'$ independent of $N,h$; the same bound holds for
$f_N,G_N$ and the original costs.
The common linear-growth bound for the state coefficients and
Lemma~\ref{lem-SDE} therefore give
\[
 |J(t,x;\theta)|
 \le C E_{\sF_t}\left[
       \int_t^T{\varpi(X_s^{t,x;\theta})}^{p_0}\,ds
       +{\varpi(X_T^{t,x;\theta})}^{p_0}\right]
 \le C''{\varpi(x)}^{p_0},
\]
uniformly in $\theta$, and likewise for $J_N,J_{N,h}$, with
the same constant. For $p_0=0$, this estimate follows directly
from boundedness of the costs. Taking essential infima gives
the corresponding bounds for the values. Continuity extends
these bounds from a countable dense set of deterministic points
to all $(t,x)$, which proves \eqref{eq-poly-value-bound} and its
approximating versions.}

{\noindent\textbf{(v).}
For fixed $N$ and $R$, assertions (ii) and (iii) yield
\[
 \|V_{N,h}-V_N\|_{\cS^2(C(\overline B_R))}
 \le\|\eta_{N,h}\|_{\cS^2(\bR)}\longrightarrow0
 \qquad(h\downarrow0).
\]
For any fixed $a>0$ and $N>R$, assertion (iii) also gives
\[
 \|V_N-V\|_{\cS^2(C(\overline B_R))}
 \le C_a(1+R^2)^{(p_0+a)/2}N^{-a}
 \longrightarrow0
 \qquad(N\uparrow\infty).
\]
This proves the desired convergence.}
\end{proof}
\begin{lem}\label{lem-J}
Let $(\cA1)$ hold, $\theta\in\cU$, and $t\in[0,T]$.
For every $\sF_t$-measurable $\xi$ satisfying the moment condition
$E{\varpi(\xi)}^{p_0}<\infty$, we have
\[
 J(t,\xi;\theta)=E_{\sF_t}\left[
 \int_t^T f(s,X_s^{t,\xi;\theta},\theta_s)\,ds
 +G(X_T^{t,\xi;\theta})\right],\qquad
 V(t,\xi)=\essinf_{\theta\in\cU}J(t,\xi;\theta).
\]
Under $(\cA1_{\mathrm b})$, these identities hold for every finite
$\sF_t$-measurable $\xi$, without a moment condition.
\end{lem}

\begin{proof}
The bounded-initial-state identities follow from
Lemma \ref{lem-poly-approximation}.
 We use the continuous versions constructed there and write
\[
 H^\theta(\zeta)=\int_t^T
 f(s,X_s^{t,\zeta;\theta},\theta_s)\,ds
 +G(X_T^{t,\zeta;\theta})
\]
for the full cost from an $\sF_t$-measurable state $\zeta$.
For $p_0>0$, polynomial growth and \eqref{eq-poly-state-moment} give
\[
 E|H^\theta(\xi)|
 \le C E\sup_{t\le s\le T}\varpi(X_s^{t,\xi;\theta})^{p_0}
 \le C' E\varpi(\xi)^{p_0}<\infty.
\]
For $p_0=0$, integrability follows directly from boundedness
of the costs.

Set $A_n=\{|\xi|\le n\}\in\sF_t$ and
$\xi_n=\xi\mathbf1_{A_n}$. Pathwise uniqueness gives
$H^\theta(\xi)=H^\theta(\xi_n)$ on $A_n$. Thus the bounded-state
identity and locality of conditional expectation imply
\begin{align*}
 \mathbf1_{A_n}J(t,\xi;\theta)
 &=\mathbf1_{A_n}J(t,\xi_n;\theta)
 =E_{\sF_t}[\mathbf1_{A_n}H^\theta(\xi_n)]
 =\mathbf1_{A_n}E_{\sF_t}H^\theta(\xi),\quad \text{a.s.}
\end{align*}
Similarly, locality of essential infima yields
\begin{align}
 \mathbf1_{A_n}V(t,\xi)
 &=\mathbf1_{A_n}V(t,\xi_n)
 =\mathbf1_{A_n}\essinf_{\theta\in\cU}J(t,\xi_n;\theta)
 =\mathbf1_{A_n}\essinf_{\theta\in\cU}J(t,\xi;\theta),\quad \text{a.s.}
\end{align}
  Since
$\lim_{n\to\infty}\bP(\Omega\setminus A_n)=0$, both asserted identities follow.
Under $(\cA1_{\mathrm b})$, the full costs are uniformly bounded,
so the same argument applies to all finite $\xi$ without a
moment condition.
\end{proof}

\begin{prop}\label{prop-value-func}
Let $(\cA1)$ hold.
\begin{enumerate}
\item[(i)]  Let $\tau\le T$ be a stopping time and let
$\xi$ be $\sF_\tau$-measurable with $E\varpi(\xi)^{p_0}<\infty$.
For every $\theta\in\cU$,
\[
 J(\tau,\xi;\theta)=E_{\sF_\tau}\left[
 \int_\tau^T f(s,X_s^{\tau,\xi;\theta},\theta_s)\,ds
 +G(X_T^{\tau,\xi;\theta})\right],\qquad
 V(\tau,\xi)=\essinf_{\theta\in\cU}J(\tau,\xi;\theta).
\]
For every deterministic $a>0$, there exists $\theta^a\in\cU$
such that
\[
 0\le J(\tau,\xi;\theta^a)-V(\tau,\xi)\le a
 \quad\text{a.s.}
\]
In particular, for $\tau=t$ deterministic and every $\eps>0$,
one may choose $\bar\theta\in\cU$ with
$E[J(t,\xi;\bar\theta)-V(t,\xi)]<\eps$.
\item[(ii)]  For every $(\tau,\xi)$ as in (i), every
fixed $\theta\in\cU$, and every stopping time
$\eta\in\mathcal T^\tau$,
\begin{equation}\label{eq-stopping-control}
 V(\tau,\xi)\le E_{\sF_\tau}\left[
 \int_\tau^\eta f(s,X_s^{\tau,\xi;\theta},\theta_s)\,ds
 +V(\eta,X_\eta^{\tau,\xi;\theta})\right]
 \le J(\tau,\xi;\theta).
\end{equation}
The stopping time $\eta$ may be chosen after fixing $\theta$;
no DPP identity for a family of control-dependent horizons is
asserted.
In particular, for every $(\bar\theta,x)\in\cU\times\bR^d$,
$J(t,X_t^{0,x;\bar\theta};\bar\theta)
-V(t,X_t^{0,x;\bar\theta})$ is a supermartingale. Equivalently,
for $0\le t\le\tilde t\le T$,
\begin{equation}\label{eq-vfunc-supM}
 V(t,X_t^{0,x;\bar\theta})\le E_{\sF_t}\left[
 V(\tilde t,X_{\tilde t}^{0,x;\bar\theta})
 +\int_t^{\tilde t}f(s,X_s^{0,x;\bar\theta},\bar\theta_s)\,ds
 \right],\quad \text{a.s}.
\end{equation}
\item[(iii)] For every $(\bar\theta,x)$,
$t\mapsto V(t,X_t^{0,x;\bar\theta})$ is continuous.
\item[(iv)] The fields $V$ and $J(\cdot,\cdot;\theta)$, for each
fixed $\theta$, have jointly continuous versions. Under the
additional bounded assumption $(\cA1_{\mathrm b})$, there is a
deterministic $L_V$, independent of $\theta$, such that
\[
 |V(t,x)-V(t,y)|+|J(t,x;\theta)-J(t,y;\theta)|
 \le L_V|x-y|.
\]
No local Lipschitz assertion for $V$ is made under $(\cA1)$ alone.
\item[(v)] These continuous versions satisfy
\eqref{eq-poly-value-bound}. In particular,
$V,J(\cdot,\cdot;\theta)\in\cS^\infty(C(\overline B_R))$
for every $R$. Under $(\cA1_{\mathrm b})$, the moment restriction
in {(i) and (ii)} can be omitted and, with $L$ a uniform bound on the data,
\[
 \sup_{(t,x)\in[0,T]\times\bR^d}
 \max\{|V(t,x)|,|J(t,x;\theta)|\}\le L(T+1)
 \quad\text{a.s.}
\]
\end{enumerate}
The exceptional null set for a fixed-control field may depend on
that control.
\end{prop}

\begin{proof}
 The deterministic-time and regularity assertions under
$(\cA1_{\mathrm b})$ follow from
\cite[Proposition 3.3 and Appendix B]{Qiu2018viscosity}, with the
extension to the present control space explained in the proof of
Lemma~\ref{lem-poly-approximation}. For general data, (iv), (v),
and then (iii) follow from Lemma \ref{lem-poly-approximation} and continuity of the state
paths.

For (i), Lemma~\ref{lem-poly-approximation} establishes the two
identities for bounded $\xi$. They extend by locality on
$\{|\xi|\le n\}\in\sF_\tau$, as in the proof of
Lemma~\ref{lem-J}. The moment assumption and
\eqref{eq-poly-state-moment} ensure integrability; for $p_0=0$
this follows directly from boundedness of the costs.
The conditional costs are downward directed by pasting two
controls after $\tau$ on the $\sF_\tau$-event where the
cost is smaller. 
Choose a sequence
$J(\tau,\xi;\theta^j)\downarrow V(\tau,\xi)$. Selecting the
first index $j$ for which
$J(\tau,\xi;\theta^j)\le V(\tau,\xi)+a$ gives a countable
$\sF_\tau$-measurable partition. On each event, use the selected
control after $\tau$, with a fixed control before $\tau$. 
This piecewise definition gives a
progressively measurable $U$-valued control $\theta^a$ satisfying
\[
 0\le J(\tau,\xi;\theta^a)-V(\tau,\xi)\le a
 \quad\text{a.s.}
\]
Taking $\tau=t$ and $a=\eps/2$ proves the expected
$\eps$-optimality assertion. Under $(\cA1_{\mathrm b})$, the
full costs are uniformly bounded, so the moment condition is
unnecessary.

For (ii), put $X=X^{\tau,\xi;\theta}$. The state moment
estimate ensures that (i) applies at $(\eta,X_\eta)$; under
bounded data, only finiteness of $X_\eta$ is needed.
The right inequality in \eqref{eq-stopping-control} follows from
the flow and tower properties and
$V(\eta,X_\eta)\le J(\eta,X_\eta;\theta)$.
For the left inequality, keep $\theta$ up to $\eta$ and then
use a conditionally $a$-optimal continuation from (i) at
$(\eta,X_\eta)$. The pasted control $\tilde\theta^a$ is
admissible, and the flow and tower properties give
\[
\begin{aligned}
 V(\tau,\xi)
 &\le J(\tau,\xi;\tilde\theta^a)
 \le E_{\sF_\tau}\left[
 \int_\tau^\eta f(s,X_s,\theta_s)\,ds+V(\eta,X_\eta)
 \right]+a,\quad\text{a.s.}
\end{aligned}
\]
The arbitrariness of $a$ implies the left inequality.

Finally, fix $(\bar\theta,x)$ and write
\[
 X_t=X_t^{0,x;\bar\theta},\qquad
 A_t=\int_0^t f(s,X_s,\bar\theta_s)\,ds.
\]
The growth bounds and Lemma~\ref{lem-SDE} imply that
$A_T+G(X_T)$ is integrable and that
\[
 E\sup_{0\le t\le T}
 \bigl(|J(t,X_t;\bar\theta)|+|V(t,X_t)|\bigr)<\infty.
\]
The flow property and Lemma~\ref{lem-J} yield
\[
 J(t,X_t;\bar\theta)+A_t
 =E_{\sF_t}[A_T+G(X_T)],
\]
which is a martingale. Applying \eqref{eq-stopping-control}
with initial pair $(t,X_t)$ and horizon $\tilde t$ gives
\eqref{eq-vfunc-supM}; equivalently, $V(t,X_t)+A_t$ is a
submartingale. Their difference $J(t,X_t;\bar\theta)-V(t,X_t)$ is the claimed supermartingale,
and the martingale identity also yields the stated equivalence.
\end{proof}

\begin{thm}\label{thm-DPP}
Under $(\cA1)$, for stopping times $\tau\le\hat\tau\le T$ and
$\sF_\tau$-measurable $\xi$ with
$E{\varpi(\xi)}^{p_0}<\infty$,
\[
 V(\tau,\xi)=\essinf_{\theta\in\cU}E_{\sF_\tau}\left[
 \int_\tau^{\hat\tau}f(s,X_s^{\tau,\xi;\theta},\theta_s)\,ds
 +V(\hat\tau,X_{\hat\tau}^{\tau,\xi;\theta})\right]
 \quad\text{a.s.}
\]
Under $(\cA1_{\mathrm b})$, it suffices that $\xi$ be finite.
\end{thm}

\begin{proof}
 For each $\theta\in\cU$, set
\[
 Q^\theta=E_{\sF_\tau}\left[
 \int_\tau^{\hat\tau}f(s,X_s^{\tau,\xi;\theta},\theta_s)\,ds
 +V(\hat\tau,X_{\hat\tau}^{\tau,\xi;\theta})\right].
\]
Proposition~\ref{prop-value-func}(ii), with the prescribed horizon
$\eta=\hat\tau$, gives
\[
 V(\tau,\xi)\le Q^\theta\le J(\tau,\xi;\theta)
 \quad\text{a.s.}
\]
Taking essential infima and using
Proposition~\ref{prop-value-func}(i), we obtain
\[
 V(\tau,\xi)
 \le\essinf_{\theta\in\cU}Q^\theta
 \le\essinf_{\theta\in\cU}J(\tau,\xi;\theta)
 =V(\tau,\xi),
\]
which proves the asserted equality.
\end{proof}

\section{Existence of viscosity solutions for stochastic HJB equations}

\subsection{Definition of viscosity solutions}\label{sec:def}

\begin{defn}\label{defn-testfunc}
 For $u\in\cS^2(C_{\mathrm{loc}}(\bR^d))$, we say
$u\in\mathscr C_{\sF}^2$ if there exist predictable fields
$(\mathfrak d_tu,\mathfrak d_\omega u)$ such that, for every $R>0$,
\[
 u\in\cL^2(C^2(B_R)),\qquad
 \mathfrak d_tu\in\cL^2(C(B_R)),\qquad
 \mathfrak d_\omega u\in\cL^2(C^1(B_R;\bR^m)),
\]
and, a.s.,
\begin{align*}
u(r,x)=u(T,x)-\int_r^T \mathfrak{d}_su(s,x)\,ds -\int_r^T\mathfrak{d}_{\omega}u(s,x)\,dW_s,\quad \forall\,(r,x)\in[0,T]\times\bR^d.
\end{align*}
\end{defn}

In contrast to the test spaces used in
\cite{Qiu2018viscosity,qiu2019uniqueness,QiuZhang2023games}, which
require the defining norms to be finite over the whole space
$\bR^d$, the present definition imposes only local finiteness on each
ball $B_R$. This choice reflects the local nature of test functions
in the classical viscosity theory for deterministic PDEs. It places no
boundedness condition at spatial infinity, while retaining the
adaptedness and local square-integrability required for stochastic
calculus. Moreover, imposing global finiteness of the defining norms
does not change the classes of viscosity subsolutions or
supersolutions in Definition~\ref{defn-viscosity}; see
Remark~\ref{rmk-viscosity-interpretation}.





Definition~\ref{defn-testfunc} imposes square-integrability
on every spatial ball
to justify the It\^o--Wentzell formula and the martingale estimates
used below. More general integrability conditions may be considered,
but are not pursued here.
For each fixed $x$, a field $u\in\mathscr C_{\sF}^2$ is a
square-integrable It\^o process. Uniqueness of the continuous
semimartingale decomposition determines
$(\mathfrak d_tu,\mathfrak d_\omega u)$ up to
$d\bP\otimes dt$-null sets. Definition~\ref{defn-testfunc} thus
defines both the test space and the two linear operators $\mathfrak{d}_t \text{ and } \mathfrak{d}_{\omega}$.
The functional calculus developed in 
\cite{cont2013functional,Leao-etal-2018} gives a corresponding interpretation of the
drift and diffusion coefficients for such It\^o processes.
For a deterministic field in this test class, $\mathfrak d_\omega u=0$
and $t\mapsto u(t,x)$ is absolutely continuous, with
$\mathfrak d_tu(t,x)=\partial_tu(t,x)$ for almost every $t$.
If the field is continuously differentiable in time, this is the
ordinary time derivative. Under Malliavin regularity sufficient for
a diagonal trace, the corresponding identity is
$\mathfrak d_\omega u(t,x)=D_tu(t,x)$ where $D_t$ is 
the Malliavin derivative; see \cite{Malliavin-1978}.
The pair $(\mathfrak d_t,\mathfrak d_\omega)$ should be distinguished
from the path derivatives $(\partial_t,\partial_\omega)$ used in
functional It\^o calculus; see \cite{buckdahn2015pathwise} and
\cite[Section 2.3]{ekren2016viscosity-1}.
If, for each fixed $x$, the field $u$ admits a sufficiently smooth
nonanticipative representation on Brownian path space, the functional
It\^o formula and uniqueness of the semimartingale decomposition give
\[
 \mathfrak d_tu
 =\partial_tu+\frac12\operatorname{tr}(\partial^2_{\omega\omega}u),
 \qquad
 \mathfrak d_\omega u=\partial_\omega u,
 \quad d\bP\otimes dt\text{-a.e.}
\]
For applications of this correspondence to stochastic partial
differential equations, see \cite[Section 6]{ekren2016viscosity-1}
and \cite{buckdahn2015pathwise}.

For an $[0,T]$-valued stopping time $t$, let $\mathcal T^t$ denote
the set of stopping times taking values in $[t,T]$, and set
$\mathcal T^t_+=\{\tau\in\mathcal T^t:\tau>t\text{ a.s.}\}$.
For $\tau\in\mathcal T^0$ and $\Omega_\tau\in\sF_\tau$, we write
$L^0(\Omega_\tau,\sF_\tau;\bR^d)$ for the space of
$\bR^d$-valued, $\sF_\tau$-measurable functions on $\Omega_\tau$.

We now introduce the notion of viscosity solutions. For each $(u,\tau)\in \cS^{2}(C_{loc}(\bR^d))\times \mathcal T^0$, $\Omega_{\tau}\in\sF_{\tau}$ with $\mathbb P(\Omega_{\tau})>0$ and $\xi\in L^0(\Omega_{\tau},\sF_{\tau};\bR^d)$, we define
\begin{align*}
\underline{\mathcal{G}}u(\tau,\xi;\Omega_{\tau}):=\bigg\{
\phi\in\mathscr C^2_{\sF}:\ &
(\phi-u)(\tau,\xi)=0\quad\text{a.s. on }\Omega_\tau,\\
&\exists\,(\delta,\hat\tau)\in(0,\infty)\times\mathcal T^\tau_+
\text{ such that}\\
& (\phi-u)(s,x)\ge0\quad
\forall\,(s,x)\in[\tau,\hat\tau]\times B_\delta(\xi),
\quad\text{a.s. on }\Omega_\tau
\bigg\},\\
\overline{\mathcal{G}}u(\tau,\xi;\Omega_{\tau}):=\bigg\{
\phi\in\mathscr C^2_{\sF}:\ &
(\phi-u)(\tau,\xi)=0\quad\text{a.s. on }\Omega_\tau,\\
&\exists\,(\delta,\hat\tau)\in(0,\infty)\times\mathcal T^\tau_+
\text{ such that}\\
& (\phi-u)(s,x)\le0\quad
\forall\,(s,x)\in[\tau,\hat\tau]\times B_\delta(\xi),
\quad\text{a.s. on }\Omega_\tau
\bigg\}.
\end{align*}
 Here,  
the radius $\delta$ is deterministic, 
both time endpoints are stopping
times, and the event $\Omega_\tau$ is $\sF_\tau$-measurable.
The contact center need only be finite; no moment assumption on
$\xi$ is imposed. When conditional expectations are used in the
proofs, we localize to $\Omega_\tau\cap\{|\xi|\le N\}$.
On these events the spatial extrema are integrable by the
$\cS^2$ assumptions. Obviously, if $\underline{\mathcal{G}}u(\tau,\xi;\Omega_{\tau})$ or $\overline{\mathcal{G}}u(\tau,\xi;\Omega_{\tau})$ is nonempty, we must have $0\leq\tau <T$ on $\Omega_{\tau}$.

Now it is at the stage to introduce the definition of viscosity solutions.

\begin{defn}\label{defn-viscosity}
For $\phi\in\mathscr C^2_{\sF}$, write
\[
R_\phi(s,x):=-\mathfrak d_s\phi(s,x)
-\bH(s,x,D\phi(s,x),D^2\phi(s,x),D\mathfrak d_\omega\phi(s,x)).
\]
We say $u\in \cS^2(C_{loc}(\bR^d))$ is a viscosity subsolution (resp. supersolution) of BSPDE \eqref{SHJB}, if $u(T,x)\leq (\text{ resp. }\geq) G(x)$ for all $x\in\bR^d$ a.s., and for any $\tau\in  \mathcal T^0$, $\Omega_{\tau}\in\sF_{\tau}$ with $\mathbb P(\Omega_{\tau})>0$ and $\xi\in L^0(\Omega_{\tau},\sF_{\tau};\bR^d)$ and any $\phi\in \underline{\cG}u(\tau,\xi;\Omega_{\tau})$ (resp. $\phi\in \overline{\cG}u(\tau,\xi;\Omega_{\tau})$), there holds
\begin{align}
&\operatorname*{ess\,liminf}_{(s,x)\rightarrow (\tau^+,\xi)}
R_\phi(s,x)\leq0, \text{ for almost all $\omega\in\Omega_{\tau}$}
\label{defn-vis-sub}
\\
\text{(resp., }&\operatorname*{ess\,limsup}_{(s,x)\rightarrow (\tau^+,\xi)}
R_\phi(s,x)\geq0, \text{ for almost all $\omega\in\Omega_{\tau}$).}
\label{defn-vis-sup}
\end{align}

The essential limits in \eqref{defn-vis-sub}--\eqref{defn-vis-sup}
are taken pathwise. More precisely, on $\{\tau<T\}$,
\begin{align*}
\operatorname*{ess\,liminf}_{(s,x)\to(\tau^+,\xi)}R_\phi(s,x)
&:=\lim_{r\downarrow0}
\essinf_{\tau<s<(\tau+r^2)\wedge T}
\inf_{x\in B_r(\xi)}R_\phi(s,x),\\
\operatorname*{ess\,limsup}_{(s,x)\to(\tau^+,\xi)}R_\phi(s,x)
&:=\lim_{r\downarrow0}
\esssup_{\tau<s<(\tau+r^2)\wedge T}
\sup_{x\in B_r(\xi)}R_\phi(s,x).
\end{align*}
Here the essential infimum and supremum are with respect to Lebesgue
measure in time for each fixed $\omega$; the spatial infimum and supremum
are taken before these time-essential extrema.

Equivalently, $u\in\cS^2(C_{loc}(\bR^d))$ is a viscosity subsolution
(resp. supersolution) if it satisfies the terminal inequality above and
the following exclusion property. For any $\tau$, $\Omega_\tau$ and $\xi$
as above and any $\phi\in\mathscr C^2_{\sF}$, suppose there exist
$\eps,\tilde\delta>0$, an event
$\Omega'_\tau\in\sF_\tau$ with
$\Omega'_\tau\subset\Omega_\tau$ and $\bP(\Omega'_\tau)>0$, and
a stopping time $\rho\in\cT^\tau$ with $\rho>\tau$ on $\Omega'_\tau$,
such that, for almost every $\omega\in\Omega'_\tau$,
\[
R_\phi(s,x)\geq\eps
\quad\text{(resp. }R_\phi(s,x)\leq-\eps\text{)}
\]
for Lebesgue-almost every $s\in(\tau(\omega),\rho(\omega))$ and every
$x\in B_{\tilde\delta}(\xi(\omega))$. Then
$\phi\notin\underline{\cG}u(\tau,\xi;\Omega_\tau)$
(resp. $\phi\notin\overline{\cG}u(\tau,\xi;\Omega_\tau)$).

The function $u$ is a viscosity solution of BSPDE \eqref{SHJB} if it is both a viscosity subsolution and a viscosity supersolution of \eqref{SHJB}.
\end{defn}

Now we give some remarks on the above definition of viscosity solutions.
First, we justify the equivalence of the two formulations.
 Under $(\cA1)$ and Definition \ref{defn-testfunc}, the residual $R_\phi$
has a predictable version that is continuous in $x$ a.s. Indeed, the infimum defining the Hamiltonian
{can be taken over a countable dense subset of the separable space $U$, 
by continuity in $v$. The local spatial Lipschitz bounds on $\beta,\sigma$, 
the control-uniform modulus \eqref{eq:A1-cost-modulus}, and the spatial continuity of the test-function
 derivatives give a common local spatial modulus for the expressions inside this infimum. Hence the residual is spatially continuous.}
Spatially continuous versions agree  all $x$ a.s. 
Fubini's theorem therefore shows that the pathwise
time-essential extrema are independent of the choice of such versions.

Both formulations impose the same terminal inequality. We prove the
equivalence for subsolutions. Fix $\tau $ and $\xi$,
extend $\xi$ by zero outside $\Omega_\tau$, and write
\[
L:=\operatorname*{ess\,liminf}_{(s,x)\to(\tau^+,\xi)}R_\phi(s,x).
\]
First suppose \eqref{defn-vis-sub} holds for every upper contact test.
If the strict residual inequality in the exclusion property holds on
$\Omega'_\tau$, then, for almost every $\omega\in\Omega'_\tau$, every
sufficiently small $r>0$ satisfies
$r<\tilde\delta$ and $r^2<\rho(\omega)-\tau(\omega)$.
Thus $L\geq\eps$ on $\Omega'_\tau$.
Since this event has positive probability and is contained in $\Omega_\tau$,
$\phi$ cannot belong to $\underline{\cG}u(\tau,\xi;\Omega_\tau)$.
This proves the exclusion property.

Conversely, suppose the exclusion property holds but there is
$\phi\in\underline{\cG}u(\tau,\xi;\Omega_\tau)$ such that
$\bP(\Omega_\tau\cap\{L>0\})>0$. For positive integers $n,k$, put
\begin{align*}
b_n(s)&:=\inf_{|z|<1/n}R_\phi(s,\xi+z),\qquad s>\tau,\\
I_t^{n,k}&:=\int_0^t
1_{\{s>\tau\}}1_{\{b_n(s)<1/k\}}\,ds,\\
\rho_{n,k}&:=T\wedge\inf\{t\in[0,T]:I_t^{n,k}>0\},
\end{align*}
where $\inf\emptyset=+\infty$ and the integrand is defined to be zero
for $s\leq\tau$. The spatial infimum can be taken over rational
$z$. Moreover, the process $\xi1_{\{s>\tau\}}$ is predictable, so the
integrand defining $I^{n,k}$ is predictable.
Consequently $I^{n,k}$ is continuous, adapted and nondecreasing.
The right-continuity of the filtration implies that
$\rho_{n,k}\in\cT^\tau$ and
\[
D_{n,k}:=\{\rho_{n,k}>\tau\}\in\sF_\tau.
\]
For almost every $\omega\in\Omega_\tau$, $I_t^{n,k}=0$ for $t<\rho_{n,k}$.
It follows that
\[
R_\phi(s,x)\geq1/k
\quad\text{for Lebesgue-almost every }s\in(\tau,\rho_{n,k})
\quad\text{and every }x\in B_{1/n}(\xi).
\]
The exceptional time set here is common to all spatial points.

For almost every $\omega\in\Omega_\tau\cap\{L>0\}$, we have $\tau<T$ and can choose $k$ with $1/k<L$.
For sufficiently large $n$, the defining shrinking-cylinder infimum
is greater than $1/k$. Therefore
$I_{(\tau+n^{-2})\wedge T}^{n,k}=0$ and
$\rho_{n,k}\geq(\tau+n^{-2})\wedge T>\tau$. Hence, up to a null set,
\[
\Omega_\tau\cap\{L>0\}\subset\bigcup_{n,k\geq1}(\Omega_\tau\cap D_{n,k}).
\]
Countability gives a pair $(n,k)$ for which
$\Omega'_\tau:=\Omega_\tau\cap D_{n,k}$ has positive probability.
This event belongs to $\sF_\tau$ and satisfies the exclusion antecedent
with $\eps=1/k$, $\tilde\delta=1/n$ and $\rho=\rho_{n,k}$.
It follows that $\phi\notin\underline{\cG}u(\tau,\xi;\Omega_\tau)$, a contradiction.
Thus \eqref{defn-vis-sub} holds.
Applying the same argument to $-R_\phi$ and lower contact tests proves
the supersolution equivalence.

Next is about the local versus global test functions. 
 \begin{rmk}\label{rmk-viscosity-interpretation}
Requiring the defining norms to be finite on $\bR^d$ gives
the same viscosity subsolutions and supersolutions in
Definition~\ref{defn-viscosity}. Indeed, every such global test belongs
to $\mathscr C_{\sF}^2$. Conversely, suppose that the viscosity
inequality holds for all global tests, and let
$\phi\in\mathscr C_{\sF}^2$ be an upper contact test at
$(\tau,\xi;\Omega_\tau)$, with contact radius $\delta$. Extend $\xi$ by zero
outside $\Omega_\tau$ and set $A_j=\Omega_\tau\cap\{|\xi|\le j\}$. Choose a
deterministic $\chi_j\in C_c^\infty(B_{j+2})$ equal to one on
$B_{j+1}$, and put $\phi_j=\chi_j\phi$. The local norm bounds for
$\phi$ give all the global norm bounds for $\phi_j$, with
\[
 \mathfrak d_s\phi_j=\chi_j\mathfrak d_s\phi,
 \qquad
 \mathfrak d_\omega\phi_j=\chi_j\mathfrak d_\omega\phi.
\]
On $A_j$, the field $\phi_j$ retains the contact on the same
stopped interval and the smaller ball
$B_{\min\{\delta,1/2\}}(\xi)$. All derivatives entering the
residual agree there; hence $R_{\phi_j}=R_\phi$ on $B_{j+1}$
outside one $d\bP\otimes ds$-null set. Their pathwise essential
limits at $(\tau,\xi)$ therefore agree a.s. on $A_j$.
Applying the global-test inequality on each $A_j$ of positive
probability and using $A_j\uparrow \Omega_\tau$ proves the inequality for
$\phi$. The same argument applies to lower contacts. Thus the
local spatial requirements enlarge the test class without changing
the viscosity-solution notion or the growth conditions imposed on
solutions in the comparison results.
\end{rmk}

Third, in contrast to definitions of contact functions in \cite{Qiu2018viscosity,qiu2019uniqueness,QiuZhang2023games}, 
we drop the conditional expectations therein. 
The following lemma relates the conditional and pathwise
contact criteria while keeping the residual inequalities fixed.

{ \begin{lem}
\label{lem-contact-martingale}
Let $u\in\cS^2(C_{\mathrm{loc}}(\bR^d))$,
 and suppose that $\phi\in\mathscr C^2_\sF$
satisfies $(\phi-u)(\tau,\xi)=0$ a.s. on $\Omega_\tau$.
Extend $\xi$ by zero outside $\Omega_\tau$ and put
$A_{j}=\Omega_\tau\cap\{|\xi|\le j\}$, for $j=1,2,3,\ldots$.
Suppose that, for some $\delta>0$ and
$\hat\tau\in\mathcal T^\tau_+$, every
$\bar\tau\in\mathcal T^\tau$ and $j\ge1$ satisfy
\begin{equation}\label{eq:conditional-contact-criterion}
 E_{\sF_\tau}\!\left[
 1_{A_j}\inf_{x\in B_\delta(\xi)}
 (\phi-u)(\bar\tau\wedge\hat\tau,x)\right]\ge0.
\end{equation}
Then, for every $j$ with $\bP(A_j)>0$, there is a spatially
constant square-integrable continuous martingale $N^j$, with
$N^j_\tau=0$, such that
\[
 \phi+N^j\in\underline\cG u(\tau,\xi;A_j),
 \qquad R_{\phi+N^j}=R_\phi.
\]
If the spatial infimum and inequality in
\eqref{eq:conditional-contact-criterion} are replaced by a
supremum and $\le0$, respectively, then
$\phi-N^j\in\overline\cG u(\tau,\xi;A_j)$ and
$R_{\phi-N^j}=R_\phi$.
Consequently, for \eqref{SHJB}, Definition~\ref{defn-viscosity}
is equivalent to imposing the same residual inequalities on
tests satisfying the conditional criterion above and equality
at $(\tau,\xi)$.
\end{lem}

\begin{proof}
Fix $j$ and $0<r<\delta$, and define
\[
 Y_s=
 \begin{cases}
 0,&s<\tau,\\
 1_{A_j}\displaystyle\max_{|z|\le r}
 (u-\phi)(s\wedge\hat\tau,\xi+z),&s\ge\tau.
 \end{cases}
\]
At $\bar\tau=\tau$, the random variable inside
\eqref{eq:conditional-contact-criterion} is
$\sF_\tau$-measurable. Together with equality at $\xi$, this
implies $\phi(\tau,x)\ge u(\tau,x)$ throughout
$B_\delta(\xi)$ on $A_j$ and hence $Y_\tau=0$.
Thus the zero extension before $\tau$ is continuous and adapted.
The maximum over the closed ball is continuous in time.
For any deterministic $R>j+r$, we have
\[
 \sup_s|Y_s|\le 1_{A_j}\left(
 \sup_s\|u(s,\cdot)\|_{C(B_R)}
 +\sup_s\|\phi(s,\cdot)\|_{C(B_R)}\right).
\]
The right-hand side is square integrable, so $Y\in\cS^2$.
Apply Lemma~\ref{lem-snell-envelope} with $a=0$.
Let $Z$ be the Snell envelope of $Y$, and write
\[
 Z_s=Z_0+M_s-K_s,\qquad
 M_s=\int_0^s\zeta_r\,dW_r,\qquad M_0=K_0=0.
\]
For every $\rho\in\mathcal T^\tau$, the inclusion
$\overline B_r(\xi)\subset B_\delta(\xi)$ gives
\[
 Y_\rho\le-1_{A_j}\inf_{x\in B_\delta(\xi)}
       (\phi-u)(\rho\wedge\hat\tau,x).
\]
Thus \eqref{eq:conditional-contact-criterion} yields
$E_{\sF_\tau}Y_\rho\le0$. The stopping-time characterization
of the Snell envelope therefore gives
\[
 Z_\tau=\esssup_{\rho\in\mathcal T^\tau}
 E_{\sF_\tau}Y_\rho=0,
\]
where immediate stopping gives the equality because
$Y_\tau=0$.
The process $r\mapsto1_{\{r>\tau\}}$ is left-continuous
and adapted, hence predictable. Using the decomposition above, set
\[
 N^j_s=\int_0^s1_{\{r>\tau\}}\zeta_r\,dW_r
      =M_s-M_{s\wedge\tau},\qquad 0\le s\le T.
\]
This is a global square-integrable continuous martingale with
$N^j_\tau=0$. Simultaneously for all $s\ge\tau$,
\[
 N^j_s=M_s-M_\tau
 =Z_s-Z_\tau+K_s-K_\tau
 =Z_s+K_s-K_\tau\ge Z_s\ge Y_s.
\]
On $A_j$, it follows that
\[
 (\phi+N^j-u)(\tau,\xi)=0,\qquad
 (\phi+N^j-u)(s,x)\ge0
 \quad\text{for }(s,x)\in[\tau,\hat\tau]\times B_r(\xi).
\]
The martingale bounds in \eqref{eq:contact-snell-bounds} and
Doob's inequality show that the scalar field $N^j$ belongs to
$\mathscr C^2_\sF$. Its drift is zero and its spatial derivatives
vanish. In particular,
\[
 \mathfrak d_s(\phi+N^j)=\mathfrak d_s\phi,\quad
 D(\phi+N^j)=D\phi,\quad D^2(\phi+N^j)=D^2\phi,\quad
 D\mathfrak d_\omega(\phi+N^j)=D\mathfrak d_\omega\phi.
\]
This proves the upper-contact assertion and the equality of
residuals  
for all $x$ a.s. For lower contact, use the maximum of $\phi-u$ in
the definition of $Y$ and subtract $N^j$.

Every pathwise contact satisfies the corresponding conditional
criterion. Conversely, applying the pathwise viscosity inequality
to the corrected fields and then taking the countable union
$A_j\uparrow A$ gives the viscosity inequality for each test
satisfying that criterion. This proves the equivalence.
It uses the fact that $\bH$ depends on $D\mathfrak d_\omega\phi$
but not on the spatially constant value of
$\mathfrak d_\omega\phi$ or of $\phi$ itself.
\end{proof}

Let $\underline\cG_{\mathrm{cond}}u$ and
$\overline\cG_{\mathrm{cond}}u$ denote the contact classes obtained
by replacing the pathwise touching inequalities by
\eqref{eq:conditional-contact-criterion} and its lower-contact
counterpart, respectively, as in \cite{Qiu2018viscosity,qiu2019uniqueness,QiuZhang2023games}. 
The localization to $A_j$ makes these
conditional expectations well defined for a finite random center.
Every pathwise contact is a conditional contact. Conversely,
Lemma~\ref{lem-contact-martingale} converts each conditional
contact, on each $A_j$ and a smaller spatial ball, into a pathwise
contact by adding or subtracting a spatially constant martingale.
For \eqref{SHJB}, this correction preserves the residual.
Consequently the two choices of contact classes give the same
viscosity subsolutions, supersolutions and solutions when the
residual inequalities \eqref{defn-vis-sub}--\eqref{defn-vis-sup}
are kept fixed. This is an equivalence of solution notions; it
does not identify the individual contact classes.

The definition in \cite{Qiu2018viscosity} uses conditional
expectations in both the contact condition and the residual
inequalities. The preceding lemma concerns the contact condition
with the residual rule fixed; it does not assert that
pathwise essential limits of residuals agree with essential limits
of their conditional expectations. In the present definition the
contact and residual conditions are both pathwise, while
conditional expectations remain available in the proofs.

In addition, in the classical Markovian setting 
when all coefficients in problem \eqref{Control-probm} are deterministic, 
the stochastic HJB equation \eqref{SHJB} reduces to a deterministic parabolic PDE.
 Refined equivalent formulations of viscosity solutions for such equations are discussed in \cite{crandall2000lp,juutinen2001definition}. 
 After reversing time and restricting to deterministic test functions, 
 one can verify that our definition agrees with the standard notion for continuous viscosity 
 solutions, although it permits a broader class of test functions.

Finally, if the coefficients $\beta,\sigma,f$ and the terminal cost $G$ are 
deterministic functionals of time, the control, 
and the paths of $X$ and $W$, the problem is generally 
non-Markovian in the classical finite-dimensional sense. 
It can instead be viewed as a control problem on path space, 
with the histories of $X$ and $W$ as state variables. 
In this formulation, the value function is deterministic and 
may be characterized by a path-dependent PDE on an 
infinite-dimensional space. We refer to 
\cite{ekren2014viscosity,ekren2016viscosity-1,peng2011note,tang2012path} for the theory 
of viscosity solutions for such nonlinear path-dependent PDEs.
The path-dependent viscosity framework developed in 
\cite{ekren2014viscosity,ekren2016viscosity-1} has also been applied 
to certain stochastic HJB equations, including equations of the form \eqref{SHJB}. 
Those applications require the coefficients to be continuous in $\omega\in\Omega$, 
so that the test functions and the resulting viscosity solution 
can be treated pointwise. In contrast, the coefficients considered 
here are only measurable with respect to $\omega$, and no topology is imposed 
on $\Omega$. This motivates our use of touching and differential 
inequalities that hold almost surely on each nontrivial 
$\sF_\tau$-measurable event $\Omega_\tau$. In particular, 
the contact classes $\underline{\cG}u(\tau,\xi;\Omega_\tau)$ 
and $\overline{\cG}u(\tau,\xi;\Omega_\tau)$ are defined without 
conditional expectations.

\subsection{Some properties of viscosity solutions}
 
This subsection contains two basic properties of viscosity solutions. 
First, we introduce the notion of classical solutions and show that every classical solution is a viscosity solution.

\begin{defn} \label{defn-classical}
 Assume $(\cA1)$. A field $u\in\mathscr C_{\sF}^2$
is a classical solution of \eqref{SHJB} if $u(T,x)=G(x)$ for all $x\in\bR^d$ a.s., and, 
 there holds $d\bP\otimes dt$-a.e.  
\begin{equation}\label{eq:classical-equation}
 \mathfrak d_tu(t,x)
 +\bH(t,x,Du(t,x),D^2u(t,x),D\mathfrak d_\omega u(t,x))=0,\quad \forall x\in\bR^d.
\end{equation}
 \end{defn}

 It is typical that the classical solutions $u$ of backward SPDEs (see \cite{DuQiuTang10,Tang-Wei-2013} for instance) may not be  
 differentiable in the time variable $t$ and $(\mathfrak{d}_tu,\mathfrak{d}_{\omega}u)$ 
 may not be time-continuous but just measurable in $t$.
 This occurs already for the spatially independent BSDE
\[
 -dY_t=f_t\,dt-Z_t\,dW_t,\qquad Y_T=\eta,
\]
where $f\in\cL^2(\bR)$ and $\eta\in L^2(\Omega,\sF_T;\bR)$.
Its solution satisfies $\mathfrak d_tY=-f$ and
$\mathfrak d_\omega Y=Z$, $d\bP\otimes dt$-a.e., and these
coefficients need not admit continuous versions.
This possible lack of time continuity is reflected both in
Definitions~\ref{defn-testfunc} and~\ref{defn-viscosity}, which use
essential limits in \eqref{defn-vis-sub} and~\eqref{defn-vis-sup},
and in the definition of classical solutions above, which does not
require $\mathfrak{d}_tu,\mathfrak{d}_{\omega}u$ to be continuous in time.


\begin{prop}
\label{prop-classical-viscosity}
Under $(\cA1)$, every classical solution in the sense of
Definition~\ref{defn-classical} is a viscosity solution in the sense
of Definition~\ref{defn-viscosity}.
\end{prop}
As the proof uses some techniques from subsequent sections, we postpone it to the appendix.
The second property is a stability result for viscosity solutions.
 
\begin{thm}[Stability under local $\cS^2$ convergence]
\label{thm-viscosity-stability}
Let $(\beta_n,\sigma_n,f_n,G_n)$, $n\ge1$, and
$(\beta,\sigma,f,G)$ satisfy $(\cA1)$ on the same Brownian
probability space and with the same control space $U$. Suppose
that their linear-growth and polynomial-growth bounds hold with
common deterministic constants $K,C,p_0$. For every $R<\infty$,
define
\begin{align}
 d_{n,R}(s)&=\sup_{\substack{|x|\le R\\v\in U}}
 \bigl(|\beta_n-\beta|+{|\sigma_n-\sigma|}
                         +|f_n-f|\bigr)(s,x,v),\notag\\*
 g_{n,R}&=\sup_{|x|\le R}|G_n(x)-G(x)|,
 \label{eq:stability-data-errors}
\end{align}
and assume
\begin{equation}\label{eq:stability-data-convergence}
 \|d_{n,R}\|_{\cL^2(\bR)}+\|g_{n,R}\|_{L^2(\Omega)}
       \longrightarrow0\qquad\text{for every }R<\infty.
\end{equation}
Write $\bH_n$ for the corresponding Hamiltonians. If $u_n$ are
viscosity subsolutions (resp. supersolutions) for $(\bH_n,G_n)$,
$u\in\cS^2(C_{\mathrm{loc}}(\bR^d))$, and
\begin{equation}\label{eq:stability-field-convergence}
 \|u_n-u\|_{\cS^2(C(\overline B_R))}\longrightarrow0
       \qquad\text{for every }R<\infty,
\end{equation}
then $u$ is a viscosity subsolution (resp. supersolution) for
$(\bH,G)$. No polynomial-growth bound on $u_n$ or $u$ is
required. The local spatial Lipschitz constants in $(\cA1)$
may depend on $n$.
\end{thm}

\begin{proof}
For $F\in\mathscr C^2_\sF$ and a deterministic $R<\infty$, put
\[
 a_n(s)=\sup_{|x|\le R}\left|
 (\bH_n-\bH)(s,x,DF,D^2F,D\mathfrak d_\omega F)\right|.
\]
The spatial derivatives in this expression are evaluated at
$(s,x)$. Spatial continuity makes $a_n$ predictable by taking
the supremum over a countable dense subset. The common local
coefficient bounds and the test-function norms give
\[
 a_n\longrightarrow0\quad\text{in }\cL^2(\bR).
\]
Indeed,   
\[
\begin{aligned}
 a_n(s)\le C_Rd_{n,R}(s)\bigl(&1+\|DF(s)\|_{C(\overline B_R)}
                  +\|D^2F(s)\|_{C(\overline B_R)}
                  +\|D\mathfrak d_\omega F(s)\|_{C(\overline B_R)}\bigr),
\end{aligned}
\]
where $C_R$ is deterministic and independent of $n$.
By \eqref{eq:stability-data-convergence} and the finite $\cL^2(\bR)$ norms of $DF,D^2F$ and
$D\mathfrak d_\omega F$, we have $a_n\to0$ in $\cL^2(\bR)$.

Define
\begin{equation}\label{eq:stability-residual-correction}
 q_n(s)=E_{\sF_s}\int_s^T a_n(r)\,dr.
\end{equation}
Martingale representation and Doob's inequality give
\[
 dq_n=-a_n\,ds+z_n\,dW_s,\qquad q_n\ge0,\qquad
 \|q_n\|_{\cS^2(\bR)}\le2\sqrt T\,\|a_n\|_{\cL^2(\bR)}\longrightarrow0.
\]
As a spatially constant field, $q_n$ belongs to
$\mathscr C^2_\sF$. With $R^{\bH_n}$ denoting the residual
formed with $\bH_n$, we have on $\overline B_R$
\begin{equation}\label{eq:stability-residual-order}
 R^{\bH_n}_{F+q_n}\ge R^{\bH}_F,
 \qquad R^{\bH_n}_{F-q_n}\le R^{\bH}_F.
\end{equation}
Here the scalar corrections have $D\mathfrak{d}_{\omega}q_n\equiv 0$.

The terminal inequalities pass to the limit by
\eqref{eq:stability-data-convergence} and
\eqref{eq:stability-field-convergence}. We prove the differential subsolution condition.
If it fails, the equivalent formulation in
Definition~\ref{defn-viscosity} gives an upper
contact test $\phi$, an event $A\in\sF_\tau$ with
$\bP(A)>0$, and a stopping time $\rho\ge\tau$, such that
$\rho>\tau$ on $A$ and, for some $r,\eps>0$,
\begin{align}
 (\phi-u)(s,x)&\ge0,
     &&(s,x)\in[\tau,\rho]\times\overline B_r(\xi),\notag\\
 (\phi-u)(\tau,\xi)&=0,\qquad
 R^{\bH}_\phi(s,x)\ge2\eps,
     &&s\in(\tau,\rho),\ x\in\overline B_r(\xi).
 \label{eq:stability-strict-cylinder}
\end{align}
These assertions hold a.s. on $A$, with the residual
inequality understood for all $x$ almost everywhere in time.
 The closed ball is chosen strictly inside both original
spatial neighborhoods. By reducing $A$ while preserving positive probability
 if necessary, we may
assume $|\xi|\le K_0$ and $\tau\le T-b$ there, for deterministic
$K_0<\infty$ and $b>0$. Extend $\xi$ by zero off $A$ and fix
$R>K_0+r$.

Choose a smooth compactly supported cutoff $\chi$ equal to one
on a neighborhood of $\overline B_R$, and set
\[
 P_s(x)=\chi(x)\left(|x|^2-2x\cdot E_{\sF_s}\xi
                                  +E_{\sF_s}|\xi|^2\right).
\]
The bounded variables $\xi$ and $|\xi|^2$ and their
Brownian martingale representations show that
$P\in\mathscr C^2_\sF$. 
For $s\ge\tau$ and $x\in\overline B_R$,
$P_s(x)=|x-\xi|^2$ and its drift and noise vanish. Consequently,
for a sufficiently small deterministic $\kappa>0$, the field
\[
 F(s,x)=\phi(s,x)+\kappa P_s(x)+\kappa(s-\tau)^+
\]
belongs to $\mathscr C^2_\sF$ and satisfies, on the cylinder in
\eqref{eq:stability-strict-cylinder},
\begin{equation}\label{eq:stability-contact-penalty}
 F-u\ge\kappa\bigl(|x-\xi|^2+s-\tau\bigr),
 \qquad R^{\bH}_F\ge\eps.
\end{equation}
In fact the absolute change in the residual is at most
\[
 \kappa\left[1+\sup_{v\in U}
 \bigl(2r|\beta(s,x,v)|+{|\sigma(s,x,v)|^2}\bigr)\right],
\]
which has a deterministic bound on this ball. Construct $q_n$
from this fixed $F$ by \eqref{eq:stability-residual-correction}
and put $F_n=F+q_n$. Then
\begin{equation}\label{eq:stability-strict-approx-residual}
 R^{\bH_n}_{F_n}\ge\eps
 \quad\text{on the same stopped cylinder.}
\end{equation}

For $s\ge\tau$, define
\[
 Y_s^n=\max_{|z|\le r}
       (u_n-F_n)(s\wedge\rho,\xi+z),
\]
and, for $s<\tau$, define $Y_s^n=E_{\sF_s}Y_\tau^n$.
This is a continuous adapted process in $\cS^2(\bR)$: boundedness of
$\xi$ places the spatial maximum in $\overline B_R$, and the
continuous martingale extension agrees with $Y_\tau^n$ at $\tau$.
Let $Z^n$ be its Snell envelope and let
\[
 \gamma_n=\inf\{s\ge\tau:Z_s^n=Y_s^n\}.
\]
Lemma~\ref{lem-snell-envelope}(iii) gives optimality,
$Z_\tau^n=E_{\sF_\tau}Y_{\gamma_n}^n$. Since the reward is
constant after $\rho$, $Z_\rho^n=Y_\rho^n$ and
$\tau\le\gamma_n\le\rho$. By Theorem~\ref{thm-MS} on the
compact offset ball, choose an $\sF_{\gamma_n}$-measurable
maximizer $\xi_n\in\overline B_r(\xi)$ at $\gamma_n$.
Set
\[
 e_n=\sup_{s,\,|x|\le R}|u_n-u|(s,x)+\sup_s|q_n(s)|.
\]
Then $Ee_n\to0$. On $A$, evaluation at $\xi$ gives
$Y_\tau^n\ge-e_n$, while \eqref{eq:stability-contact-penalty}
gives
\[
 Y_{\gamma_n}^n\le e_n-
       \kappa\bigl(\gamma_n-\tau+|\xi_n-\xi|^2\bigr).
\]
Optimality and $A\in\sF_\tau$ therefore imply
\begin{equation}\label{eq:stability-contact-convergence}
 E\left[1_A\bigl(\gamma_n-\tau+|\xi_n-\xi|^2\bigr)\right]
       \le\frac{2Ee_n}{\kappa}\longrightarrow0.
\end{equation}
This implies $\bP(A\cap\{\gamma_n=\rho\})\to0$. Indeed,
for every $a>0$ this probability is at most
\[
 \bP(A\cap\{\rho-\tau\le a\})
       +a^{-1}E[1_A(\gamma_n-\tau)],
\]
and we let $n\to\infty$ followed by $a\downarrow0$.
Put $T_0=T-b/2$, $\tau_n=\gamma_n\wedge T_0$ and
\[
 A_n=A\cap\{\gamma_n<T_0,\ \gamma_n<\rho,
                                      \ |\xi_n-\xi|<r/2\}.
\]
Since $A_n\in\sF_{\gamma_n}$ and
$A_n\subset\{\gamma_n<T_0\}$, we have
$A_n\in\sF_{\gamma_n\wedge T_0}=\sF_{\tau_n}$. Moreover,
$\bP(A_n)\to\bP(A)>0$, by
\eqref{eq:stability-contact-convergence} and $\tau\le T-b$ on $A$.
In particular, no restriction to a future-measurable event is
used in the conditional optimality argument.

Write the Snell decomposition as $Z^n=\mathcal M^n-K^n$,
with the initial value included in the square-integrable
martingale $\mathcal M^n$. Lemma~\ref{lem-snell-envelope}(ii)
gives $E(K_T^n)^2<\infty$. Define the scalar martingale
\[
 N_s^n=E_{\sF_s}[\mathcal M_T^n-K_{\tau_n}^n].
\]
For $s\ge\tau_n$,
\[
 N_s^n=Z_s^n+K_s^n-K_{\tau_n}^n\ge Y_s^n,
\]
and $N_{\tau_n}^n=Y_{\tau_n}^n$ on $A_n$. Hence $F_n+N^n$
is a pathwise upper contact test for $u_n$ at $(\tau_n,\xi_n;A_n)$
on $B_{r/2}(\xi_n)$. Here $\xi_n$ is extended by zero off $A_n$;
this extension is $\sF_{\tau_n}$-measurable. Its contact horizon
is $\rho$ on $A_n$ and $T$ elsewhere. Pasting at $\tau_n$
makes this a stopping time globally strictly greater than
$\tau_n$, since $\tau_n\le T_0<T$.
The correction $N^n$ belongs to $\mathscr C^2_\sF$ and has zero
drift and spatial derivatives. Thus
\eqref{eq:stability-strict-approx-residual} contradicts the
subsolution property of $u_n$.

For supersolutions, start from a lower contact test with
$R^{\bH}_\phi\le-2\eps$, take
$F=\phi-\kappa P-\kappa(s-\tau)^+$ and $F_n=F-q_n$, and use
the reward $\max_{|z|\le r}(F_n-u_n)(s\wedge\rho,\xi+z)$.
The same estimates locate a contact, and the corrected lower
test is $F_n-N^n$, with strictly negative residual. This gives
the supersolution condition and completes the proof.
\end{proof}

\subsection{Existence of the viscosity solution}
 
 We use the It\^o-Kunita-Wentzell-Krylov formula for compositions of
random fields and controlled state processes; see
\cite[pp.~118--119]{kunita1981some} and
\cite[Theorem~3.1]{Krylov_09}. The latter formulation applies to
the temporal integrability in Definition~\ref{defn-testfunc}.
Throughout this work, we define for any $\phi\in\mathscr C^2_{\sF}$ and $v\in U$,
\begin{align*}
\mathscr L^{v}\phi(t,x)=
 \mathfrak{d}_t \phi (t,x)+\text{tr}\left(\frac{1}{2}\sigma \sigma'(t,x,v) D^2\phi(t,x)
     	+\sigma(t,x,v) D\mathfrak{d}_{\omega}\phi(t,x)\right)       +D\phi(t,x)\beta(t,x,v).
\end{align*}
\begin{lem}\label{lem-ito-wentzell}
Let assumption $(\cA1)$ hold and suppose $u\in\mathscr C_{\sF}^2$.
For each stopping time $\tau\in\mathcal T^0$, initial state
$\xi\in L^0(\Omega,\sF_\tau;\bR^d)$ and control $\theta\in\cU$, put
$X=X^{\tau,\xi;\theta}$. Then, a.s., simultaneously for all
$t\in[\tau,T]$,
\begin{align*}
 u(t,X_t)-u(\tau,\xi)
 &=\int_\tau^t\mathscr L^{\theta_s}u(s,X_s)\,ds\\
 &\quad+\int_\tau^t\left[
 \mathfrak d_\omega u(s,X_s)
 +Du(s,X_s)\sigma(s,X_s,\theta_s)\right]dW_s.
\end{align*}
The stochastic integral, extended by zero before $\tau$, is a
continuous local martingale. For every $R>0$, it is square
integrable after stopping at
\[
 \rho_R=\inf\{s\in[\tau,T]:|X_s|\ge R\}\wedge T,
 \qquad \inf\emptyset=+\infty.
\]
In particular, its conditional expectation given $\sF_\tau$ is
zero at every stopping time $\eta$ with $\tau\le\eta\le\rho_R$.
\end{lem}

\begin{proof}
First suppose $\xi$ is bounded and fix $R>\|\xi\|_{L^\infty}+1$.
Using the radial projection $\pi_R$, set
\[
 \beta_R(s,x,v)=\beta(s,\pi_Rx,v),\qquad
 \sigma_R(s,x,v)=\sigma(s,\pi_Rx,v).
\]
These coefficients are bounded and globally Lipschitz in $x$,
uniformly in the control. Let $\chi_R$ be a deterministic smooth
cutoff equal to one on $B_{R+1}$ and supported in $B_{R+2}$,
and put $u_R=\chi_Ru$. The defining norms of $u_R$ are finite
over the whole space. For deterministic initial pairs,
\cite[Theorem~3.1]{Krylov_09} gives the asserted formula for
$u_R$ and the projected dynamics. Indeed, the composed
coefficients admit predictable representatives equal to them
$d\bP\otimes ds$-a.e.; their bounds and the test-function norms
give the spatial continuity and integrability required by that
theorem. Continuity of the composition makes the identity
holds for all time $t\in[\tau,T]$ a.s.

Choose finite-valued stopping times $\tau_j\downarrow\tau$ with
$0\le\tau_j-\tau\le2^{-j}T$, and uniformly bounded simple
$\sF_\tau$-measurable states $\xi_j\to\xi$ a.s.
Let $Y$ and $Y^j$ solve the projected SDE with control $\theta$
from $(\tau,\xi)$ and $(\tau_j,\xi_j)$, respectively.
On each event $\{\tau_j=r,\xi_j=x\}\in\sF_r$, apply the
deterministic-start formula from $(r,x)$, with endpoint
$t\vee\tau_j$. Pasting these finitely many identities gives
the formula for $(\tau_j,\xi_j)$.
The bounded-coefficient increment estimate and global SDE
stability imply
\[
 E\sup_{\tau_j\le s\le T}|Y_s^j-Y_s|^2
 \le C_R\left(E|\xi_j-\xi|^2+2^{-j}T\right)
 \longrightarrow0,\quad \text{as }j\to\infty.
\]
The spatial continuity of $u_R,Du_R,D^2u_R,\mathfrak{d}_\omega u_R, \mathfrak{d}_t u_R, D\mathfrak{d}_\omega u_R$, together with their global $L^2$ supremum bounds,
therefore gives convergence of the drift integrands in
$L^1(\Omega\times[0,T])$ and of the stochastic integrands in
$L^2(\Omega\times[0,T])$, after extending the integrands by zero
before their respective initial times. 
Passing to the limit first for rational $t$ and then using
continuity proves the identity for $Y$ for all
$t\in[\tau,T]$ a.s. Pathwise uniqueness gives $Y=X$ up to
$\rho_R$, where the projected coefficients and $u_R$ agree
with the original ones. We thus obtain the asserted stopped
identity. Nonexplosion allows $R\to\infty$.

For a merely finite $\xi$, put
$A_N=\{|\xi|\le N\}\in\sF_\tau$ and
$\xi_N=\xi1_{A_N}$. Apply the bounded-state result to $\xi_N$.
On $A_N$, pathwise uniqueness and locality of stochastic
integration identify its formula with the one for $\xi$.
Since $A_N\uparrow\Omega$ a.s., the desired identity follows.
Finally, for every $R>0$,
\begin{align*}
 &E\int_\tau^{\rho_R}
 \left|\mathfrak d_\omega u(s,X_s)
       +Du(s,X_s)\sigma(s,X_s,\theta_s)\right|^2ds
       \\
 &\qquad
 \le C_R E\int_0^T\left(
 \|\mathfrak d_\omega u(s)\|_{C(B_{R+1})}^2
 +\|Du(s)\|_{C(B_{R+1})}^2\right)ds<\infty.
\end{align*}
Here $\rho_R=\tau$ if $|\xi|\ge R$. This proves the stopped
square-integrability and the conditional-mean assertion; letting
$R\to\infty$ also proves the local-martingale assertion.
\end{proof}

\begin{thm}\label{thm-existence}
 Let $(\cA1)$ hold. The value function $V$ defined by \eqref{eq-value-func} is a viscosity solution of the stochastic Hamilton-Jacobi-Bellman equation \eqref{SHJB}.
\end{thm}

{
\begin{proof}
\textbf{Step 1.}
First, in view of Proposition \ref{prop-value-func}, we have
$V\in\cS^\infty(C_{loc}(\bR^d))$ and $V(T,\cdot)=G$.
We prove that $V$ is a viscosity subsolution. For each
$\phi\in\underline\cG V(\tau,\xi;\Omega_\tau)$, with
$\tau\in\mathcal T^0$, $\Omega_\tau\in\sF_\tau$,
$\bP(\Omega_\tau)>0$, and
$\xi\in L^0(\Omega_\tau,\sF_\tau;\bR^d)$, let
$(\hat\tau,B_\delta(\xi))$ be the corresponding contact neighborhood.
Write
\[
 Q_\phi(s,x,v):=\mathscr L^v\phi(s,x)+f(s,x,v),\qquad
 R_\phi(s,x)=-\inf_{v\in U}Q_\phi(s,x,v).
\]

Suppose to the contrary that the subsolution inequality fails.
By the equivalent formulation in Definition \ref{defn-viscosity},
there exist $\eps,\tilde\delta>0$, an event
$\Omega'\in\sF_\tau$ with $\Omega'\subset\Omega_\tau$ and
$\bP(\Omega')>0$, and $\rho\in\mathcal T^\tau$ with
$\rho>\tau$ on $\Omega'$, such that
\begin{equation}\label{eq-existence-sub-residual}
 R_\phi(s,x)\geq2\eps,
 \qquad s\in(\tau,\rho),\quad x\in B_{\tilde\delta}(\xi),
 \quad\text{a.e. in time, a.s. on }\Omega'.
\end{equation}
Without loss of generality, assume $\tilde\delta<\delta\wedge1$.
Since $\xi$ is finite and $\tau<T$ on $\Omega_\tau$, we may reduce
$\Omega'$ while preserving positive probability so that
$|\xi|\leq N$ and $\tau\leq T-b$ there, for deterministic
$N<\infty$ and $b\in(0,T)$. Extend $\xi$ by zero outside $\Omega'$;
all contact inequalities below are used on $\Omega'$.

For each $\theta\in\cU$, put $X^\theta=X^{\tau,\xi;\theta}$ and
define
\[
 \zeta^\theta=\inf\{s\in[\tau,T]:
                  |X_s^\theta-\xi|\geq\tilde\delta/2\},\qquad
 \eta_h^\theta=(\tau+h)\wedge T\wedge\hat\tau\wedge\rho
                       \wedge\zeta^\theta,
\]
where $\inf\emptyset=+\infty$. By the fourth-moment increment
estimate in Lemma \ref{lem-SDE}, for $0<h\leq1$,
\begin{align}
 E_{\sF_\tau}1_{\{\zeta^\theta\leq(\tau+h)\wedge T\}}
 &\leq\frac{16}{\tilde\delta^4}
 E_{\sF_\tau}\sup_{\tau\leq s\leq(\tau+h)\wedge T}
                   |X_s^\theta-\xi|^4
 \leq\frac{Ch^2}{\tilde\delta^4}.
 \label{eq-existence-exit}
\end{align}
Here and below the constants may depend on $N$, but not on the
control. Since $\hat\tau\wedge\rho>\tau$ on $\Omega'$, choose
and fix $0<h<b\wedge1$ such that
\[
 \bP\bigl(\Omega'\cap\{\hat\tau\wedge\rho\leq\tau+h\}\bigr)
 \leq\tfrac14\bP(\Omega'),\qquad
 \frac{Ch^2}{\tilde\delta^4}\leq\tfrac14.
\]
Then it follows that
\begin{equation}\label{eq-existence-duration}
 \inf_{\theta\in\cU}
 E\bigl[1_{\Omega'}(\eta_h^\theta-\tau)\bigr]
 \geq\tfrac h2\bP(\Omega')=:m>0.
\end{equation}

We will apply Lemma \ref{lem-ito-wentzell} from the random initial
pair $(\tau,\xi)$ to $\eta_h^\theta$. The stopped state stays in
$B_{N+\tilde\delta+1}$, so the stochastic integral is square
integrable and has conditional mean zero.

The endpoint $X^\theta_{\eta_h^\theta}$ belongs to
$\overline B_{\tilde\delta/2}(\xi)\subset B_\delta(\xi)$.
Thus pathwise contact gives
\[
 (\phi-V)(\eta_h^\theta,X^\theta_{\eta_h^\theta})\ge0
 \quad\text{on }\Omega'.
\]
Taking the localized conditional expectation and using
\eqref{eq-stopping-control}, by Lemma \ref{lem-ito-wentzell}, we obtain,
a.s. on $\Omega'$,
\begin{align}
 0
 &\leq E_{\sF_\tau}
       (\phi-V)(\eta_h^\theta,X^\theta_{\eta_h^\theta})\nonumber\\
 &\leq E_{\sF_\tau}\left[
       \phi(\eta_h^\theta,X^\theta_{\eta_h^\theta})-\phi(\tau,\xi)
       +\int_\tau^{\eta_h^\theta}f(s,X_s^\theta,\theta_s)\,ds\right]
       \nonumber\\
 &=E_{\sF_\tau}\int_\tau^{\eta_h^\theta}
                   Q_\phi(s,X_s^\theta,\theta_s)\,ds.
 \label{eq-existence-upper-contact}
\end{align}

To contradict this inequality, fix a countable dense subset
$\{v_j:j\geq1\}$ of $U$. Outside a common null set,
$Q_\phi$ is continuous in $(x,v)$; let $a(s,x)$ be the first
$v_j$ satisfying
\[
 Q_\phi(s,x,v_j)
 \leq\inf_{i\geq1}Q_\phi(s,x,v_i)+\eps.
\]
The countable infimum equals the infimum over $U$, and this
first-index choice defines a
$\sP\otimes\cB(\bR^d)$-measurable map; set $a=v_1$ on the
exceptional set. For $n\geq1$, let
$t_k=(\tau+kh/n)\wedge T$, $0\leq k\leq n$, and construct
$\theta^n$ and $X^n=X^{\tau,\xi;\theta^n}$ successively by
\[
 \theta_s^n=a(s,X^n_{t_k}),\qquad t_k<s\leq t_{k+1},
 \quad 0\leq k<n.
\]
Use $v_1$ for $s\leq\tau$ and for $s>t_n$. At each step the frozen
state $X^n_{t_k}$ is $\sF_{t_k}$-measurable, so the control is
progressively measurable. For this prescribed process, $(\cA1)$
gives a unique nonexplosive strong solution on the next interval.
This proves the recursive construction without imposing any
regularity on the selector as a function of the state.

Put $Y_s^n=X^n_{t_k}$ for $t_k<s\leq t_{k+1}$, and set
\[
 \Delta_n=\max_{0\leq k<n}
       \sup_{t_k\leq s\leq t_{k+1}}|X_s^n-X^n_{t_k}|.
\]
The conditional increment estimate and the uniform fourth-moment
bound of Lemma \ref{lem-SDE} yield
\[
 E\Delta_n^4
 \leq C(h/n)^2\sum_{k=0}^{n-1}E[\varpi(X^n_{t_k})^4]
 \leq Ch^2/n\longrightarrow0.
\]
Before $\eta_h^{\theta^n}$, both $X_s^n$ and $Y_s^n$ lie in
$B_{\tilde\delta/2}(\xi)$. Hence \eqref{eq-existence-sub-residual}
and the choice of $a$ give
\[
 Q_\phi(s,Y_s^n,\theta_s^n)\leq-\eps
 \quad\text{for a.e. }s\in(\tau,\eta_h^{\theta^n}),
 \quad\text{a.s. on }\Omega'.
\]

We next estimate the error in replacing $Y_s^n$ by $X_s^n$.
On the fixed ball $B_{N+\tilde\delta+1}$, $(\cA1)$ and the
test-function norms give
\[
 \sup_{|x|\leq N+\tilde\delta+1,\,v\in U}|Q_\phi(s,x,v)|
 \leq K_s,\qquad E\int_0^T K_s^2\,ds<\infty,
\]
where one may take
\[
 K_s=C_N\bigl(1+\|\mathfrak d_s\phi(s)\|_{C(B_{N+\tilde\delta+1})}
 +\|D\phi(s)\|_{C(B_{N+\tilde\delta+1})}+\|D^2\phi(s)\|_{C(B_{N+\tilde\delta+1})}
 +\|D\mathfrak d_\omega\phi(s)\|_{C(B_{N+\tilde\delta+1})}\bigr).
\]
For $\gamma>0$, define
\[
 w(s,\gamma)=\sup\bigl\{
 |Q_\phi(s,x,v)-Q_\phi(s,y,v)|:
 x,y\in\overline B_{N+\tilde\delta},\ |x-y|\leq\gamma,\ v\in U
 \bigr\}.
\]
The control-uniform local Lipschitz bounds on $\beta,\sigma$,
condition \eqref{eq:A1-cost-modulus}, and uniform continuity of
the test-function derivatives on this ball imply
$w(s,\gamma)\to0$ as $\gamma\downarrow0$ for almost every
$(\omega,s)$. Moreover, $w(s,\gamma)\leq2K_s$. Therefore
\begin{align*}
 e_n&:=E\left[1_{\Omega'}\int_\tau^{\eta_h^{\theta^n}}
   |Q_\phi(s,X_s^n,\theta_s^n)-Q_\phi(s,Y_s^n,\theta_s^n)|\,ds\right]\\
 &\leq E\int_0^T w(s,\gamma)\,ds
       +2E\left[1_{\{\Delta_n>\gamma\}}\int_0^T K_s\,ds\right]
 \longrightarrow0,
\end{align*}
where we first let $n\to\infty$ and then $\gamma\downarrow0$,
using Cauchy--Schwarz and dominated convergence, respectively.
Multiplying \eqref{eq-existence-upper-contact} by $1_{\Omega'}$
and taking expectations, we obtain
\[
 0\leq E\left[1_{\Omega'}\int_\tau^{\eta_h^{\theta^n}}
                   Q_\phi(s,X_s^n,\theta_s^n)\,ds\right]
 \leq-\eps m+e_n,
\]
which is impossible when $n$ is large enough. Here $h$ remains fixed while the
mesh is refined. Thus,   a.s. on $\Omega_\tau$,
\[
 \operatorname*{ess\,liminf}_{(s,x)\to(\tau^+,\xi)}R_\phi(s,x)
 \leq0.
\]
Hence $V$ is a viscosity subsolution of \eqref{SHJB}.\medskip

\textbf{Step 2.}
It remains to prove that $V$ is a viscosity supersolution of
\eqref{SHJB}. Let $\phi\in\overline\cG V(\tau,\xi;\Omega_\tau)$
and let $(\hat\tau,B_\delta(\xi))$ be its contact neighborhood.
We again argue by contradiction. If the supersolution inequality
fails, the equivalent formulation gives $\eps,\tilde\delta>0$,
$\Omega'\in\sF_\tau$ with $\Omega'\subset\Omega_\tau$ and
$\bP(\Omega')>0$, and $\rho\in\mathcal T^\tau$ with
$\rho>\tau$ on $\Omega'$, such that
\[
 R_\phi(s,x)\leq-\eps,
 \qquad s\in(\tau,\rho),\quad x\in B_{\tilde\delta}(\xi),
 \quad\text{a.e. in time, a.s. on }\Omega'.
\]
Equivalently, $Q_\phi(s,x,v)\geq\eps$ for every $v\in U$ there.
Repeating the localization and stopping construction of
\textbf{Step 1}, with $\tilde\delta<\delta\wedge1$, gives a
bounded initial state and fixed $h,m>0$ satisfying
\eqref{eq-existence-duration} for the present contact.

Choose $\alpha>0$ so that $\alpha\bP(\Omega')<\eps m$.
By Proposition~\ref{prop-value-func}(i), there is a full-horizon
control $\theta^\alpha\in\cU$ such that
$J(\tau,\xi;\theta^\alpha)\leq V(\tau,\xi)+\alpha$.
Put $X=X^{\tau,\xi;\theta^\alpha}$ and
$\eta=\eta_h^{\theta^\alpha}$.
Pathwise contact gives
$(\phi-V)(\eta,X_\eta)\le0$ on $\Omega'$ and hence
$E_{\sF_\tau}(\phi-V)(\eta,X_\eta)\leq0$ there.
By Lemma \ref{lem-ito-wentzell} and the right inequality in
\eqref{eq-stopping-control},   a.s. on $\Omega'$,
\begin{equation*}
\begin{aligned}
 E_{\sF_\tau}\int_\tau^\eta Q_\phi(s,X_s,\theta_s^\alpha)\,ds
 &=E_{\sF_\tau}\left[
       \phi(\eta,X_\eta)-\phi(\tau,\xi)
       +\int_\tau^\eta f(s,X_s,\theta_s^\alpha)\,ds\right]\\
 &\leq E_{\sF_\tau}\left[
       V(\eta,X_\eta)+\int_\tau^\eta f(s,X_s,\theta_s^\alpha)\,ds
       \right]-V(\tau,\xi)\\
 &\leq J(\tau,\xi;\theta^\alpha)-V(\tau,\xi)\leq\alpha.
\end{aligned}
\end{equation*}
Taking expectations on $\Omega'$ and using the strict residual
inequality, we conclude that
\[
 \eps m
 \leq\eps E[1_{\Omega'}(\eta-\tau)]
 \leq E\left[1_{\Omega'}\int_\tau^\eta
                   Q_\phi(s,X_s,\theta_s^\alpha)\,ds\right]
 \leq\alpha\bP(\Omega')<\eps m,
\]
a contradiction. Thus,   a.s. on $\Omega_\tau$,
\[
 \operatorname*{ess\,limsup}_{(s,x)\to(\tau^+,\xi)}R_\phi(s,x)
 \geq0.
\]
Hence $V$ is a viscosity supersolution of \eqref{SHJB}, which
completes the proof.
\end{proof}
}


\section{Uniqueness of the viscosity solution}

{
This section is devoted to proving maximality and uniqueness under $(\cA1)$ in the
polynomial-growth class. We first construct auxiliary upper and
lower comparison fields for the bounded approximating problems of
Lemma~\ref{lem-poly-approximation}. A comparison argument on spatial
balls and an adapted polynomial barrier then give the two results
for the original data.
Recall the residual $R_\phi$ (from Definition~\ref{defn-viscosity})
\begin{equation}\label{eq:max-residual}
 R_\phi(t,x)=\sup_{v\in U}
 \{-\mathscr L^v\phi(t,x)-f(t,x,v)\}.
\end{equation}

\subsection{Auxiliary upper comparison fields}
First, we construct a family of regular upper comparison fields for the
bounded approximating problems of Lemma~\ref{lem-poly-approximation}.
\begin{lem}
\label{lem-upper-comparison-fields}
For data satisfying $(\cA1_{\mathrm b})$, fix a deterministic
$(t_0,x_0)\in[0,T)\times\bR^d$ and $\theta\in\cU$.
 There are deterministic constants $B_*$ and $C_{t_0}>0$,
independent of $0<\delta\le1$, and fields
$\Phi_\delta\in\mathscr C^2_\sF$ such that
\begin{align*}
 \mathscr L^{\theta_s}\Phi_\delta(s,x)+f(s,x,\theta_s)&\le0,
       &&\forall\, x\in\bR^d,\quad d\bP\otimes ds\text{-a.e.},\\
 \Phi_\delta(T,x)&\ge G(x),
       &&\forall\,x\in\bR^d,\quad\text{a.s.},\\
 \Phi_\delta(s,x)&\ge-B_*,
       &&\forall\,(s,x)\in[0,T]\times\bR^d,\quad\text{a.s.},\\
 \Phi_\delta(t_0,x_0)&=J(t_0,x_0;\theta)
                     + C_{t_0}\frac{\delta}{1+\delta}.
\end{align*}
\end{lem}

\begin{proof}
      \textbf{Step 1.} We first construct a regular upper field on the time interval $[t_0,T]$.
Enlarge the constant in $(\cA1_{\mathrm b})$ to a deterministic $L_*\ge1$
which bounds $|f|$, $|G|$, $|\beta|$, ${|\sigma|}$,
and all their spatial Lipschitz constants. 
 Fix deterministic $(t_0,x_0)\in[0,T)\times\bR^d$ and
$\theta\in\cU$. Let $X$ be the controlled state starting from
$X_{t_0}=x_0$, extended by $X_s=x_0$ for $s<t_0$.
Let $(Y,Z)$ solve the scalar BSDE on $[0,T]$
\begin{equation}\label{eq:max-scalar-bsde}
 Y_s=G(X_T)+\int_s^T\mathbf1_{[t_0,T]}(r)
 f(r,X_r,\theta_r)\,dr-\int_s^T Z_r\,dW_r.
\end{equation}
Martingale representation gives its solution, with
\begin{equation}\label{eq:max-Y-bound}
 |Y_s|\le B_*:=L_*(T+1),\qquad Z\in\cL^2(\bR^m),
 \qquad Y_{t_0}=J(t_0,x_0;\theta).
\end{equation}

For $0<\delta\le1$, put
\begin{equation}\label{eq:max-penalty}
 r_\delta(z)=\sqrt{\delta^2+|z|^2},\qquad
 p_\delta(z)=\frac{r_\delta(z)}{1+r_\delta(z)}.
\end{equation}
For each fixed $\delta$, $p_\delta$ is smooth with bounded derivatives.
Writing $r=r_\delta(z)$, direct differentiation gives
\begin{align*}
 Dp_\delta(z)&=\frac{z}{r(1+r)^2},\\
 D^2p_\delta(z)&=
 \frac1{(1+r)^2}\left(\frac{I}{r}-\frac{zz'}{r^3}\right)
 -\frac{2zz'}{r^2(1+r)^3}.
\end{align*}
Consequently,
\begin{equation}\label{eq:max-penalty-bounds}
 |z||Dp_\delta(z)|\le p_\delta(z),\qquad
 |z|^2\|D^2p_\delta(z)\|_{\mathrm{op}}\le3p_\delta(z).
\end{equation}
Also $\min\{|z|,2\}\le3p_\delta(z)$, so boundedness and
Lipschitz continuity give for all $x,y\in\bR^d$,
\begin{align}\label{eq:max-data-bounds}
 |f(s,x,\theta_s)-f(s,y,\theta_s)|&\le4L_*p_\delta(x-y),\notag\\
 |G(x)-G(y)|&\le4L_*p_\delta(x-y).
\end{align}

Choose
\begin{equation}\label{eq:max-constants}
 M\ge\max\{4L_*,B_*+1\},\qquad
 \lambda=1+L_*+\tfrac32L_*^2,\qquad
 c(s)=Me^{\lambda(T-s)}.
\end{equation}
To obtain a regular upper field on the entire time interval, choose
the finite deterministic constant
\begin{equation}\label{eq:max-extension-constant}
 A_\delta=L_*+\|c\|_\infty
 \left(L_*\|Dp_\delta\|_\infty+
       \tfrac12 L_*^2\|D^2p_\delta\|_{\mathrm{op},\infty}\right)
\end{equation}
and define
\begin{equation}\label{eq:max-upper-field}
 \Phi_\delta(s,x)=Y_s+c(s)p_\delta(x-X_s)
                       +A_\delta(t_0-s)^+,
\end{equation}
where the last term is zero on $[t_0,T]$ and
 $\|D^2p_\delta\|_{\mathrm{op},\infty}$ denotes
  the supremum of the operator norm of $D^2p_\delta$.

\textbf{Step 2.} We justify the construction.
First, we verify $\Phi_\delta\in\mathscr C^2_\sF$.
For $s>t_0$, abbreviate
$\beta_X=\beta(s,X_s,\theta_s)$,
$\sigma_X=\sigma(s,X_s,\theta_s)$, and $z=x-X_s$.
Writing $\sigma_X^k$ for the $k$th column of $\sigma_X$, It\^o's formula gives
\begin{align}\label{eq:max-field-derivatives}
 \mathfrak d_s\Phi_\delta
 &=-f(s,X_s,\theta_s)+c'p_\delta(z)
   -cDp_\delta(z)\cdot\beta_X
   +\tfrac c2\sigma_X\sigma_X':D^2p_\delta(z),\notag\\
 \mathfrak d_\omega^k\Phi_\delta
 &=Z_s^k-cDp_\delta(z)\cdot\sigma_X^k,\notag\\
 D(\mathfrak d_\omega^k\Phi_\delta)
 &=-cD^2p_\delta(z)\sigma_X^k.
\end{align}
For $s<t_0$, these quantities are respectively
$c'p_\delta(x-x_0)-A_\delta$, $Z_s^k$, and zero.
The value at the single deterministic time $t_0$ is immaterial.
Bounded coefficients and bounded derivatives of $p_\delta$ imply
\[
 \Phi_\delta\in\cS^2(C(\bR^d))\cap\cL^2(C^2(\bR^d)),\quad
 \mathfrak d_t\Phi_\delta\in\cL^2(C(\bR^d)),\quad
 \mathfrak d_\omega\Phi_\delta\in\cL^2(C^1(\bR^d)).
\]
We take predictable versions of the progressively
measurable processes $f(s,X_s,\theta_s)$, $\beta_X$, and $\sigma_X$,
obtained as $d\bP\otimes ds$-almost everywhere limits of continuous
adapted backward time averages. Substitution in
\eqref{eq:max-field-derivatives} gives predictable coefficient fields.
These replacements preserve the semimartingale decomposition and the
differential inequalities outside one $d\bP\otimes ds$-null set
common to all $x$.
It\^o's formula supplies the required semimartingale decomposition.

We next verify the differential inequality.
For $s>t_0$, put
$\beta_x=\beta(s,x,\theta_s)$ and
$\sigma_x=\sigma(s,x,\theta_s)$. Substitution of
\eqref{eq:max-field-derivatives} into $\mathscr L^{\theta_s}$ gives
\begin{align}\label{eq:max-cancellation}
 \mathscr L^{\theta_s}\Phi_\delta(s,x)+f(s,x,\theta_s)
 ={}&f(s,x,\theta_s)-f(s,X_s,\theta_s)+c'p_\delta(z)\notag\\
 &+cDp_\delta(z)\cdot(\beta_x-\beta_X)\notag\\
 &+\tfrac c2(\sigma_x-\sigma_X)(\sigma_x-\sigma_X) '
                      :D^2p_\delta(z),
\end{align}
where we have used the identity
\[
 \tfrac c2\sigma_X\sigma_X':D^2p_\delta
 +\tfrac c2\sigma_x\sigma_x':D^2p_\delta
 -c\sum_k(\sigma_x^k)'D^2p_\delta\,\sigma_X^k
 =\tfrac c2(\sigma_x-\sigma_X)(\sigma_x-\sigma_X)'
                    :D^2p_\delta.
\]
Using \eqref{eq:max-penalty-bounds}--\eqref{eq:max-constants} and the spatial
Lipschitz bounds in $(\cA1_{\mathrm b})$, we obtain
\begin{align}\label{eq:max-sign}
 \mathscr L^{\theta_s}\Phi_\delta+f(s,x,\theta_s)
 &\le\left[4L_*+c\left(-\lambda+L_*+\tfrac32L_*^2\right)\right]
                    p_\delta(z)\notag\\
 &=(4L_*-c)p_\delta(z)\le0.
\end{align}
For $s<t_0$, $c'\le0$ and \eqref{eq:max-extension-constant} give directly
\[
 \mathscr L^{\theta_s}\Phi_\delta+f(s,x,\theta_s)
 \le -A_\delta+L_*+
 \|c\|_\infty\left(L_*\|Dp_\delta\|_\infty+
 \tfrac12L_*^2\|D^2p_\delta\|_{\mathrm{op},\infty}\right)=0.
\]
It follows from \eqref{eq:max-residual} that $R_{\Phi_\delta}\ge0$.
These inequalities hold outside one $d\bP\otimes ds$-null set
simultaneously for all $x$, by spatial continuity and a countable
dense set of spatial points.

The terminal inequality follows from \eqref{eq:max-data-bounds}:
\[
 \Phi_\delta(T,x)=G(X_T)+Mp_\delta(x-X_T)\ge G(x).
\]
Both penalty terms in \eqref{eq:max-upper-field} are nonnegative,
so \eqref{eq:max-Y-bound} gives $\Phi_\delta\ge-B_*$.
At the starting point, $X_{t_0}=x_0$ and the extension term is zero;
hence
\[
 \Phi_\delta(t_0,x_0)=J(t_0,x_0;\theta)
                     +c(t_0)\frac{\delta}{1+\delta}.
\]
 Set $C_{t_0}=c(t_0)$. This constant is positive and independent of
$\delta$, as required.
\end{proof}

\subsection{Auxiliary lower comparison fields}
\label{sec:uniqueness-A1}

Throughout this subsection the data satisfy $(\cA1_{\mathrm b})$,
as do the bounded approximating problems in 
Lemma~\ref{lem-poly-approximation}. We enlarge a deterministic constant
$L$ to bound the coefficients norms and
all their spatial Lipschitz constants.

{For $q$ twice and $\psi=(\psi^k)_{k=1}^m$ once
continuously differentiable in the spatial variable, write}
\[
 \mathcal A^v(q,\psi)(t,x)
 =\tfrac12\sigma\sigma'(t,x,v):D^2q(t,x)
  +\sum_{k=1}^m\sigma^k(t,x,v)\cdot D\psi^k(t,x)
  +\beta(t,x,v)\cdot Dq(t,x).
\]
Here $\sigma^k$ is the $k$th column. The random field $q\in\mathscr C^2_\sF$ is called a regular lower
comparison field if it satisfies a.s. for all $x$ in the concerned domain and almost every $t$,
\[
 \mathfrak d_tq(t,x)+\mathcal A^v(q,\mathfrak d_\omega q)(t,x)
 +f(t,x,v)\ge0
 \quad\text{for every }v\in U.
\]

We construct regular lower comparison fields from spatially shaken value functions.
Krylov's method~\cite{krylov1999-policies} perturbs both time and
space in problems with deterministic coefficients.
Picarelli and Reisinger~\cite{picarelli-reisinger2020-error}
use spatial perturbations of Euler dynamics for error estimates in the approximation
of control problems.
For deterministic coefficients, Bouchard and Nutz
\cite{bouchard-nutz2016-target} construct smooth supersolutions for controlled forward--backward SDEs
using time-space shaking coefficients.
Related increasing-process and random-measure representations
of stochastic potentials appear in \cite{Qiu2014weak},
in a special setting with spatially independent controlled diffusion.

Here the shaken value is an adapted random field defined by
conditional expectations. Under $(\cA1_{\mathrm b})$,
Lemma~\ref{lem-uniq-A1-weak} gives its local weak semimartingale
decomposition with an $H^{-1}$ drift, an $L^2$ martingale
coefficient, and a positive random measure with a nonatomic
time marginal. 
After spatial convolution, we regularize the
resulting positive finite-variation term by adapted time
averages and correct the differential residual using a scalar
BSDE; see Lemma~\ref{lem-uniq-A1-regularize}.
This construction preserves adaptation and the one-sided
differential inequality required of lower comparison fields,
including when the finite-variation term has a time-singular
part.  It thereby supplies comparison fields for predictable,
possibly degenerate diffusion depending on both the state
and the control.
All Sobolev spaces used locally are on bounded balls; write
$H^j=H^{j,2}$ and let $H^{-1}$ be the dual of $H^{1,2}_0$. A distributional
identity means an identity after pairing with every compactly supported
smooth spatial test function.  

{
\begin{lem}[Parameterized Doob--Meyer decomposition]
\label{lem-uniq-A1-param-dm}
Let $O\subset\bR^d$ be a bounded Borel set with $|O|>0$.
Suppose $U_t(z)$ is a jointly measurable family of continuous
$(\sF_t)$-submartingales, bounded in absolute value by a common
deterministic constant $B$. There exist jointly measurable fields
$Z,K$ such that, for almost every $z\in O$, almost surely for all
$t\in[0,T]$,
\begin{equation}\label{eq:uniq-A1-param-dm}
 U_t(z)=U_0(z)+\int_0^t Z_s(z)\,dW_s+K_t(z),
\end{equation}
where $Z$ is $\sP\otimes\cB(O)$-measurable and square-integrable in
$(\omega,t,z)$, and $K_0=0$ with $K_\cdot(z)$ continuous and
increasing for almost every $(\omega,z)$. Moreover,
\begin{equation}\label{eq:uniq-A1-dm-bounds}
 E\int_O\left(K_T(z)^2+\int_0^T|Z_s(z)|^2\,ds\right)dz
 \le C_B|O|.
\end{equation}
where $C_B$ depends only on $B$.
When the construction
starts at a stopping time, the corresponding conditional estimate holds.
\end{lem}

\begin{proof}
{This follows from the classical Doob--Meyer theorem on
the product probability space
\[
 \widehat\Omega=\Omega\times O,\qquad
 \widehat{\bP}=\bP\otimes\frac{dz}{|O|},
\]
with the usual augmentation of $\sF_t\otimes\cB(O)$ as its
filtration $\widehat{\sF}_t$.
After choosing a product-adapted version,
$\widehat U_t(\omega,z)=U_t(\omega,z)$ is a bounded continuous
submartingale, hence of class D. The classical theorem gives
$\widehat U=\widehat U_0+M+K$, where $M_0=K_0=0$ and $K$ is
predictable and increasing; see~\cite[Chapter~III, Section~3]{pP05}.
Since $\widehat U$ is continuous, both $M$ and $K$ are continuous.
Put $C_0=B$. Localized It\^o's formula
for $(\widehat U+C_0)^2$, using $dK\ge0$ and Fatou's lemma, gives
\[
 \widehat E[M]_T\le4C_0^2.
\]
Thus $M$ is square-integrable on the product space, and
\[
 E\|M_T\|_{L^2(O)}^2
 =|O|\widehat E|M_T|^2\le4C_0^2|O|.
\]
Set $H=L^2(O)$, a separable Hilbert space. Product conditional
expectation identifies $M_t$ with $E[M_T\mid\sF_t]$ as an
$H$-valued martingale on the original Brownian filtration.
Applying the Hilbert-valued Brownian martingale representation
theorem~\cite[Theorem~4.7]{van-neerven-veraar-weis2007}, with state
space $L^2(O)$, noise space $\mathbb R^m$, and $p=2$, gives an
$L^2(O;\mathbb R^m)$-valued predictable process $Z$ satisfying
\[
 M_t=\int_0^t Z_s\,dW_s,\qquad
 E\int_0^T\|Z_s\|_{L^2(O;\mathbb R^m)}^2\,ds
 =E\|M_T\|_H^2.
\]
Choose a $\sP\otimes\cB(O)$-measurable representative of $Z$.
Stochastic Fubini gives \eqref{eq:uniq-A1-param-dm} for almost every
spatial parameter.}

{To make the conditional bounds explicit, fix stopping times
$\tau\le\gamma\le T$.}
For almost every $z$, apply It\^o's formula to $(U(z)+C_0)^2$ on
$[\tau,\gamma]$. Since
$0\le U+C_0\le2C_0$ and $dK\ge0$, taking conditional expectation   gives
\[
 E_{\sF_\tau}\int_\tau^\gamma|Z_s(z)|^2\,ds\le4C_0^2.
\]
Using
$K_\gamma(z)-K_\tau(z)=U_\gamma(z)-U_\tau(z)
-\int_\tau^\gamma Z_s(z)\,dW_s$ and conditional It\^o isometry yields
\[
 E_{\sF_\tau}|K_\gamma(z)-K_\tau(z)|^2\le16C_0^2.
\]
Integrating in $z$ proves
\eqref{eq:uniq-A1-dm-bounds} and its conditional version. In particular,
by Cauchy--Schwarz,
\[
 E_{\sF_\tau}\left[\left(\int_O
       (K_\gamma(z)-K_\tau(z))\,dz\right)^2\right]
 \le16C_0^2|O|^2.
\]
Thus $dK_t(z)\,dz$ is a positive random measure with square-integrable
total mass. The preceding conditional estimate holds for every pair
of stopping times $\tau\le\gamma$, with a deterministic bound
independent of $\tau$ and $\gamma$. Before spatial
smoothing, bounded variation is asserted in $L^1(O)$, not in $L^2(O)$.
\end{proof}
}

\begin{lem}[A weak semimartingale decomposition of a shaken value]
\label{lem-uniq-A1-weak}
For $0<\delta\le1$, let $\mathcal E_\delta$ be the set of all
$\overline B_\delta(0)$-valued progressively measurable processes.
For $(\theta,{\epsilon})\in\cU\times\mathcal E_\delta$, let
$X=X^{t,x;\theta,{\epsilon}}$ solve
\[
 \begin{cases}
 dX_s=\beta(s,X_s+{\epsilon}_s,\theta_s)\,ds
       +\sigma(s,X_s+{\epsilon}_s,\theta_s)\,dW_s,\qquad t\le s\le T,\\
 X_t=x.
 \end{cases}
\]
Define the shaken value function by
\[
 w^\delta(t,x)
 =\essinf_{(\theta,{\epsilon})\in\cU\times\mathcal E_\delta}
 E_{\sF_t}\!\left[\int_t^T f(s,X_s+{\epsilon}_s,\theta_s)\,ds+G(X_T)\right].
\]
{For every bounded ball $D$, there exist fields $a$ and
$\psi$ and a positive random measure $\mu$ on $[0,T]\times D$,
satisfying:
\begin{enumerate}
\item[(i)] There holds the following identity locally in $D$:
\begin{equation}\label{eq:uniq-A1-weak}
 dw^\delta=a\,dt+\psi\,dW_t+\mu(dt,dx),
\end{equation}
with $\psi\in \cL^2(L^2( D;\bR^m))$,
$a\in \cL^2( H^{-1}(D))$;
\item[(ii)] The time marginal of $\mu$ is nonatomic, and
\[
 E\big[\mu([0,T]\times D')^2\big]<\infty
 \quad\text{for every }D'\Subset D.
\]
\item[(iii)] For every
$\varphi\in C_c^\infty(D)$, the process
\[
 t\longmapsto\int_{(0,t]\times D}\varphi(x)\,\mu(ds,dx)
\]
is continuous and adapted.
\end{enumerate}}
\end{lem}

\begin{proof}
The enlarged control problem still satisfies $(\cA1_{\mathrm b})$. Its
dynamic programming principle and Proposition~\ref{prop-value-func} give a
jointly continuous value $w=w^\delta$ satisfying
\begin{equation}\label{eq:uniq-A1-w-bounds}
 |w|\le B:=L(T+1),\qquad
 |w(t,x)-w(t,y)|\le C|x-y|,
\end{equation}
with deterministic $C$ independent of $\delta$.

{\textbf{Step 1.}}
We now obtain a local decomposition by translations. Put $h=\delta/8$,
fix a center $x_i$ and $v_0\in U$, and start at a stopping time $\tau$.
Define, for $s\ge\tau$,
\[
 A_s=\int_\tau^s\beta(r,x_i,v_0)\,dr
       +\int_\tau^s\sigma(r,x_i,v_0)\,dW_r,
\]
stopped when $|A_s|=h$, or at $T$. For $z\in B_{4h}(x_i)$, the
state $z+A_s$ can use the adapted shaking control
${\epsilon}_s=x_i-z-A_s$. Before stopping, $|{\epsilon}_s|\le5h<\delta$ and
$(z+A_s)+{\epsilon}_s=x_i$.
{More precisely, define
\[
 \gamma:=\inf\{s\in[\tau,T]:|A_s|\ge h\}\wedge T,
 \qquad \inf\varnothing:=\infty.
\]
Since $A$ is continuous and adapted, $\gamma$ is a stopping time.
Extend}
${\epsilon}$ by zero outside $[\tau,\gamma]$. Then ${\epsilon}\in\mathcal E_\delta$,
and the corresponding controlled state agrees with $z+A$ on
$[\tau,\gamma]$.
Therefore the dynamic programming principle makes
\[
 U_s(z)=w(s,z+A_s)+\int_\tau^s f(r,x_i,v_0)\,dr
\]
a uniformly bounded continuous submartingale on this stopped interval.
{Apply Lemma~\ref{lem-uniq-A1-param-dm}, in particular
\eqref{eq:uniq-A1-param-dm}--\eqref{eq:uniq-A1-dm-bounds}.} Stochastic intervals
cause no measurability difficulty, since before $\tau$ one may extend $U$ by the
martingale $E[U_\tau\mid\sF_s]$, and stop at the right endpoint.

{To recover the weak semimartingale decomposition of $w$,
we apply It\^o's product rule to $U_s(z)\varphi(z+A_s)$, followed by
stochastic Fubini and spatial integration by parts.}
For a smooth test $\varphi$ supported in $B_{3h}(x_i)$, pair $U_s(z)$
with $\varphi(z+A_s)$ and integrate over $z$. Its support stays in
$B_{4h}(x_i)$. The product rule and stochastic Fubini give the
martingale coefficient
\[
 \psi(s,x)=Z_s(x-A_s)-Dw(s,x)\sigma(s,x_i,v_0)
\]
and the distributional identity on $B_{3h}(x_i)$
\begin{align}\label{eq:uniq-A1-translated-weak}
 dw={}&-\left[\tfrac12 a_i:D^2w+
              \sum_k\sigma_i^k\cdot D\psi^k+
              \beta_i\cdot Dw+f_i\right]\,{ds}
       +\psi\,dW_s+\mu_i,\\
 &a_i=\sigma_i\sigma_i',\quad
 (\beta_i,\sigma_i,f_i)(s)=(\beta,\sigma,f)(s,x_i,v_0),\notag\\
 &\int\varphi(x)\,\mu_i(ds,dx)
       =\int\varphi(z+A_s)\,dK_s(z)\,dz.\notag
\end{align}
{We give the calculation on $[\tau,\gamma]$ explicitly.
Put $O=B_{4h}(x_i)$, $w_s=w(s,\cdot)$,
\[
 F_s=\int_\tau^s f_i(r)\,dr,\qquad
 \widetilde Z_s(x)=Z_s(x-A_s),\qquad
 \varphi_s^A(z)=\varphi(z+A_s),
\]
and write $\langle\cdot,\cdot\rangle$ for spatial pairing. We suppress
the time argument of the frozen coefficients in the formulas below.
Since $|A_s|\le h$ and $\operatorname{supp}\varphi\Subset B_{3h}(x_i)$,
the support of $\varphi_s^A$ is compactly contained in $O$.
The change of variables $x=z+A_s$ therefore gives
\[
 I_s:=\int_O U_s(z)\varphi_s^A(z)\,dz
     =\langle w_s,\varphi\rangle+F_s\int\varphi(x)\,dx.
\]
For almost every $z$, the preceding Doob--Meyer decomposition and
It\^o's formula give
\begin{align*}
 dU_s(z)&=\sum_k Z_s^k(z)\,dW_s^k+dK_s(z),\\
 d\varphi_s^A(z)
 &=\left(\beta_i\cdot D\varphi+
                  \tfrac12a_i:D^2\varphi\right)(z+A_s)\,ds
   +\sum_k(\sigma_i^k\cdot D\varphi)(z+A_s)\,dW_s^k.
\end{align*}
In particular, their quadratic covariation is
\[
 d[U(z),\varphi^A(z)]_s
 =\sum_k Z_s^k(z)(\sigma_i^k\cdot D\varphi)(z+A_s)\,ds.
\]
Thus the full product formula reads
\begin{align*}
 d\bigl(U_s(z)\varphi_s^A(z)\bigr)
 ={}&\varphi_s^A(z)\,dK_s(z)\\
 &+\left[U_s(z)\left(\beta_i\cdot D\varphi+
                      \tfrac12a_i:D^2\varphi\right)(z+A_s)
       +\sum_k Z_s^k(z)(\sigma_i^k\cdot D\varphi)(z+A_s)\right]ds\\
 &+\sum_k\left[\varphi_s^A(z)Z_s^k(z)
       +U_s(z)(\sigma_i^k\cdot D\varphi)(z+A_s)\right]dW_s^k.
\end{align*}
For every multi-index $\alpha$ with $1\le|\alpha|\le2$, the terms
containing $F_s$ disappear after spatial integration, since
\begin{align*}
 \int_O U_s(z)D^\alpha\varphi(z+A_s)\,dz
 &=\langle w_s,D^\alpha\varphi\rangle
      +F_s\int D^\alpha\varphi(x)\,dx\\*
 &=\langle w_s,D^\alpha\varphi\rangle.
\end{align*}
Define the continuous finite-variation process
\[
 C_t^\varphi:=\int_O\int_\tau^t
          \varphi(z+A_s)\,dK_s(z)\,dz,
 \qquad \tau\le t\le\gamma.
\]
Integrating the product formula over $O$, using stochastic Fubini,
and subtracting $dF_s\int\varphi=f_i\int\varphi\,ds$ now gives
\begin{align*}
 d\langle w_s,\varphi\rangle
 ={}&\left[
       \left\langle w_s,\beta_i\cdot D\varphi+
                         \tfrac12a_i:D^2\varphi\right\rangle
       +\sum_k\left\langle\widetilde Z_s^k,
                         \sigma_i^k\cdot D\varphi\right\rangle
       -f_i\int\varphi\right]ds\\
 &+\sum_k\left[
       \langle\widetilde Z_s^k,\varphi\rangle
       +\langle w_s,\sigma_i^k\cdot D\varphi\rangle
       \right]dW_s^k+dC_s^\varphi.
\end{align*}
Since $\sigma_i^k$ is spatially constant, spatial integration by parts
identifies the $k$th martingale coefficient as
\[
 \langle\widetilde Z_s^k,\varphi\rangle
       +\langle w_s,\sigma_i^k\cdot D\varphi\rangle
 =\langle\widetilde Z_s^k-\sigma_i^k\cdot Dw_s,\varphi\rangle
 =\langle\psi_s^k,\varphi\rangle.
\]
For the drift, the identity $a_i=\sigma_i\sigma_i'$ gives
\begin{align*}
 \sum_k\langle\sigma_i^k\cdot Dw_s,\sigma_i^k\cdot D\varphi\rangle
 &=\sum_{\ell,j=1}^d a_i^{\ell j}
                   \langle D_\ell w_s,D_j\varphi\rangle\\
 &=-\sum_{\ell,j=1}^d a_i^{\ell j}
                   \langle w_s,D_{\ell j}\varphi\rangle
  =-\langle w_s,a_i:D^2\varphi\rangle.
\end{align*}
Writing $b_s(\varphi)$ for the $ds$-coefficient above and substituting
$\widetilde Z_s^k=\psi_s^k+\sigma_i^k\cdot Dw_s$, we obtain
\begin{align*}
 b_s(\varphi)
 &=\left\langle w_s,\beta_i\cdot D\varphi+
                         \tfrac12a_i:D^2\varphi\right\rangle
   +\sum_k\langle\psi_s^k,\sigma_i^k\cdot D\varphi\rangle
   -\langle w_s,a_i:D^2\varphi\rangle-f_i\int\varphi\\
 &=\left\langle w_s,\beta_i\cdot D\varphi-
                         \tfrac12a_i:D^2\varphi\right\rangle
   +\sum_k\langle\psi_s^k,\sigma_i^k\cdot D\varphi\rangle
   -f_i\int\varphi\\
 &=-\left\langle
       \tfrac12a_i:D^2w_s+\sum_k\sigma_i^k\cdot D\psi_s^k
          +\beta_i\cdot Dw_s+f_i,\varphi\right\rangle.
\end{align*}
This proves the drift and martingale terms in
\eqref{eq:uniq-A1-translated-weak}. All spatial derivatives in the last
identity are distributional: $Dw$ is bounded, $\widetilde Z$ and
$\psi$ are locally square-integrable, and the frozen coefficients
are spatially constant.

For completeness, the integrability needed for the preceding Fubini
steps can be checked directly. The stochastic integrands
\[
 \mathcal H_s^k(z)
 :=\varphi_s^A(z)Z_s^k(z)
     +U_s(z)(\sigma_i^k\cdot D\varphi)(z+A_s)
\]
satisfy, by boundedness of $U$ and $\sigma_i$,
\[
 E\int_\tau^\gamma\int_O\sum_k|\mathcal H_s^k(z)|^2\,dz\,ds
 \le C_\varphi\left(
       E\int_\tau^\gamma\int_O|Z_s(z)|^2\,dz\,ds+|O|T\right)<\infty.
\]
Together with $|O|<\infty$, this justifies stochastic Fubini.
The drift terms containing $Z$ are integrable by Cauchy--Schwarz,
and all other drift factors are bounded. For the finite-variation
term, positivity of $dK_s(z)$ gives
\[
 \int_O\int_\tau^\gamma
       |\varphi(z+A_s)|\,dK_s(z)\,dz
 \le\|\varphi\|_\infty
       \int_O\bigl(K_\gamma(z)-K_\tau(z)\bigr)\,dz<\infty,
\]
where the right-hand side has finite second moment by the conditional
bounds above. This also proves continuity of $C^\varphi$ by dominated
convergence and continuity of $K_\cdot(z)$ for almost every $z$.

More explicitly, the transported measure is defined, for Borel
$B\subset B_{3h}(x_i)$ and $\tau\le t\le\gamma$, by
\[
 \mu_i((\tau,t]\times B)
 :=\int_O\int_\tau^t\mathbf1_B(z+A_s)\,dK_s(z)\,dz\ge0.
\]
Consequently,
\[
 \int_{(\tau,t]\times B_{3h}(x_i)}\varphi(x)\,\mu_i(ds,dx)
 =C_t^\varphi
 =\int_O\int_\tau^t\varphi(z+A_s)\,dK_s(z)\,dz.
\]
Thus the cumulative measure uses the translations $A_s$ at the
integration times $s$, and the measure term in
\eqref{eq:uniq-A1-translated-weak} is positive.}

{\textbf{Step 2.}}
To obtain the stated norms on $D$, choose finitely many centers $x_1,\ldots,x_N$
whose $B_h(x_i)$ cover a neighborhood of $\overline D$, with a
nonnegative smooth partition of unity on that neighborhood supported
in the $B_{2h}(x_i)$. Run the translations at all
centers from the same stopping time and restart all of them when
the first translation moves distance $h$. More precisely, for a
stopping time $\tau\le T$, put
\[
 A_s^{\tau,i}=\int_\tau^s\beta(r,x_i,v_0)\,dr
       +\int_\tau^s\sigma(r,x_i,v_0)\,dW_r,
 \qquad \tau\le s\le T.
\]
The coefficient bounds and conditional Doob inequality give, for $b>0$,
\[
 E_{\sF_\tau}\!\left[
   \sup_{\tau\le s\le(\tau+b)\wedge T}|A_s^{\tau,i}|^2\right]
 \le 2L^2b^2+8L^2b.
\]
Indeed, the drift contributes at most $2L^2b^2$, while conditional
Doob inequality and It\^o isometry bound the martingale contribution
by $8L^2b$. Conditional Markov inequality and a union bound therefore give
\[
 \bP\!\left(
   \max_{1\le i\le N}\sup_{\tau\le s\le(\tau+b)\wedge T}
       |A_s^{\tau,i}|\ge h\,\middle|\,\sF_\tau\right)
 \le \frac{NL^2}{h^2}(2b^2+8b).
\]
Choose a deterministic $b>0$ so that the right-hand side is at most
$1/2$. This choice is uniform over all starting stopping times.
Set $\tau_0=0$ and, on $\{\tau_j<T\}$, define
\[
 \tau_{j+1}
 =\inf\left\{s\in(\tau_j,T]:
       \max_{1\le i\le N}|A_s^{\tau_j,i}|\ge h\right\}\wedge T,
\]
where $\inf\varnothing=\infty$, and set $\tau_{j+1}=T$ on
$\{\tau_j=T\}$. Writing $\Delta_j=\tau_{j+1}-\tau_j$, we have
\[
 \bP\!\left(\Delta_j\le b,\ \tau_{j+1}<T
       \,\middle|\,\sF_{\tau_j}\right)
 \le \tfrac12\mathbf1_{\{\tau_j<T\}}.
\]
For $\lambda>0$, spliting according to whether
$\Delta_j\le b$ or $\Delta_j>b$ yields
\begin{align*}
 &E_{\sF_{\tau_j}}\!\left[
       e^{-\lambda\Delta_j}\mathbf1_{\{\tau_{j+1}<T\}}\right]\\*
 &\quad\le e^{-\lambda b}\mathbf1_{\{\tau_j<T\}}
   +(1-e^{-\lambda b})
       \bP\!\left(\Delta_j\le b,\ \tau_{j+1}<T
                      \,\middle|\,\sF_{\tau_j}\right)\\*
 &\quad\le q\mathbf1_{\{\tau_j<T\}},
 \qquad q:=\tfrac12(1+e^{-\lambda b})<1.
\end{align*}
The indicator of $\{\tau_{j+1}<T\}$ excludes intervals ending only
because the terminal time is reached. No independence between the
centers or successive intervals is used. Multiplying by
$e^{-\lambda\tau_j}$, taking expectations, and iterating gives
\[
 E\!\left[e^{-\lambda\tau_j}\mathbf1_{\{\tau_j<T\}}\right]
 \le q^j.
\]
Since $e^{-\lambda\tau_j}\ge e^{-\lambda T}$ on $\{\tau_j<T\}$,
\begin{equation}\label{eq:uniq-A1-reset-tail}
 \bP(\tau_j<T)\le e^{\lambda T}q^j.
\end{equation}
In particular,
\[
 E\!\left[\sum_{j\ge1}\mathbf1_{\{\tau_j<T\}}\right]
 \le e^{\lambda T}\frac{q}{1-q}<\infty,
\]
so there are finitely many restarts a.s.

{On each interval, multiply the local identities
\eqref{eq:uniq-A1-translated-weak} by the spatial partition of unity and sum.
More precisely, denote the partition by $\eta_1,\ldots,\eta_N$, so that
\[
 \eta_i\in C_c^\infty(B_{2h}(x_i)),\qquad \eta_i\ge0,\qquad
 \sum_{i=1}^N\eta_i=1
 \quad\hbox{on a neighborhood of }\overline D.
\]
Put $I_j=(\tau_j,\tau_{j+1}]$, $O_i=B_{4h}(x_i)$, and
$A_s^{j,i}=A_s^{\tau_j,i}$. Denote the corresponding local
Doob--Meyer fields by $Z^{j,i},K^{j,i}$ and the local measure by
$\mu^{j,i}$. On $I_j\times B_{3h}(x_i)$, write
\[
 \psi_s^{j,i}(x)
 =Z_s^{j,i}(x-A_s^{j,i})-Dw_s(x)\sigma_i(s),\qquad
 dw=b^{j,i}\,ds+\psi^{j,i}\,dW_s+\mu^{j,i},
\]
where the local drift is the distribution
\[
 b_s^{j,i}
 =-\left[\tfrac12a_i:D^2w_s+
       \sum_k\sigma_i^k\cdot D\psi_s^{j,i,k}
       +\beta_i\cdot Dw_s+f_i\right].
\]
For $\varphi\in C_c^\infty(D)$, extended by zero outside $D$,
the function $\eta_i\varphi$ is an admissible local test. Since
$\sum_i\eta_i\varphi=\varphi$, summing the tested identities gives
\begin{align*}
 d\langle w_s,\varphi\rangle
 ={}&\sum_i\langle b_s^{j,i},\eta_i\varphi\rangle\,ds
   +\sum_k\left(\sum_i
       \langle\psi_s^{j,i,k},\eta_i\varphi\rangle\right)dW_s^k\\
 &+\sum_i\int\eta_i(x)\varphi(x)\,\mu^{j,i}(ds,dx).
\end{align*}
Thus, on $I_j$, define
\[
 b_s^j=\sum_i\eta_i b_s^{j,i},\qquad
 \psi_s^j=\sum_i\eta_i\psi_s^{j,i},\qquad
 \mu^j(ds,dx)=\sum_i\eta_i(x)\mu^{j,i}(ds,dx)
 \quad (x\in D).
\]
Products with distributions are understood through
$\langle\eta_i b_s^{j,i},\varphi\rangle
=\langle b_s^{j,i},\eta_i\varphi\rangle$.
The cutoffs are deterministic and independent of time, so their
multiplication introduces no additional It\^o terms.

The martingale coefficients agree on overlaps. Indeed, for a test
$\chi\in C_c^\infty(B_{3h}(x_i)\cap B_{3h}(x_\ell))$,
subtracting the two local identities shows that the continuous
local martingale with integrand
$\langle\psi_s^{j,i}-\psi_s^{j,\ell},\chi\rangle$ has finite
variation. Its quadratic variation therefore vanishes:
\[
 \int_{\tau_j}^{\tau_{j+1}}
   \left|\langle\psi_s^{j,i}-\psi_s^{j,\ell},\chi\rangle\right|^2\,ds
 =0\quad\hbox{a.s.}
\]
A countable dense set of tests then gives
$\psi^{j,i}=\psi^{j,\ell}$ a.e. on the overlap.

\textbf{Step 3.}
We verify the norms. Constants below are deterministic,
independent of $j$, and may depend on the fixed cover, the data,
and $\delta$. Since $|A_s^{j,i}|\le h$ on $I_j$,
$B_{2h}(x_i)-A_s^{j,i}\subset O_i$. Translation invariance of
Lebesgue measure and boundedness of $Dw$ and $\sigma_i$ give
\[
 \|\psi_s^{j,i}\|_{L^2(B_{2h}(x_i))}^2
 \le 2\|Z_s^{j,i}\|_{L^2(O_i)}^2+C.
\]
The conditional version of \eqref{eq:uniq-A1-dm-bounds} implies
\[
 E_{\sF_{\tau_j}}\int_{\tau_j}^{\tau_{j+1}}
       \sum_i\|\psi_s^{j,i}\|_{L^2(B_{2h}(x_i))}^2\,ds
 \le C\mathbf1_{\{\tau_j<T\}}.
\]
Because the weights are nonnegative and sum to one on $D$,
Jensen's inequality yields
\[
 \|\psi_s^j\|_{L^2(D)}^2
 \le\sum_i\int_D\eta_i(x)|\psi_s^{j,i}(x)|^2\,dx
 \le\sum_i\|\psi_s^{j,i}\|_{L^2(B_{2h}(x_i))}^2.
\]
Consequently,
\[
 E\int_{\tau_j}^{\tau_{j+1}}\|\psi_s^j\|_{L^2(D)}^2\,ds
 \le C\bP(\tau_j<T).
\]

To estimate the drift, integration by parts gives
\begin{align*}
 \langle b_s^j,\varphi\rangle
 =\sum_i\bigg[
 &\tfrac12\langle a_iDw_s,D(\eta_i\varphi)\rangle
 +\sum_k\langle\sigma_i^k\psi_s^{j,i,k},
                         D(\eta_i\varphi)\rangle\\
 &-\langle\beta_i\cdot Dw_s+f_i,\eta_i\varphi\rangle
 \bigg].
\end{align*}
In particular, the cutoff derivatives are included through
\[
 D(\eta_i\varphi)=\eta_iD\varphi+(D\eta_i)\varphi,\qquad
 \|\eta_i\varphi\|_{H^1(\mathbb R^d)}
 \le C_i\|\varphi\|_{H^1(D)}.
\]
Cauchy--Schwarz, the coefficient bounds, and boundedness of $Dw$
therefore imply
\[
 |\langle b_s^j,\varphi\rangle|
 \le C\left(1+\sum_i
             \|\psi_s^{j,i}\|_{L^2(B_{2h}(x_i))}\right)
          \|\varphi\|_{H^1(D)}.
\]
Taking the supremum over tests with $H^1$ norm at most one, and
then using the preceding conditional estimate, gives
\[
 \|b_s^j\|_{H^{-1}(D)}^2
 \le C\left(1+\sum_i
             \|\psi_s^{j,i}\|_{L^2(B_{2h}(x_i))}^2\right),
 \qquad
 E_{\sF_{\tau_j}}\int_{\tau_j}^{\tau_{j+1}}\|b_s^j\|_{H^{-1}(D)}^2\,ds
 \le C \mathbf1_{\{\tau_j<T\}}.
\] 

Concatenate the fields and measures by setting
\[
 a_s=\sum_{j\ge0}\mathbf1_{I_j}(s)b_s^j,\qquad
 \psi_s=\sum_{j\ge0}\mathbf1_{I_j}(s)\psi_s^j,\qquad
 \mu=\sum_{j\ge0}\mathbf1_{I_j}\mu^j,
\]
with zero measure at time zero. The intervals cover $(0,T]$
and there are finitely many nonempty intervals a.s. Summing the
paired local identities,
by \eqref{eq:uniq-A1-reset-tail}, we have
\begin{align*}
 E\int_0^T\bigl(\|a_s\|_{H^{-1}(D)}^2
                    +\|\psi_s\|_{L^2(D)}^2\bigr)\,ds
 &\le C\sum_{j\ge0}\bP(\tau_j<T)
 \le \frac{Ce^{\lambda T}}{1-q}<\infty.
\end{align*}
Spatial difference quotients of the continuous adapted field $w$
give a predictable version of $Dw$. The translated fields
$Z_s^{j,i}(\cdot-A_s^{j,i})$ are predictable, and spatial
multiplication and distributional differentiation are bounded
linear operations between the spaces used above. Moreover,
$\mathbf1_{(\tau_j,\tau_{j+1}]}$ is predictable. These facts give
predictable versions of $a$ and $\psi$.

Positivity follows from $\eta_i\ge0$. For a fixed $D'\Subset D$, put
\[
 M_j=\mu^j(I_j\times D'),\qquad
 Q_j=\sum_i\int_{O_i}
       \bigl(K_{\tau_{j+1}}^{j,i}(z)-K_{\tau_j}^{j,i}(z)\bigr)\,dz.
\]
The transported-measure formula gives
\begin{align*}
 M_j
 &=\sum_i\int_{O_i}\int_{\tau_j}^{\tau_{j+1}}
       \mathbf1_{D'}(z+A_s^{j,i})\eta_i(z+A_s^{j,i})
                                      \,dK_s^{j,i}(z)\,dz\\
 &\le Q_j,
\end{align*}
since $0\le\eta_i\le1$ on $D$. Using
$(\sum_i r_i)^2\le N\sum_i r_i^2$ and the conditional
compensator estimate above, with $C_0=B+LT$, yields
\[
 E_{\sF_{\tau_j}}M_j^2
 \le E_{\sF_{\tau_j}}Q_j^2
 \le 16C_0^2N\sum_i|O_i|^2\,\mathbf1_{\{\tau_j<T\}}.
\]
Thus $EM_j^2\le C\bP(\tau_j<T)$, and Minkowski's inequality
and \eqref{eq:uniq-A1-reset-tail} give
\[
 \|\mu([0,T]\times D')\|_{L^2(\Omega)}
 =\left\|\sum_{j\ge0}M_j\right\|_{L^2(\Omega)}
 \le\sum_{j\ge0}\|M_j\|_{L^2(\Omega)}
 \le \frac{Ce^{\lambda T/2}}{1-\sqrt q}<\infty.
\]

Finally, for $\varphi\in C_c^\infty(D)$, the concatenated identity is
\begin{align*}
 \int_{(0,t]\times D}\varphi(x)\,\mu(ds,dx)
 ={}&\langle w_t-w_0,\varphi\rangle
       -\int_0^t\langle a_s,\varphi\rangle\,ds\\
 &-\sum_k\int_0^t\langle\psi_s^k,\varphi\rangle\,dW_s^k.
\end{align*}
The first term is continuous and adapted. The drift integral is
absolutely continuous, since
\[
 \int_0^T|\langle a_s,\varphi\rangle|\,ds
 \le\sqrt T\,\|\varphi\|_{H^1(D)}
       \left(\int_0^T\|a_s\|_{H^{-1}(D)}^2\,ds\right)^{1/2}<\infty
 \quad\hbox{a.s.}
\]
The stochastic integral has a continuous version, because
\[
 E\int_0^T\sum_k|\langle\psi_s^k,\varphi\rangle|^2\,ds
 \le\|\varphi\|_{L^2(D)}^2
       E\int_0^T\|\psi_s\|_{L^2(D)}^2\,ds<\infty.
\]
Hence the cumulative pairing of $\mu$ is continuous and adapted.
For almost every sample point, continuity of
$K_\cdot^{j,i}(z)$ for almost every $z$ gives
$dK^{j,i}(z)(\{t\})=0$ for every $t$, including the random restart
times. Transport in space, multiplication by the cutoffs, and
concatenation therefore preserve the nonatomic time marginal;
the value at time zero was set to zero. This proves
\eqref{eq:uniq-A1-weak} with all the stated properties.}

\end{proof}

Choose a nonnegative $\rho\in C_c^\infty(B_{1/4}(0))$ with integral
one, put $\rho_\delta(y)=\delta^{-d}\rho(y/\delta)$, and define
\begin{equation}\label{eq:uniq-A1-q}
 q^\delta(t,x)=\int\rho_\delta(y)w^\delta(t,x-y)\,dy.
\end{equation}

\begin{lem}[The exact lower differential inequality]\label{lem-uniq-A1-q}
 Fix $x_*\in\bR^d$ and $R>0$, and set
\[
 D_0=B_R(x_*),\qquad
 D_1=B_{R+\delta/2}(x_*),\qquad
 D=\overline B_{R+\delta/4}(x_*).
\]
Then $q=q^\delta$ satisfies:
\begin{enumerate}
\item[{(i)}] {On
$[0,T]\times\overline{D_0}$, $q$ has a decomposition}
\begin{equation}\label{eq:uniq-A1-q-decomp}
 dq=a_0\,dt+\psi_0\,dW_t+dK,
\end{equation}
where $a_0,\psi_0$ are square-integrable with arbitrarily many bounded
continuous spatial derivatives  on a neighborhood of
$\overline{D_0}$;

\item[{(ii)}] {For every integer $j$, $K$ is
continuous with bounded
variation in $H^j(D_0)$},
$K_0=0$, and a.s.
\[
 K_t(x)-K_s(x)\ge0\
 {(0\le s\le t\le T,\ x\in\overline{D_0})},\qquad
 E[\operatorname{Var}_{H^j}(K)^2]<\infty;
 \footnote{{Here $H^j=H^j(D_0)$, and the}
pathwise total variation is
\[ 
 \operatorname{Var}_{H^j}(K)(\omega)
 :=\sup_{\pi}\sum_{\ell=1}^{N_\pi}
   \|K_{t_\ell}(\omega,\cdot)-K_{t_{\ell-1}}(\omega,\cdot)\|_{H^j(D_0)},
\]
where the supremum is over all finite partitions
$\pi=\{0=t_0<t_1<\cdots<t_{N_\pi}=T\}$ of $[0,T]$. Thus
\[
 E[\operatorname{Var}_{H^j}(K)^2]
 =\int_\Omega\bigl(\operatorname{Var}_{H^j}(K)(\omega)\bigr)^2
       \,\bP(d\omega)
\]
is the second moment of this total variation. By continuity of $K$,
the same supremum is obtained by restricting the interior partition
points to rational times; hence the total variation is measurable.}
\]

\item[{(iii)}] {For an integer $j>d/2+2$,
let $dK=k\,dt+dK^{\mathrm s}$, with $dK^{\mathrm s}\perp dt$, be
the time Lebesgue decomposition of the $H^j(D_0)$-valued measure
$dK$ {(see
\cite[Theorem~I.5.9 and Corollary~IV.1.4]{DiestelUhl1977})}.
Thus $k$ is the density of the absolutely continuous part
of $dK$. Then} $d\bP\otimes dt$-a.e., the density $k$ is spatially continuous 
on $\overline{D_0}$  and 
\begin{equation}\label{eq:uniq-A1-all-v}
 k(t,x)+a_0(t,x)+\mathcal A^v(q,\psi_0)(t,x)+f(t,x,v)\ge0
\quad \forall\,(x,v)\in\overline{D_0}\times U.
\end{equation}

\end{enumerate}
\end{lem}

\begin{proof}
{ 
\textbf{Step 1.} We prove assertions (i) and (ii).
Take the weak decomposition of
Lemma~\ref{lem-uniq-A1-weak} on the open ball $D_1$ and convolve in
space. For $x\in\overline{D_0}$, define
\begin{align*}
 a_0(t,x)
 &=\langle a_t,\rho_\delta(x-\cdot)\rangle_{H^{-1}(D_1),H_0^1(D_1)},\\
 \psi_0(t,x)
 &=\int_{D_1}\rho_\delta(x-z)\psi(t,z)\,dz,\\
 K_t(x)
 &=\int_{[0,t]\times D_1}\rho_\delta(x-z)\,\mu(ds,dz),
 \qquad 0\le t\le T.
\end{align*}
Since $\operatorname{supp}\rho_\delta\Subset B_{\delta/4}(0)$,
\[
 \overline{D_0}-\operatorname{supp}\rho_\delta
 \subset\operatorname{int}D,\qquad D\Subset D_1.
\]
Thus $\rho_\delta(x-\cdot)\in C_c^\infty(D_1)$, and all points
sampled by the convolution on $\overline{D_0}$ lie in the compact
ball $D$. In fact, $D-\operatorname{supp}\rho_\delta\Subset D_1$, 
so the same formulas define the fields on $\operatorname{int}D$,
a neighborhood of $\overline{D_0}$. Differentiating the kernels
in the dual pairing and the spatial integral gives, for every
nonnegative integer $r$,
\[
 \|a_0(t)\|_{C^r(\operatorname{int}D)}
 \le C_{\delta,r}\|a_t\|_{H^{-1}(D_1)},\qquad
 \|\psi_0(t)\|_{C^r(\operatorname{int}D)}
 \le C_{\delta,r}\|\psi_t\|_{L^2(D_1)}.
\]
The square integrability follows from the bounds on $a$ and $\psi$
in Lemma~\ref{lem-uniq-A1-weak}. Convolving its weak identity gives
\eqref{eq:uniq-A1-q-decomp}, proving (i).

Turn to the proof of assertion (ii).
Positivity of $\rho$ and $\mu$
gives monotonicity of $K_t(x)$ for every $x\in\overline{D_0}$.
The preceding lemma gives
$E[\mu([0,T]\times D)^2]<\infty$, so this mass is finite a.s.
For $0\le s<t\le T$ and every spatial multi-index $\alpha$,
differentiation under the integral gives
\[
 D^\alpha(K_t-K_s)(x)
 =\int_{(s,t]\times D}D^\alpha\rho_\delta(x-z)\,\mu(dr,dz),
 \qquad x\in\overline{D_0}.
\]
This is justified by the bounded kernel derivatives and the finite
mass of $\mu$ on $[0,T]\times D$. For the fixed ball $D_0$, put
\[
 C_{\delta,j,D}
 :=\sup_{z\in D}\|\rho_\delta(\cdot-z)\|_{H^j(D_0)}<\infty.
\]
The finiteness follows from smoothness of the kernel and compactness
of $D$. Viewing the convolution increment as an $H^j(D_0)$-valued
integral and using the triangle inequality yields
\begin{align*}
 \|K_t-K_s\|_{H^j(D_0)}
 &\le\int_{(s,t]\times D}
       \|\rho_\delta(\cdot-z)\|_{H^j(D_0)}\,\mu(dr,dz)\\
 &\le C_{\delta,j,D}\,\mu((s,t]\times D).
\end{align*}
Consequently, for every finite partition
$\pi=\{0=t_0<\cdots<t_{N_\pi}=T\}$,
\begin{align*}
 \sum_{\ell=1}^{N_\pi}
    \|K_{t_\ell}-K_{t_{\ell-1}}\|_{H^j(D_0)}
 &\le C_{\delta,j,D}\sum_{\ell=1}^{N_\pi}
          \mu((t_{\ell-1},t_\ell]\times D)\\
 &=C_{\delta,j,D}\,\mu((0,T]\times D).
\end{align*}
Taking the supremum over $\pi$, squaring, and taking expectations gives
\[
 \operatorname{Var}_{H^j}(K)
 \le C_{\delta,j,D}\,\mu([0,T]\times D),\qquad
 E[\operatorname{Var}_{H^j}(K)^2]
 \le C_{\delta,j,D}^2 E[\mu([0,T]\times D)^2]<\infty.
\]
The increment estimate also proves continuity: for $t_n\to t$,
\[
 \|K_{t_n}-K_t\|_{H^j(D_0)}
 \le C_{\delta,j,D}\,
       \mu\bigl((t_n\wedge t,t_n\vee t]\times D\bigr)
 \longrightarrow0
 \quad\hbox{a.s.},
\]
because the finite time marginal $\mu(\,\cdot\,\times D)$ has no atoms.}
The same continuity also
follows from \eqref{eq:uniq-A1-q-decomp}: $q$ is locally continuous with
all spatial derivatives after convolution, and its other two
integrals are continuous in {$H^j(D_0)$}.
{This proves (ii).}

  \textbf{Step 2.} We prove assertion (iii).
{Fix an integer $j>d/2+2$ and put
$H=H^j(D_0)$. By (ii), outside a $\bP$-null set, $dK$ is a
countably additive $H$-valued measure of finite total variation.
The vector-valued Lebesgue decomposition theorem
\cite[Chapter~I, \S5, Theorem~9, p.~31]{DiestelUhl1977} gives
unique measures
\[
 dK=dK^{\mathrm{ac}}+dK^{\mathrm s},\qquad
 dK^{\mathrm{ac}}\ll dt,\qquad |dK^{\mathrm s}|_H\perp dt.
\]
Since $H$ is a Hilbert space, it has the Radon--Nikodym property
\cite[Chapter~IV, \S1, Corollary~4, p.~100]{DiestelUhl1977}.
By the definition of this property
\cite[Chapter~III, \S1, Definition~3, p.~61]{DiestelUhl1977},
there is a Bochner-integrable density $k\in L^1([0,T];H)$ such that
$dK^{\mathrm{ac}}=k\,dt$ pathwise.
The total variation formula for a Bochner integral
\cite[Chapter~II, \S2, Theorem~4(iv), p.~46]{DiestelUhl1977},
together with the variation identity in the Lebesgue decomposition
theorem, gives
\[
 \int_0^T\|k_t\|_H\,dt
 =|dK^{\mathrm{ac}}|_H([0,T])
 \le |dK|_H([0,T])
 =\operatorname{Var}_H(K).
\]
}
{To specify a predictable version, set
\[
 k_t=\lim_{n\to\infty}
       n\bigl(K_t-K_{(t-1/n)^+}\bigr)
 \quad\text{in }H^j(D_0)
\]
where the limit exists, and set $k_t=0$ otherwise. Each
approximating process is continuous and adapted, hence predictable.
Hilbert-valued Lebesgue differentiation identifies this limit with
the density of the absolutely continuous part of $dK$, since the singular part
has derivative zero for almost every time. Thus this choice is
predictable and agrees with the density $d\bP\otimes dt$-a.e.}
Taking $j>d/2+2$
gives the claimed spatial continuity {on $\overline{D_0}$}.

The differential inequality is obtained directly from the dynamic
programming principle, avoiding a commutator with random coefficients.
For any original admissible control $v$ and its state process
{$X=X^{a,x;v}$, with $a\in[0,T]$ and $x\in D_0$},
the translated state $X_s-y$ is admissible for the shaken problem
using the constant shake ${\epsilon}_s=y$. Its coefficients are evaluated
at $X_s$. Consequently, for every $y$ in the support of
$\rho_\delta$, the process
\[
 w^\delta(s,X_s-y)+\int_a^s f(r,X_r,v_r)\,dr
\]
is a submartingale.
{More explicitly, for $a\le r\le t\le T$, the DPP and
conditional Fubini give
\begin{align*}
 q(r,X_r)
 &=\int\rho_\delta(y)w^\delta(r,X_r-y)\,dy\\
 &\le E_{\sF_r}\!\left[
       \int\rho_\delta(y)w^\delta(t,X_t-y)\,dy
       +\int_r^t f(s,X_s,v_s)\,ds\right]\\
 &=E_{\sF_r}\!\left[q(t,X_t)+\int_r^t f(s,X_s,v_s)\,ds\right].
\end{align*}
Conditional Fubini is justified by boundedness of $w^\delta$ and
$f$, and $\int\rho_\delta=1$.} Thus
\begin{equation}\label{eq:uniq-A1-q-submart}
 q(s,X_s)+\int_a^s f(r,X_r,v_r)\,dr
 \quad\text{is a submartingale.}
\end{equation}
{Stop the process in \eqref{eq:uniq-A1-q-submart} at
\[
 \tau_{D_0}^{a,x;v}
 :=\inf\{s\in[a,T]:X_s^{a,x;v}\notin D_0\}\wedge T,
 \qquad \inf\varnothing:=\infty.
\]
Before this exit time, $X_s\in D_0$, and the stopped path takes
values in $\overline{D_0}$. All subsequent density statements along
$X$ are for $a\le s<\tau_{D_0}^{a,x;v}$.}

The It\^o-Kunita-Wentzell-Krylov formula for \eqref{eq:uniq-A1-q-decomp} gives the finite
variation part of \eqref{eq:uniq-A1-q-submart} as
\begin{equation}\label{eq:uniq-A1-fv-on-path}
 \big[a_0+\mathcal A^v(q,\psi_0)+f\big](s,X_s)\,ds
                  +dK_s(X_s).
\end{equation}
Here $dK_s(X_s)$ means integration of the spatially smooth measure
$dK_s(x)$ evaluated at {$x=X_s\in\overline{D_0}$},
not the differential of $K_s(X_s)$.
For a direct justification, write $F=q-K$, so that
$dF=a_0\,dt+\psi_0\,dW$. The local spatial regularity and integrability
above allow the usual It\^o-Kunita-Wentzell-Krylov formula for $F$, after a spatial
cutoff {equal to one on a neighborhood of
$\overline{D_0}$ and supported where the convolution is defined}.
It remains to compute $K(t,X_t)$. All identities below are
stopped at the preceding exit time; we suppress this stopping in the
notation. {Here and below, the $C^1$ and $C^2$ norms and
total variations are taken on $\overline{D_0}$.} For a partition
$\pi=\{a=t_0<\cdots<t_N=t\}$ of mesh
$|\pi|=\max_j(t_j-t_{j-1})$, split each increment as
\begin{align*}
 K_{t_j}(X_{t_j})-K_{t_{j-1}}(X_{t_{j-1}})
 ={}& (K_{t_j}-K_{t_{j-1}})(X_{t_j})\\
 &+K_{t_{j-1}}(X_{t_j})-K_{t_{j-1}}(X_{t_{j-1}}).
\end{align*}
If $\omega_X(h)=\sup_{|r-s|\le h}|X_r-X_s|$, the difference between
the sum of the first terms and $\int_a^t dK_s(X_s)$ is at most
\[
 \operatorname{Var}_{C^1}(K)\,\omega_X(|\pi|)\longrightarrow0
 \quad\text{a.s}.
\]
For the second terms, $K_{t_{j-1}}$ is an
$\sF_{t_{j-1}}$-measurable $C^2$ function, so spatial It\^o's formula
applies on $[t_{j-1},t_j]$. Continuity of $K$ in $C^2$, boundedness
of the state coefficients, and
$E[\operatorname{Var}_{C^2}(K)^2]<\infty$ give convergence of the
drift integrals in $L^1$ and the stochastic integrals in $L^2$. Hence
\begin{align*}
 K_t(X_t)-K_a(X_a)
 ={}&\int_a^t dK_s(X_s)
     +\int_a^t\mathcal A^{v_s}(K,0)(s,X_s)\,ds\\*
 &+\int_a^t DK_s(X_s)\sigma(s,X_s,v_s)\,dW_s.
\end{align*}
Adding this identity to the formula for $F$ proves
\eqref{eq:uniq-A1-fv-on-path}, including the cross term
$\sum_k\sigma^k\cdot D\psi_0^k$. The field $K$ contributes no
additional quadratic-covariation term.

{Write $dK=k\,ds+dK^{\mathrm s}$ in $H^j(D_0)$,
with $j>d/2+2$, and put
$\lambda^{\mathrm s}:=|dK^{\mathrm s}|_{H^j(D_0)}$.
Work pathwise outside a null set. Since $\lambda^{\mathrm s}\perp ds$,
there is a Borel set $N_\omega\subset[0,T]$ such that
\[
 |N_\omega|=0,\qquad
 \lambda^{\mathrm s}([0,T]\setminus N_\omega)=0.
\]
The Hilbert-valued Radon--Nikodym representation
{(applied to
$dK^{\mathrm s}\ll\lambda^{\mathrm s}$), together with the total
variation formula
\cite[Corollary~IV.1.4 and Theorem~II.2.4(iv)]{DiestelUhl1977}}, gives
\[
 dK_s^{\mathrm s}
   =\kappa_s^{\mathrm s}\,\lambda^{\mathrm s}(ds),\qquad
 \|\kappa_s^{\mathrm s}\|_{H^j(D_0)}=1
 \quad\lambda^{\mathrm s}\text{-a.e.}
\]
Restore the stopping notation by writing
$\tau=\tau_{D_0}^{a,x;v}$. For a Borel set $A\subset[0,T]$, define
the evaluated singular measure by
\[
 \nu_X^{\mathrm s}(A)
 =\int_{A\cap(a,\tau]}
        \kappa_s^{\mathrm s}(X_s)\,\lambda^{\mathrm s}(ds).
\]
The embedding $H^j(D_0)\hookrightarrow C^2(\overline{D_0})$
and $X_s\in\overline{D_0}$ for $a\le s\le\tau$ imply
\begin{align*}
 |\nu_X^{\mathrm s}|(A)
 &\le\int_{A\cap(a,\tau]}
        |\kappa_s^{\mathrm s}(X_s)|\,\lambda^{\mathrm s}(ds)\\*
 &\le C\int_{A\cap(a,\tau]}
        \|\kappa_s^{\mathrm s}\|_{H^j(D_0)}
             \,\lambda^{\mathrm s}(ds)
 \le C\lambda^{\mathrm s}(A).
\end{align*}
Consequently,
$|\nu_X^{\mathrm s}|([0,T]\setminus N_\omega)=0$ and
$\nu_X^{\mathrm s}\perp ds$. Thus evaluation preserves singularity
in time; it cannot create mass outside $N_\omega$.

For $a\le t\le T$, set
\begin{align*}
 S_t^v
 &:=q(t\wedge\tau,X_{t\wedge\tau})
       +\int_a^{t\wedge\tau}f(s,X_s,v_s)\,ds,\\
 M_t^v
 &:=\int_a^{t\wedge\tau}
       [\psi_0(s,X_s)+Dq(s,X_s)\sigma(s,X_s,v_s)]\,dW_s,\\
 R_s^v(x)
 &:=k(s,x)+a_0(s,x)
       +\mathcal A^{v_s}(q,\psi_0)(s,x)+f(s,x,v_s),\\
 C_t^v
 &:=\int_a^{t\wedge\tau}R_s^v(X_s)\,ds
       +\nu_X^{\mathrm s}((a,t]).
\end{align*}
The preceding composition formula gives
\[
 S_t^v=S_a^v+M_t^v+C_t^v.
\]
All spatial norms in the following bounds are over
$\overline{D_0}$. It\^o's isometry and boundedness of $\sigma$ give
\[
 E|M_T^v|^2
 \le C E\int_0^T
       \bigl(\|\psi_0(s)\|_\infty^2+\|Dq(s)\|_\infty^2\bigr)\,ds
 <\infty.
\]
Moreover,
\begin{align*}
 E\operatorname{Var}_{[a,T]}(C^v)
 &\le E\int_a^\tau |R_s^v(X_s)|\,ds
        +C E\lambda^{\mathrm s}([0,T])\\
 &\le C E\!\left[
       \int_0^T\bigl(\|a_0(s)\|_\infty
              +\|q(s)\|_{C^2}+\|\psi_0(s)\|_{C^1}+1\bigr)\,ds
       +\operatorname{Var}_{H^j(D_0)}(K)\right]
 <\infty.
\end{align*}
Here we used
$\int_0^T\|k_s\|_{H^j(D_0)}\,ds
\le\operatorname{Var}_{H^j(D_0)}(K)$ and
$\lambda^{\mathrm s}([0,T])
\le\operatorname{Var}_{H^j(D_0)}(K)$.
Since $C^v=S^v-S_a^v-M^v$, it is continuous and adapted,
hence predictable. The stopped process $S^v$ is a bounded
submartingale. Uniqueness of its Doob--Meyer decomposition therefore
implies that $C^v$ is increasing. Since $\nu_X^{\mathrm s}\perp ds$,
the absolutely continuous part of its measure is
\[
 (dC^v)^{\mathrm{ac}}(ds)
 =\mathbf1_{(a,\tau]}(s)R_s^v(X_s)\,ds\ge0.
\]
It follows that, on $a<s<\tau$,}
\[
 [k+a_0+\mathcal A^v(q,\psi_0)+f](s,X_s)\ge0
 \quad d\bP\otimes ds\text{-a.e.}
\]

{To obtain one null set for all $x\in\overline{D_0}$
and $v\in U$, choose a countable dense set $\mathcal D_U\subset U$,
and put
\[
 \mathcal D_x:=D_0\cap\mathbb Q^d,\qquad
 \mathcal Q_T:=\mathbb Q\cap[0,T).
\]
For a constant control $v$, use the residual $R_s^v(x)$ defined
above with $v_s\equiv v$. For each
$(a,x,v)\in\mathcal Q_T\times\mathcal D_x\times\mathcal D_U$,
the preceding inequality says that
\[
 \mathcal N_{a,x,v}
 :=\bigl\{(\omega,s):a<s<\tau_{D_0}^{a,x;v}(\omega),\
                  R_s^v(X_s^{a,x;v})<0\bigr\}
\]
has $d\bP\otimes ds$-measure zero. By countability, the same is
true of
\[
 \mathcal N
 :=\bigcup_{(a,x,v)\in\mathcal Q_T\times\mathcal D_x\times\mathcal D_U}
       \mathcal N_{a,x,v}.
\]
Outside $\mathcal N$, the density inequality holds  
for all such triples, whenever $a<s<\tau_{D_0}^{a,x;v}$.

For the underlying unstopped SDE, fix $(x,v)\in\mathcal D_x\times
\mathcal D_U$ and extend $X^{a,x;v}$ by $x$ before its starting
time $a$. All trajectories below use the same Brownian motion and
the same constant control $v$. For every $p>2$ and $|a-b|\le1$,
the standard estimates on the stability of SDEs with bounded Lipschitz 
coefficients give
\[
 E\left[\sup_{0\le s\le T}
       |X_s^{a,x;v}-X_s^{b,x;v}|^p\right]
 \le C_p|a-b|^{p/2}.
\]
Since $p/2>1$, Kolmogorov's continuity criterion applies with values
in the separable Banach space $\mathbb B=C([0,T];\bR^d)$.
For any fixed $0<\alpha<1/2-1/p$, it gives versions satisfying
\[
 \|X^{a,x;v}-X^{b,x;v}\|_{\mathbb B}
 \le C_{x,v}(\omega)|a-b|^\alpha,\qquad a,b\in[0,T],
\]
where $C_{x,v}(\omega)<\infty$ a.s. Taking a countable intersection
over $(x,v)\in\mathcal D_x\times\mathcal D_U$ gives a common
event $\Omega_*$ of probability one. We also require on $\Omega_*$
that these versions agree with the original trajectories for all
$a\in\mathcal Q_T$; this is possible by countability.
Enlarge $\mathcal N$ by
$(\Omega\setminus\Omega_*)\times[0,T]$ and by the product-null set
where the residual is not continuous in $(x,v)$.

Fix $(\omega,s)\notin\mathcal N$ with $0<s<T$ and
$(x,v)\in\mathcal D_x\times\mathcal D_U$.
For rational $b_m\downarrow s$ with $b_m>s$, the extension convention
gives $X_r^{b_m,x;v}=x$ for $0\le r\le s$. Hence
\[
 \sup_{0\le r\le s}|X_r^{s,x;v}-x|
 \le\|X^{s,x;v}-X^{b_m,x;v}\|_{\mathbb B}
 \longrightarrow0.
\]
Thus the limiting trajectory $X^{s,x;v}$ also equals $x$ on $[0,s]$.
Now choose rational $a_n\uparrow s$ with $a_n<s$. Then
\[
 S_n:=\sup_{a_n\le r\le s}|X_r^{a_n,x;v}-x|
 \le\|X^{a_n,x;v}-X^{s,x;v}\|_{\mathbb B}
 \longrightarrow0.
\]
Put $d_x=\operatorname{dist}(x,\partial D_0)>0$ and
$\tau_n=\tau_{D_0}^{a_n,x;v}$. For sufficiently large $n$,
$S_n<d_x/2$. If $\tau_n\le s<T$, continuity of the trajectory and
the definition of the exit time give
$X_{\tau_n}^{a_n,x;v}\in\partial D_0$, and therefore
\[
 d_x\le |X_{\tau_n}^{a_n,x;v}-x|\le S_n<d_x/2,
\]
a contradiction. Thus $\tau_n>s$. Independently, the convergence
$S_n\to0$ gives
\[
 |X_s^{a_n,x;v}-x|\le S_n\longrightarrow0.
\]
Since $(\omega,s)\notin\mathcal N$, the density inequality holds
along each of these rational-start trajectories before its exit.
Spatial continuity consequently yields
\[
 R_s^v(x)=\lim_{n\to\infty}R_s^v(X_s^{a_n,x;v})\ge0,
 \qquad (x,v)\in\mathcal D_x\times\mathcal D_U.
\]
For any $x\in\overline{D_0}$, choose $x_m\in\mathcal D_x$ with
$x_m\to x$. For every $v\in\mathcal D_U$,
\[
 R_s^v(x)=\lim_{m\to\infty}R_s^v(x_m)\ge0.
\]
Finally, for any $v\in U$, choose $v_m\in\mathcal D_U$ with
$v_m\to v$. Continuity of the coefficients in $v$ gives
\[
 R_s^v(x)=\lim_{m\to\infty}R_s^{v_m}(x)\ge0,
 \qquad (x,v)\in\overline{D_0}\times U.
\]
All these limits are taken at the same $(\omega,s)\notin\mathcal N$,
so this proves \eqref{eq:uniq-A1-all-v} on one common
$d\bP\otimes ds$-full set. Since $x_*$ and $R$ were arbitrary,
this proves (iii) on any fixed compact spatial ball.}
\end{proof}

\begin{lem}[Removing the positive finite-variation part]
\label{lem-uniq-A1-regularize}
Fix a compact spatial ball $D=\overline B_R(x_*)$ on which \eqref{eq:uniq-A1-q-decomp} and
\eqref{eq:uniq-A1-all-v} hold, with the regularity of Lemma~\ref{lem-uniq-A1-q} on
 an open neighborhood $O\supset D$.
There are adapted fields $q_n$ on $[0,T]\times O$ and predictable
fields $(\mathfrak d_tq_n,\mathfrak d_\omega q_n)$ satisfying:
\begin{enumerate}
      \item [(i)] 
      For every $n\ge1$ and every open ball $B\Subset O$,
 \begin{align*}
 &\|q_n\|_{\cS^2(C(\overline B))}^2
   +\|q_n\|_{\cL^2(C^2(\overline B))}^2
 \quad
   +\|\mathfrak d_tq_n\|_{\cL^2(C(\overline B))}^2
   +\|\mathfrak d_\omega q_n\|_{\cL^2(C^1(\overline B;\bR^m))}^2
   <\infty;
 \end{align*}
\item [(ii)] 
There holds a.s.,
\begin{equation*} 
 q_n(t,x)=q_n(T,x)
   -\int_t^T\mathfrak d_sq_n(s,x)\,ds
   -\int_t^T\mathfrak d_\omega q_n(s,x)\,dW_s,
 \qquad (t,x)\in[0,T]\times O;
\end{equation*}
 \item [(iii)] On $[0,T]\times D$,
\begin{align}\label{eq:uniq-A1-qn}
 &q_n\le q,\qquad q_n\longrightarrow q
       \quad\text{in } \cS^2(C^2(D)),\notag\\
 &\mathfrak d_tq_n+\mathcal A^v(q_n,\mathfrak d_\omega q_n)+f(t,x,v)\ge0
       \quad\forall\, (x,v)\in D\times U,\quad d\bP\otimes dt\text{-a.e.}
\end{align}
\end{enumerate}
\end{lem}

\begin{proof}
{\textbf{Step 1. Construction.}
We retain the fields $q,a_0,\psi_0,K$ on $O$ supplied by the
convolution construction in Lemma~\ref{lem-uniq-A1-q}. On each ball
$B\Subset O$, the same convolved weak identity gives
\eqref{eq:uniq-A1-q-decomp}; the fields are not chosen anew on
different balls.
Fix an integer $j>d/2+2$, write
$H^j=H^j(\operatorname{int}D)$, and put
$\mathcal V_j=\operatorname{Var}_{H^j}(K)$.
We write $\|u\|_\infty=\sup_{x\in D}|u(x)|$.}
Extend $K_t=0$ for $t<0$, and put
\[
 K_t^n=n\int_{t-1/n}^tK_s\,ds,\qquad
 k_t^n=n(K_t-K_{t-1/n}),\qquad d_n=K^n-K.
\]
{These fields are defined on $O$ and are continuous and
adapted, hence predictable.}
Pointwise monotonicity gives $k^n\ge0$ and $d_n\le0$.
{Since $K_0=0$, the extension remains continuous in $H^j$.
Consequently, pathwise,
\[
 \sup_{0\le t\le T}\|d_n(t)\|_{H^j}
 \le\sup_{\substack{-1\le s\le t\le T\\t-s\le1/n}}
       \|K_t-K_s\|_{H^j}\longrightarrow0,
 \qquad
 \sup_{0\le t\le T}\|d_n(t)\|_{H^j}\le\mathcal V_j.
\]
Since $E[\mathcal V_j^2]<\infty$, dominated convergence gives
$d_n\to0$ in $\cS^2(H^j)$. The embedding
$H^j(\operatorname{int}D)\hookrightarrow C^2(D)$ therefore yields
$d_n\to0$ in $\cS^2(C^2(D))$.

For the time derivatives, extend $k$ and $dK^{\mathrm s}$ by zero
before time zero. The Lebesgue decomposition of $dK$ gives
\[
 k_t^n
 =n\int_{t-1/n}^t k_s\,ds
   +n\,dK^{\mathrm s}((t-1/n,t]).
\]
Banach-valued Lebesgue differentiation gives ($d\bP\otimes dt$-a.e.) convergence of the
first term to $k_t$ in $H^j$.
{For the second term, fix a $\omega$
outside the exceptional probability-zero set and put
$\nu_\omega=|dK^{\mathrm s}(\omega)|_{H^j}$, extended by zero
outside $[0,T]$. This finite positive total variation measure is
concentrated on the same Lebesgue-null set as
$dK^{\mathrm s}(\omega)$, so $\nu_\omega\perp dt$.
The Lebesgue differentiation theorem for finite measures gives
\[
 \frac{\nu_\omega((t-h,t+h))}{2h}\longrightarrow0
 \quad(h\downarrow0),\qquad\text{for (Lebesgue-)a.e. }t.
\]
Taking $h=1/n$ and using the definition of total variation and
the inclusion $(t-1/n,t]\subset(t-1/n,t+1/n)$, we obtain
\[
 \begin{aligned}
 \bigl\|n\,dK^{\mathrm s}(\omega)((t-1/n,t])\bigr\|_{H^j}
 &\le n\,\nu_\omega((t-1/n,t])\\
 &\le2\,\frac{\nu_\omega((t-1/n,t+1/n))}{2/n}
 \longrightarrow0.
 \end{aligned}
\]
The exceptional time set may depend on $\omega$. Combining the
two terms above gives $k_t^n\to k_t$ in $H^j$, $d\bP\otimes dt$-a.e. 
Joint measurability and Fubini's theorem
therefore give this convergence $d\bP\otimes dt$-a.e.; Sobolev
embedding gives the same convergence in $C^2(D)$.}}

Define
\[
 g=-a_0-\bH(t,x,Dq,D^2q,D\psi_0),\qquad
 r_n(t)=\sup_{{x\in D}}(g(t,x)-k_t^n(x))^+.
\]
By \eqref{eq:uniq-A1-all-v}, we have $d\bP\otimes dt$-a.e. $k\ge g$ and 
\[
 0\le r_n(t)\le\|k_t-k_t^n\|_{C(D)}\longrightarrow0.
\]
Also, $k^n\ge0$ gives $r_n\le\|g^+\|_\infty$, and
\[
 \|g^+\|_\infty
 \le\|a_0\|_\infty
   +C_L\bigl(\|Dq\|_\infty+\|D^2q\|_\infty
                         +\|D\psi_0\|_\infty\bigr)+L
 \in L^2(\Omega\times[0,T]).
\]
Here $C_L$ is a deterministic constant depending on the coefficient
bounds. This domination does not require $k\in L^2$.
Dominated convergence gives
$\|r_n\|_{\cL^2}\to0$. 
{Separability of $U$ and continuity in $v$ allow the
Hamiltonian infimum to be taken over a countable dense subset of $U$.
The uniform spatial Lipschitz bounds on the coefficients and spatial
continuity of $Dq,D^2q,D\psi_0$ give a common spatial modulus of
continuity for the expressions under this infimum. Thus $g$ is
spatially continuous and predictable. Taking the supremum defining
$r_n$ over a countable dense subset of $D$ shows that $r_n$ is
predictable.}

For a deterministic coefficient bound $C_L$, set
\[
 e_n=r_n+C_L(\|Dd_n\|_\infty+\|D^2d_n\|_\infty),\qquad
 \eta_n(t)=E\!\left[\int_t^T e_n(s)\,ds\,\middle|\,\sF_t\right].
\]
{Then $e_n\ge0$ and $e_n\to0$ in $\cL^2(\bR)$.
Standard BSDE theory and its $L^2$ estimates
(see \cite[Section~2.1]{Karoui_Peng_Quenez}) give a unique pair
$(\eta_n,\zeta_n)\in\cS^2(\bR)\times\cL^2(\bR^m)$ satisfying
\[
 d\eta_n=-e_n\,dt+\zeta_n\,dW,\qquad \eta_n(T)=0.
\] 
 Taking conditional
expectations identifies $\eta_n$ with the expression above. Moreover, for a constant $C_T$ independent of $n$,
\[
 \|\eta_n\|_{\cS^2}^2+\|\zeta_n\|_{\cL^2}^2
 \le C_T\|e_n\|_{\cL^2}^2\longrightarrow0.
\]}
Define
\[
 q_n=q+d_n-\eta_n.
\]
{\textbf{Step 2. Assertion (i).}
The field $q_n$ is adapted. Set
\[
 \mathfrak d_tq_n=a_0+k^n+e_n,\qquad
 \mathfrak d_\omega q_n=\psi_0-\zeta_n.
\]
These coefficient fields are predictable: this holds for
$a_0,\psi_0$ by the preceding construction, for $k^n$ by continuity
and adaptedness, and for $e_n,\zeta_n$ by their definitions above.

Fix an open ball $B\Subset O$ and put
$\mathcal V_{j,B}=\operatorname{Var}_{H^j(B)}(K)$.
The estimates of Step 1, now on $B$, give
\[
 \begin{aligned}
 \|d_n\|_{\cS^2(C^2(\overline B))}^2
   &\le C_{B,j}E[\mathcal V_{j,B}^2]<\infty,\\
 \|k^n\|_{\cL^2(C^2(\overline B))}^2
   &\le C_{B,j}n^2T\,E[\mathcal V_{j,B}^2]<\infty.
 \end{aligned}
\]
For the field $q=\rho_\delta*w^\delta$, joint continuity of
$w^\delta$ and \eqref{eq:uniq-A1-w-bounds} imply, for every spatial
multi-index $\alpha$,
\[
 D^\alpha q(t,x)
 =\int D^\alpha\rho_\delta(x-z)w^\delta(t,z)\,dz,\qquad
 \sup_{t,x}|D^\alpha q(t,x)|
 \le L(T+1)\|D^\alpha\rho_\delta\|_{L^1}.
\]
These derivatives are jointly continuous; hence
$q\in\cS^2(C^r(\overline B))$ for every nonnegative integer $r$.
In particular,
\[
 \|q_n\|_{\cS^2(C^2(\overline B))}
 \le\|q\|_{\cS^2(C^2(\overline B))}
     +\|d_n\|_{\cS^2(C^2(\overline B))}
     +\|\eta_n\|_{\cS^2(\bR)}<\infty.
\]
This gives both the $\cS^2(C(\overline B))$ bound and
\[
 \|q_n\|_{\cL^2(C^2(\overline B))}^2
 \le T\|q_n\|_{\cS^2(C^2(\overline B))}^2<\infty.
\]
Finally, the preceding lemma and the bounds on
$k^n,e_n,\zeta_n$ give
\[
 \mathfrak d_tq_n\in\cL^2(C(\overline B)),\qquad
 \mathfrak d_\omega q_n\in\cL^2(C^1(\overline B;\bR^m)).
\]
This proves all four bounds in (i).

\textbf{Step 3. Assertion (ii).}
On the fixed ball $B$, the time average satisfies
$dK^n=k^n\,dt$ in $H^j(B)$. Combining
\eqref{eq:uniq-A1-q-decomp} with the identity for $\eta_n$ cancels
the entire $dK$ term and gives
\[
 q_n(t)-q_n(0)
 =\int_0^t(a_0+k^n+e_n)(s)\,ds
   +\int_0^t(\psi_0-\zeta_n)(s)\,dW_s
 \quad\text{in }H^j(B).
\]
The arbitrarily high spatial regularity of $a_0,\psi_0$ and the
$H^j(B)$ variation bound for $K$ imply
\[
 a_0+k^n+e_n\in\cL^2(H^j(B)),\qquad
 \psi_0-\zeta_n\in\cL^2(H^j(B;\bR^m)).
\]
The drift integral is continuous in $H^j(B)$, and the stochastic
integral has a continuous $H^j(B)$-valued version, with
\[
 E\sup_{0\le t\le T}
   \left\|\int_0^t(\psi_0-\zeta_n)(s)\,dW_s\right\|_{H^j(B)}^2
 \le4E\int_0^T\|\psi_0-\zeta_n\|_{H^j(B;\bR^m)}^2\,ds<\infty.
\]
Sobolev embedding therefore makes the identity   for
all $(t,x)\in[0,T]\times\overline B$ a.s.
Choose a countable cover of $O$ by balls compactly contained in
$O$, and intersect these probability-one events for all balls and
all $n$. The versions agree on overlaps because they integrate
the same coefficient fields. Subtracting the identity at $t$
from the identity at $T$ proves (ii) on
$[0,T]\times O$ a.s.

\textbf{Step 4. Assertion (iii).}
Since $\eta_n$ and $\zeta_n$ are spatially constant,
$Dq_n=Dq+Dd_n$, $D^2q_n=D^2q+D^2d_n$, and
$D\mathfrak d_\omega q_n=D\psi_0$.}
{More explicitly, for every $v\in U$,
\[
\begin{aligned}
 &\mathfrak d_tq_n+\mathcal A^v(q_n,\mathfrak d_\omega q_n)+f\\
 &\quad=a_0+k^n+e_n+\mathcal A^v(q,\psi_0)+f
       +\beta\cdot Dd_n
       +\tfrac12\sigma\sigma':D^2d_n\\
 &\quad\ge k^n-g+e_n
       -C_L\bigl(\|Dd_n\|_\infty+\|D^2d_n\|_\infty\bigr)\\
 &\quad=k^n-g+r_n\ge0.
\end{aligned}
\]
}
The definition
of $e_n$ therefore gives the last inequality in \eqref{eq:uniq-A1-qn}.
{This holds on the common $d\bP\otimes dt$-full set
used in Step 1, for every $(x,v)\in D\times U$.
Moreover, $d_n\le0$ and $\eta_n\ge0$ give $q_n\le q$ on
$[0,T]\times D$ a.s., and
\[
 \|q_n-q\|_{\cS^2(C^2(D))}
 \le\|d_n\|_{\cS^2(C^2(D))}+\|\eta_n\|_{\cS^2}
 \longrightarrow0.
\]
This completes the proof of (iii).}
\end{proof}

The lower fields approximate the value of this bounded-data
problem. Indeed, couple the original state with the shaken state
using the same original control and any shaking control with
$|{\epsilon}_s|\le\delta$. The uniform spatial Lipschitz bounds give,
conditionally at the starting time $t$,
\[
 E\!\left[\sup_{t\le s\le T}|X_s^{\theta,{\epsilon}}-X_s^\theta|^2
               \,\middle|\,\sF_t\right]\le C\delta^2.
\]
The cost difference is bounded by $C\delta$, uniformly in the
controls, the starting time, and the state. Since ${\epsilon}=0$ is allowed,
\begin{equation}\label{eq:uniq-A1-shaking-stability}
 0\le V-w^\delta\le C\delta,
 \qquad |q^\delta-w^\delta|\le C\delta.
\end{equation}
The second estimate uses the uniform spatial Lipschitz bound on
$w^\delta$ and the support of $\rho_\delta$. Consequently,
$q^\delta\to V$ uniformly as $\delta\downarrow0$. Moreover,
\[
 q^\delta(T,x)=(G*\rho_\delta)(x),\qquad
 q_n(T,x)\le (G*\rho_\delta)(x),
\]
by the definition of $w^\delta$ and Lemma~\ref{lem-uniq-A1-regularize}.
The shaking control is used only to construct these auxiliary
fields; it does not enter the definition of viscosity solutions.

\subsection{Maximality and uniqueness under $(\cA1)$}
\label{sec-polynomial-comparison}

From now on, we work under Assumption $(\cA1)$ and in the polynomial-growth class
$\cS^2_{\mathrm{pol}}$. The bounded-data constructions above are
applied only to the approximating data of
Lemma~\ref{lem-poly-approximation} under Assumption $(\cA1_{\mathrm b})$.

\begin{lem}[Comparison on a spatial ball]
\label{lem-cmp-polynomial-local}
Let $(\cA1)$ hold, let $D=\overline B_R$ with $R>0$, and let
  $O\subset\bR^d$ be a deterministic open set
with $D\subset O$. Let $\phi$ be an adapted field on $[0,T]\times O$
with predictable fields $(\mathfrak d_t\phi,\mathfrak d_\omega\phi)$
such that, for every open ball $B\Subset O$,
\begin{align*}
 \phi&\in\cS^2(C(\overline B))\cap\cL^2(C^2(\overline B)),\\
 \mathfrak d_t\phi&\in\cL^2(C(\overline B)),\qquad
 \mathfrak d_\omega\phi\in\cL^2(C^1(\overline B;\bR^m)),
\end{align*}
and, a.s.,  
\begin{align*}
 \phi(t,x)=\phi(T,x)
   -\int_t^T\mathfrak d_s\phi(s,x)\,ds
   -\int_t^T\mathfrak d_\omega\phi(s,x)\,dW_s, \qquad \forall\, 
   (t,x)\in[0,T]\times O.
\end{align*}
Then the following comparison statements hold.
\begin{enumerate}
\item[(i)] Suppose that $u$ is a viscosity subsolution,
$\phi(T,\cdot)\ge G$ on $D$ a.s., and, for some deterministic
$0<r_0<R$,
\begin{equation}\label{eq:cmp-polynomial-local-collar}
 u(s,x)-\phi(s,x)\le0
 \quad\text{for }0\le s\le T,\quad r_0\le|x|\le R,
 \quad\text{a.s.}
\end{equation}
If $R_\phi\ge0$ on $D$ $d\bP\otimes dt$-a.e.,
then $u\le\phi$ on $[0,T]\times D$ a.s.

\item[(ii)] Suppose that $u$ is a viscosity supersolution,
$\phi(T,\cdot)\le G$ on $D$ a.s., and, for some deterministic
$0<r_0<R$,
\[
 \phi(s,x)-u(s,x)\le0
 \quad\text{for }0\le s\le T,\quad r_0\le|x|\le R,
 \quad\text{a.s.}
\]
If $R_\phi\le0$ on $D$ $d\bP\otimes dt$-a.e.,
then $\phi\le u$ on $[0,T]\times D$ a.s.
\end{enumerate}
\end{lem}

\begin{proof}
We only prove assertion (i) for viscosity subsolution, 
as the proof for the supersolution case is similar.

  Set
\begin{align*}
 M_s=\max_{x\in D}(u-\phi)^+(s,x),\quad s\in[0,T].
\end{align*}
Then $M\in\cS^2(\bR)$ with $M_T=0$. Suppose $\alpha=M_t$ is positive with positive
probability at a deterministic $t<T$. Choose $\kappa>0$ with
$\bP(\alpha\ge\kappa)>0$, and let $Z$ be the Snell envelope of
\[
 Y_s=M_s+\frac{\alpha(s-t)}{2(T-t)},\qquad t\le s\le T,
\]
i.e.,
\[
 Z_s=\esssup_{\rho\in\mathcal T^s}E_{\sF_s}[Y_\rho],
 \qquad t\le s\le T,
\]
where $\mathcal T^s$ is the set of stopping times valued in $[s,T]$.
 Since $Y$ is continuous and belongs to $\cS^2(\bR)$,
Lemma~\ref{lem-snell-envelope}(iii), applied on $[t,T]$, shows
that its first contact time
\[
 \rho_*=\inf\{s\in[t,T]:Z_s=Y_s\}
\]
is optimal. Thus
$Z_t=E_{\sF_t}Y_{\rho_*}$, whereas
\begin{align*}
 E_{\sF_t}Y_T=\alpha/2<Y_t\le Z_t
 \quad\text{on }\{\alpha\ge\kappa\}.
\end{align*}
Thus $\bP(\alpha\ge\kappa,\rho_*<T)>0$; otherwise
$\rho_*=T$   a.s. on this $\sF_t$-measurable event,
contradicting the above inequality. Choose a deterministic
$T_0\in(t,T)$ such that
\[
 \tau=\rho_*\wedge T_0,\qquad
 A=\{\alpha\ge\kappa,\ \rho_*\le T_0\}\in\sF_\tau,
 \qquad \bP(A)>0.
\]
Since $\alpha$ is $\sF_t$-measurable and $t\le\tau$,
 Lemma~\ref{lem-snell-envelope}(i) gives
\[
 Z_\tau\ge E_{\sF_\tau}Y_T=\alpha/2.
\]
On $A$, we have $Z_\tau=Y_\tau$ and thus,
\begin{align*}
 M_\tau=Z_\tau-\frac{\alpha(\tau-t)}{2(T-t)}
 \ge\frac{\alpha(T-\tau)}{2(T-t)}>0.
\end{align*}
Theorem~\ref{thm-MS} gives a measurable maximizer $\xi$ in $D$ for the maximization in the definition of $M_{\tau}$.
The collar condition \eqref{eq:cmp-polynomial-local-collar} implies $|\xi|<r_0$ on $A$; set $\xi=0$
on $A^c$.

Since $D=\overline B_R\subset O$ is compact,
we take a deterministic $\ell>0$ such that
$\overline B_{R+2\ell}\subset O$, and choose
\begin{align*}
 \chi\in C_c^\infty(B_{R+2\ell}),\qquad
 0\le\chi\le1,\qquad \chi\equiv1\ \text{on }B_{R+\ell}.
\end{align*}
Putting  
\begin{align*}
 \widehat\phi(s,x)=
 \begin{cases}
  \chi(x)\phi(s,x),&x\in O,\\
  0,&x\notin O,
 \end{cases}
 \qquad s\in[0,T],
\end{align*}
we have
\begin{align*}
 \mathfrak d_s\widehat\phi=\chi\,\mathfrak d_s\phi,\qquad
 \mathfrak d_\omega\widehat\phi=\chi\,\mathfrak d_\omega\phi.
\end{align*}
 Here both products on the right-hand side are
defined on $O$ and extended by zero to $\bR^d\setminus O$.
 Note that for $k=0,1,2$ and
$v\in C^k(\overline B_{R+2\ell})$, the spatial product rule gives
\begin{align*}
 \|\chi v\|_{C^k(\bR^d)}
 \le C_{\chi,k}\|v\|_{C^k(\overline B_{R+2\ell})}.
\end{align*}
 In this estimate, $\chi v$ denotes its extension
by zero outside $B_{R+2\ell}$.
Applying these bounds with
$k=0,2$ to $\phi$, with $k=0$ to $\mathfrak d_s\phi$, and with
$k=1$ to $\mathfrak d_\omega\phi$ proves
$\widehat\phi\in\mathscr C^2_\sF$. Moreover,
$\widehat\phi=\phi$ on $[0,T]\times B_{R+\ell}$, a neighborhood
of $[0,T]\times D$.
Lemma~\ref{lem-snell-envelope}(ii)
gives the Doob--Meyer decomposition
\[
 Z_s=\mathcal M_s-K_s,\qquad t\le s\le T,
\]
where $\mathcal M$ is a square-integrable continuous martingale,
$K$ is continuous and increasing, $K_t=0$, and $E K_T^2<\infty$.
Extend $\mathcal M$ to $[0,T]$ by
$\mathcal M_s=E_{\sF_s}\mathcal M_T$, and put
\[
 N_s=E_{\sF_s}[\mathcal M_T-K_\tau],\qquad 0\le s\le T.
\]
Take
\begin{align*}
 \psi(s,x)=\widehat\phi(s,x)
              -\frac{\alpha(s-t)^+}{2(T-t)}+N_s.
\end{align*}
 
 Since $\mathcal M_T-K_\tau\in L^2(\Omega)$,
the martingale representation theorem yields
\begin{align*}
 N_s=N_0+\int_0^s\zeta_r\,dW_r,
 \qquad E\int_0^T|\zeta_r|^2\,dr<\infty,
\end{align*}
and Doob's inequality gives $N\in\cS^2(\bR)$. 
Then it is straightforward to check that  $\psi\in\mathscr C^2_\sF$.

Since $K_\tau$ is $\sF_\tau$-measurable, for
$\tau\le s\le T$ we have a.s. on $A$,
\begin{align*}
 N_s=\mathcal M_s-K_\tau
     =Z_s+K_s-K_\tau\ge Z_s\ge Y_s.
\end{align*}
On $A$, the contact property gives $N_\tau=Z_\tau=Y_\tau$.
For any deterministic $0<r<R-r_0$, the ball $B_r(\xi)$ lies in
$D$ on $A$. Thus, it holds a.s. on $A$ that
\begin{align*}
 \inf_{x\in B_r(\xi)}(\psi-u)(s,x)
 \ge N_s-Y_s\ge0,\qquad \tau\le s\le T.
\end{align*}
 Equality holds at $(\tau,\xi)$ by the maximizing property of $\xi$.
Therefore, $\psi\in \underline{\cG}u(\tau,\xi;A)$ with $\hat\tau=T$.

On $D$, the scalar corrections leave
$D\phi$, $D^2\phi$, and $D\mathfrak d_\omega\phi$ unchanged.
Consequently, for almost every $s>t$, the residual is exactly
\begin{align*}
 R_\psi=R_\phi+\frac{\alpha}{2(T-t)}.
\end{align*}
Thus, for almost every $s>\tau$ on $A$,
\begin{align*}
 R_\psi\ge\frac{\kappa}{2(T-t)}>0.
\end{align*}
The common null set and Fubini's theorem give the same strict sign
in the pathwise essential limit at $(\tau,\xi)$. This contradicts
Definition~\ref{defn-viscosity}. Countable dense deterministic
points and continuity yield the desired comparison.
\end{proof}

We next construct a barrier for a possibly random polynomial
bound. For a one-sided comparison, choose $r\ge0$ such that either
\[
 \Gamma=\sup_{0\le t\le T,\,x\in\bR^d}
              \frac{u^+(t,x)}{ \varpi(x)^r}\in L^2(\sF_T)
 \quad\text{or}\quad
 \Gamma=\sup_{0\le t\le T,\,x\in\bR^d}
              {\frac{u^-(t,x)}{\varpi(x)^r}}\in L^2(\sF_T),
\]
as appropriate. Fix $a>\max\{r,p_0\}$ and put
$w_a(x)={ \varpi(x)}^a$ and $w_r(x)={ \varpi(x)}^r$.
The linear-growth bound in $(\cA1)$ implies that there are
deterministic $c_1,c_2\ge0$, uniform in the control, with
\[
 \tfrac12\sigma\sigma':D^2w_a
       +\beta\cdot Dw_a\le c_1w_a,
 \qquad |\sigma'Dw_a|\le c_2w_a.
\]
Indeed, $|Dw_a|\le C_aw_a/{ \varpi(x)}$ and
$|D^2w_a|\le C_aw_a/{ \varpi(x)}^2$.
Let $(B,Z^B)\in\cS^2(\bR)\times \cL^2(\bR^m)$ be the unique solution to BSDE (see \cite{El-Karoui-Peng-Quenez-2001}):
\begin{equation}\label{eq:polynomial-barrier-bsde}
 B_t=1+\Gamma+\int_t^T
       \big[(c_1+1)B_s+c_2|Z^B_s|\big]\,ds
                    -\int_t^TZ^B_s\,dW_s,\quad t\in [0,T].
\end{equation} 
 Put $c=c_1+1$, then It\^o's formula gives
\begin{equation*}
 d(e^{ct}B_t)=-c_2e^{ct}|Z^B_t|\,dt+e^{ct}Z^B_t\,dW_t.
\end{equation*}
Since $e^{c\cdot}Z^B\in\cL^2$, integrating from $t$ to $T$ and
taking conditional expectations yield
\[
 B_t=e^{c(T-t)}E_{\sF_t}(1+\Gamma)
     +c_2E_{\sF_t}\!\left[
          \int_t^T e^{c(s-t)}|Z^B_s|\,ds\right].
\]
Consequently, in view of $\Gamma\ge0$, we have
\[
 B_t\ge e^{c(T-t)}E_{\sF_t}(1+\Gamma)
     \ge E_{\sF_t}(1+\Gamma)\ge1.
\]
 Let $H_t=E_{\sF_t}\Gamma$ denote the continuous
version of this square-integrable martingale. For each deterministic $(t,x)$, in view of the 
definition of $\Gamma$ and the adaptedness of $u^+$ , we have
\begin{align*}
 u^+(t,x)=E_{\sF_t}[u^+(t,x)]
 \le H_tw_r(x)\le B_tw_r(x),\qquad\text{a.s.}
\end{align*}
The same calculation applies to $u^-$ for the other choice of $\Gamma$.
Taking a countable dense subset of $[0,T]\times\bR^d$ and using
continuity of $u^+$, $u^-$, $H$, and $B$ yield these inequalities
 for all $(t,x)$, a.s.

Set $P(t,x)=B_tw_a(x)$, by It\^o's formula we have
\begin{align*}
 \mathfrak d_tP(t,x)
 =-\big[(c_1+1)B_t+c_2|Z_t^B|\big]w_a(x),\qquad
 \mathfrak d_\omega P(t,x)=Z_t^Bw_a(x).
\end{align*}
For every $R>0$, boundedness of $(w_a,Dw_a,D^2w_a)$  
on $B_R$ gives
\begin{align*}
 &E\sup_{0\le t\le T}\|P(t,\cdot)\|_{C^2(B_R)}^2
 +E\int_0^T\left(
   \|\mathfrak d_tP(t,\cdot)\|_{C(B_R)}^2
   +\|\mathfrak d_\omega P(t,\cdot)\|_{C^1(B_R;\bR^m)}^2
                         \right)dt\\
 &\le C_R\left(E\sup_{0\le t\le T}|B_t|^2
                     +E\int_0^T|Z_t^B|^2\,dt\right)<\infty,
\end{align*}
where the constant $C_R$   depends only on $R,a,c_1,c_2,T$.
Thus, we have $P\in\mathscr C^2_\sF$, and  
\begin{align}\label{eq:polynomial-barrier}
 \mathscr L^vP
 ={}&-\big[(c_1+1)B+c_2|Z^B|\big]w_a
    +B\big(\tfrac12\sigma\sigma':D^2w_a
                 +\beta\cdot Dw_a\big)
    +(\sigma'Dw_a)\cdot Z^B\notag\\
 \le{}&-Bw_a\le0,\qquad v\in U.
\end{align}

\begin{thm}[Maximality under $(\cA1)$]\label{thm-max}
Assume $(\cA1)$. If $u\in\cS^2(C_{\mathrm{loc}}(\bR^d))$ is a
viscosity subsolution of \eqref{SHJB} and
$u^+\in\cS^2_{\mathrm{pol}}$, then a.s.,
\[
 u(t,x)\le V(t,x)\qquad\text{  for }
       (t,x)\in[0,T]\times\bR^d.
\]
In particular, $V$ is the maximal viscosity subsolution in
$\cS^2_{\mathrm{pol}}$.
\end{thm}

\begin{proof}
Fix a deterministic $(t_0,x_0)$ with $t_0<T$, a control
$\theta\in\cU$, and $\varepsilon>0$. We use the preceding barrier
for $u^+$.
 More precisely, choose $r\ge0$ such that
\begin{align*}
 \Gamma=\sup_{0\le t\le T,\,x\in\bR^d}
          \frac{u^+(t,x)}{\varpi(x)^r}\in L^2(\sF_T),
\end{align*}
fix $a>\max\{r,p_0\}$, and let $(B,Z^B)$ solve
\eqref{eq:polynomial-barrier-bsde}. Then
$P=Bw_a\in\mathscr C^2_\sF$ and
\[
 B_t\ge1,\qquad u^+(t,x)\le B_tw_r(x),\qquad
 \forall\,(t,x)\in[0,T]\times\bR^d,\quad\text{a.s.},
\]
while \eqref{eq:polynomial-barrier} gives
\begin{align*}
 \mathscr L^vP(t,x)\le-P(t,x),\qquad
 \forall\,(x,v)\in\bR^d\times U,\quad d\bP\otimes dt\text{-a.e.}
\end{align*}
Write $\mathscr L_N^v$ for the operator with
coefficients $\beta_N,\sigma_N$ from
Lemma~\ref{lem-poly-approximation}.
 That is, for $N\ge1$,
$\varphi\in\mathscr C^2_\sF$, and $v\in U$,
\begin{align*}
 \mathscr L_N^v\varphi(t,x)
 &=\mathfrak d_t\varphi(t,x)
       +D\varphi(t,x)\beta_N(t,x,v)
 +\tfrac12\operatorname{tr}\!\left(
       \sigma_N\sigma_N'(t,x,v)D^2\varphi(t,x)\right)
       \\
 &\quad\,
 +\operatorname{tr}\!\left(
       \sigma_N(t,x,v)D\mathfrak d_\omega\varphi(t,x)\right),
\end{align*}
where $\beta_N(t,x,v)=\beta(t,\pi_Nx,v)$ and
$\sigma_N(t,x,v)=\sigma(t,\pi_Nx,v)$, for
$(t,x)\in[0,T]\times\bR^d$.
For fixed $N,h$, apply Lemma~\ref{lem-upper-comparison-fields}
to the bounded data
$(\beta_N,\sigma_N,f_{N,h},G_{N,h})$, starting at $(t_0,x_0)$.
By Lemma~\ref{lem-upper-comparison-fields}, there exist deterministic constants 
$\Lambda_N$ and $C_{N,h,t_0}>0$,
independent of $0<\delta\le1$, and fields
$\Phi_{N,h,\delta}\in\mathscr C^2_\sF$ such that
\begin{align*}
  \mathscr L_N^{\theta_s}\Phi_{N,h,\delta}(s,x)
                       +f_{N,h}(s,x,\theta_s)&\le0,
       &&\forall\, x\in\bR^d,\quad d\bP\otimes ds\text{-a.e.},\\
 \Phi_{N,h,\delta}(T,x)&\ge G_{N,h}(x),
       &&\forall\,x\in\bR^d,\quad\text{a.s.},\\
 \Phi_{N,h,\delta}(s,x)&\ge-\Lambda_N,
       &&\forall\,(s,x)\in[0,T]\times\bR^d,\quad\text{a.s.},\\
 \Phi_{N,h,\delta}(t_0,x_0)&=J_{N,h}(t_0,x_0;\theta)
                     + C_{N,h,t_0} \frac{\delta}{1+\delta}.
\end{align*}
Here, in view of the construction of $\Phi_{N,h,\delta}$ in \eqref{eq:max-upper-field}, 
the constant $\Lambda_N$ is from the lower bound of the scalar BSDE solution $Y$ in \eqref{eq:max-scalar-bsde} for the scalar cost, 
from which we may get
  \begin{equation}\label{eq:polynomial-upper-bound}
  \Lambda_N=C_0{ \varpi(N)}^{p_0}
\end{equation}
  with the constant $C_0$
    independent of $h,\delta$. 
    Indeed, projection and
convolution with the nonnegative unit-mass kernel $\rho_h$ give
\[
 \max\{|f_{N,h}(s,x,v)|,|G_{N,h}(x)|\}
       \le C{ \varpi(N)}^{p_0}.
\]
Taking conditional expectations in \eqref{eq:max-scalar-bsde}
  bounds the cost functional from below by
$-C(1+T){ \varpi(N)}^{p_0}$, which gives the asserted bound. 

Let $\eta_{N,h}$ be the nonnegative scalar correction in
Lemma~\ref{lem-poly-approximation}, with
$d\eta_{N,h}=-a_{N,h}\,dt+z_{N,h}\,dW$ and
$\eta_{N,h}(T)=b_{N,h}$. Put
\[
 \phi=\Phi_{N,h,\delta}+\eta_{N,h}+\varepsilon P,
 \qquad D_N=\overline B_{N/2}.
\]
On $D_N$ the original and projected dynamics agree, while
$f_N=f$ and $G_N=G$. Hence
\begin{equation}\label{eq:nonnegative-phi-residual}
 \mathscr L^{\theta_s}\phi+f(s,x,\theta_s)
 \le f(s,x,\theta_s)-f_{N,h}(s,x,\theta_s)-a_{N,h}(s)
                   +\varepsilon\mathscr L^{\theta_s}P\le0.
\end{equation}
  Likewise,
$\phi(T)\ge G_{N,h}+b_{N,h}\ge G$ on $D_N$.
Since $u\le Bw_r$, $B\ge1$, and \eqref{eq:polynomial-upper-bound}
holds, on the collar $N/3\le|x|\le N/2$ we have
\begin{equation}\label{eq:polynomial-collar}
 u-\phi\le B\big[w_r(x)+\Lambda_N-\varepsilon w_a(x)\big]<0
\end{equation}
for all sufficiently large deterministic $N$. Indeed, uniformly
for $N/3\le|x|\le N/2$,
\begin{align*}
 \frac{w_r(x)+\Lambda_N}{w_a(x)}
 \le { \varpi(N/3)}^{r-a}
      +C_0{ \varpi(N)}^{p_0}{ \varpi(N/3)}^{-a}
 \longrightarrow0, \text{ as }N\to\infty,
\end{align*}
since $a>\max\{r,p_0\}$. Thus the choice of $N$ is independent of
$\omega,h,\delta$. Also, enlarge $N$  so as to have $|x_0|<N/3$.
The residual of $\phi$ as defined in \eqref{eq:max-residual} is nonnegative by \eqref{eq:nonnegative-phi-residual}.
Lemma~\ref{lem-cmp-polynomial-local} therefore gives 
\[
 u(t_0,x_0)\le J_{N,h}(t_0,x_0;\theta)
       +{ C_{N,h,t_0}}\frac{\delta}{1+\delta}
       +\eta_{N,h}(t_0)+\varepsilon P(t_0,x_0),\quad \text{a.s.}
\]
First let $\delta\downarrow0$ for fixed $N,h$.
Using $|J_{N,h}-J_N|\le\eta_{N,h}$ gives
\[
 u(t_0,x_0)\le J_N(t_0,x_0;\theta)
       +2\eta_{N,h}(t_0)+\varepsilon P(t_0,x_0),\quad \text{a.s.}
\]
For fixed $N$, next let $h\downarrow0$ along an   a.s.
convergent subsequence of $\eta_{N,h}\to0$ in $\cS^2$ to obtain
\[
 u(t_0,x_0)\le J_N(t_0,x_0;\theta)+\varepsilon P(t_0,x_0),\quad \text{a.s.}
\]
Finally letting $N\uparrow\infty$ and using the cost
convergence in Lemma~\ref{lem-poly-approximation}, we obtain
\[
 u(t_0,x_0)\le J(t_0,x_0;\theta)+\varepsilon P(t_0,x_0),\quad \text{a.s.}
\]
Letting $\varepsilon\downarrow0$ shows that $u(t_0,x_0)$ is an
essential lower bound for all the control costs, so it is at most
$V(t_0,x_0)$.  
The terminal inequality and joint continuity extend the result from a
countable dense set to all $(t,x)$ a.s.
\end{proof}

\begin{thm}[Uniqueness under $(\cA1)$]\label{thm-uniqueness-A1}
Assume $(\cA1)$. Every viscosity supersolution
$u\in\cS^2(C_{\mathrm{loc}}(\bR^d))$ with
$u^-\in\cS^2_{\mathrm{pol}}$ satisfies
\[
 u(t,x)\ge V(t,x)\qquad\text{a.s., simultaneously for }
        (t,x)\in[0,T]\times\bR^d.
\]
Consequently, the value function is the unique viscosity solution
of \eqref{SHJB} in $\cS^2_{\mathrm{pol}}$.
\end{thm}

\begin{proof} 
Fix deterministic $(t_0,x_0)\in[0,T)\times \bR^d$ and $\varepsilon>0$, and construct $P$ using the
polynomial bound on $u^-$. 
 More precisely, choose $r\ge0$ such that
\[
 \Gamma=\sup_{0\le t\le T,\,x\in\bR^d}
           \frac{u^-(t,x)}{\varpi(x)^r}\in L^2(\sF_T),
\]
fix $a>\max\{r,p_0\}$, and let $(B,Z^B)$ solve
\eqref{eq:polynomial-barrier-bsde} with this $\Gamma$.
Set $P(t,x)=B_tw_a(x)$. The preceding construction gives
$P\in\mathscr C^2_\sF$ with
\begin{equation*}
 B_t\ge1,\qquad u^-(t,x)\le B_tw_r(x),\qquad
 u(t,x)\ge-B_tw_r(x),
\end{equation*}
  for all $(t,x)\in[0,T]\times\bR^d$ a.s.
Moreover, $P\ge0$ and \eqref{eq:polynomial-barrier} gives
\[
 \mathscr L^vP(t,x)\le-P(t,x),\qquad
 \forall\,(x,v)\in\bR^d\times U,\quad d\bP\otimes dt\text{-a.e.}
\]
The choices of $r,a$, and hence $P$, are fixed independently of
$\varepsilon,N,h,\delta,n$ where the parameters $N,h,\delta,n$ 
are introduced and used below.

For fixed integer $N>0$, and real $h\in (0,1)$,
 let $0<\delta\le1$ and denote by
$w_{N,h}^{\delta}$ the shaken value in
Lemma~\ref{lem-uniq-A1-weak} for the data
$(\beta_N,\sigma_N,f_{N,h},G_{N,h})$. Define
\[
 q_{N,h,\delta}(t,x)
 =\int_{\bR^d}\rho_\delta(y)w_{N,h}^{\delta}(t,x-y)\,dy,
 \qquad (t,x)\in[0,T]\times\bR^d.
\]
Here $\rho_\delta\ge0$, $\int\rho_\delta=1$, and
$\operatorname{supp}\rho_\delta\Subset B_{\delta/4}$, as in
\eqref{eq:uniq-A1-q}. To specify the spatial domains, set
\begin{align*}
 D_N=\overline B_{N/2},\qquad
 O_{N,\delta}=B_{N/2+\delta/4},\qquad
 D_{1,N,\delta}=B_{N/2+\delta/2}.
\end{align*}   
The convolution kernels are supported in $D_{1,N,\delta}$ because
$\overline{O_{N,\delta}}-\operatorname{supp}\rho_\delta
 \Subset D_{1,N,\delta}$.
Apply Lemma~\ref{lem-uniq-A1-weak} on $D_{1,N,\delta}$,
Lemma~\ref{lem-uniq-A1-q} with $x_*=0$ and $R=N/2$, and then
Lemma~\ref{lem-uniq-A1-regularize} on $D_N$ with neighborhood
$O_{N,\delta}$. The resulting fields $q_{N,h,\delta,n}$ satisfy,
for every open ball $B\Subset O_{N,\delta}$,
\begin{align*}
 q_{N,h,\delta,n}
 &\in\cS^2(C(\overline B))\cap\cL^2(C^2(\overline B)), \quad
 \mathfrak d_tq_{N,h,\delta,n}
 \in\cL^2(C(\overline B)),\qquad
 \mathfrak d_\omega q_{N,h,\delta,n}
 \in\cL^2(C^1(\overline B;\bR^m)),
\end{align*}
and the stochastic equation in assertion (ii) of
Lemma~\ref{lem-uniq-A1-regularize} on $[0,T]\times O_{N,\delta}$.
They satisfy
\begin{align}\label{eq:polynomial-lower-fields}
 \mathscr L_N^v q_{N,h,\delta,n} (t,x)+f_{N,h}(t,x,v)&\ge0,
                  \quad\forall\,(x,v)\in D_N\times U,
                     \quad d\bP\otimes dt\text{-a.e.},\notag\\
 q_{N,h,\delta,n}(t,x)&\le q_{N,h,\delta}(t,x)\le \Lambda_N, \quad
  \forall\,(t,x)\in[0,T]\times D_N,\quad\text{a.s.},
 \notag\\
 q_{N,h,\delta,n}(T,x)&\le(G_{N,h}*\rho_\delta)(x), \quad \forall\, x\in D_N,\quad a.s.
\end{align}
 We spell out the bounds.
The global cost bounds
\begin{align*}
 |f_{N,h}(s,z,v)|\le C\varpi(N)^{p_0},\qquad
 |G_{N,h}(z)|\le C\varpi(N)^{p_0},
\end{align*}
remain valid at the shaken arguments. Taking conditional expectations
and then the essential infimum in the definition of $w_{N,h}^{\delta}$
therefore gives
\[
 |w_{N,h}^{\delta}(t,x)|
 \le C(1+T-t)\varpi(N)^{p_0}.
\]
Consequently, enlarging the deterministic $C_0$ to satisfy
$C_0\ge C(1+T)$ if necessary, we obtain
\begin{align*}
 |q_{N,h,\delta}(t,x)|
 &\le\int\rho_\delta(y)|w_{N,h}^{\delta}(t,x-y)|\,dy
 \le C(1+T)\varpi(N)^{p_0}
 \le\Lambda_N:=C_0\varpi(N)^{p_0}.
\end{align*}
These bounds hold  for all $(t,x)\in[0,T]\times\bR^d$
a.s., by continuity. Here $C_0$ is independent of $N,h,\delta,n$.
  
 For fixed $N,h,\delta$,
Lemma~\ref{lem-uniq-A1-regularize}(iii) gives
$q_{N,h,\delta,n}\le q_{N,h,\delta}$ on $[0,T]\times D_N$ a.s.
and the convergence as $n\to\infty$ below. Together with
$q_{N,h,\delta}(T,\cdot)=G_{N,h}*\rho_\delta$, this order also gives
the terminal inequality in \eqref{eq:polynomial-lower-fields}.
The convergence as $\delta\downarrow0$ follows from
\eqref{eq:uniq-A1-shaking-stability}, applied to the fixed bounded
data $(\beta_N,\sigma_N,f_{N,h},G_{N,h})$; its constant may depend
on $N,h$.
Consequently, for fixed $N,h$,
\begin{align*}
 q_{N,h,\delta,n}&\longrightarrow q_{N,h,\delta}
       &&\text{in }\cS^2(C^2(D_N))\quad(n\to\infty),\\
 q_{N,h,\delta}&\longrightarrow V_{N,h}
       && \text{in }\cS^2(C(D_N)) \quad(\delta\downarrow0).
\end{align*}
Only these approximating values require the bounded-data spatial
Lipschitz estimates.

Choose a deterministic spatial Lipschitz constant $\ell_{N,h}$ of
$G_{N,h}$ and put
\[
 \phi=q_{N,h,\delta,n}-\ell_{N,h}\delta
                         -\eta_{N,h}-\varepsilon P.
\]
Since $\rho_\delta$ is nonnegative, has integral one, and is supported
in $B_\delta$,
\[
 (G_{N,h}*\rho_\delta)(x)-G_{N,h}(x)
 \le\ell_{N,h}\int |y|\rho_\delta(y)\,dy
 \le\ell_{N,h}\delta.
\]
Using $\eta_{N,h}(T)=b_{N,h}$ and $P\ge0$, the terminal estimate
in \eqref{eq:polynomial-lower-fields} therefore gives
\begin{align*}
 \phi(T,x)\le G_{N,h}(x)-b_{N,h}\le G_N(x)=G(x),
 \qquad \forall\, x\in D_N,\quad \text{a.s.}
\end{align*}
Furthermore, $d\bP\otimes dt$-a.e., on $D_N$, for every $v\in U$,
\[
 \mathscr L^v\phi+f(t,x,v)
 \ge f(t,x,v)-f_{N,h}(t,x,v)+a_{N,h}(t)
                              -\varepsilon\mathscr L^vP\ge0.
\]
All these inequalities hold on common null sets, as required by
the preceding lemmas. Since $u\ge-Bw_r$, the nonnegative
corrections and \eqref{eq:polynomial-lower-fields} yield a.s.
\[
 \phi-u\le B\big[\Lambda_N+w_r(x)-\varepsilon w_a(x)\big]<0
 \quad\text{for }N/3\le|x|\le N/2, \quad \forall\,t\in[0,T],
\]
for every sufficiently large deterministic $N$, independently of
$h,\delta,n$. Also choosse $N$ to be so large that $|x_0|<N/3$.
We have $R_\phi\le0$ because the differential inequality holds
for every control value. Lemma~\ref{lem-cmp-polynomial-local}
therefore gives
\[
 u(t_0,x_0)\ge q_{N,h,\delta,n}(t_0,x_0)
       -\ell_{N,h}\delta-\eta_{N,h}(t_0)
       -\varepsilon P(t_0,x_0),\quad \text{a.s.}
\]
 Let $n\to\infty$ along an almost surely
convergent subsequence of the stated $\cS^2$ convergence, keeping
$N,h,\delta$ fixed. Next, for fixed $N,h$, choose a deterministic
sequence $\delta_k\downarrow0$ along which
$q_{N,h,\delta_k}\to V_{N,h}$ almost surely in $C([0,T]\times D_N)$.
Such a sequence exists by the stated $\cS^2$ convergence.
Intersecting the probability-one events for these countably many
$\delta_k$, and then letting $k\to\infty$ with
$\ell_{N,h}\delta_k\to0$, gives
\[
 u(t_0,x_0)\ge V_{N,h}(t_0,x_0)-\eta_{N,h}(t_0)
                                      -\varepsilon P(t_0,x_0),\quad \text{a.s.}
\]
Using $|V_{N,h}-V_N|\le\eta_{N,h}$ gives
\[
 u(t_0,x_0)\ge V_N(t_0,x_0)-2\eta_{N,h}(t_0)
                                      -\varepsilon P(t_0,x_0), \quad \text{a.s.}
\]
For fixed $N$, leting $h\downarrow0$ along an   a.s. convergent
subsequence of $\eta_{N,h}\to0$ in $\cS^2(\bR)$ yields
\begin{align*}
 u(t_0,x_0)\ge V_N(t_0,x_0)-\varepsilon P(t_0,x_0),\quad \text{a.s.}
\end{align*}
Letting $N\uparrow\infty$ using
Lemma~\ref{lem-poly-approximation}, and finally letting
$\varepsilon\downarrow0$, we conclude that 
$u(t_0,x_0)\ge V(t_0,x_0)$ a.s. Continuity and a countable dense set
give the desired comparison. For a viscosity solution in
$\cS^2_{\mathrm{pol}}$, Theorem~\ref{thm-max} gives the reverse
inequality. Theorem~\ref{thm-existence} and the polynomial value
bound place $V$ in this class, proving the existence and uniqueness.
\end{proof}
}

\begin{appendix}
\section{Measurable selection theorem}
The following measurable selection theorem is referred to \cite{wagner1977survey}.
\begin{thm}\label{thm-MS}
Let $(\Lambda,\mathscr M)$ be a measurable space equipped with a nonnegative measure $\mu$ and let $(\cO,\cB(\cO))$ be a Polish space. Suppose $F$ is a set-valued function from $\Lambda$ to $\cB(\cO)$ satisfying: (i) for $\mu$-a.e. $\lambda \in\Lambda$, $F(\lambda)$ is a closed nonempty subset of $\cO$; (ii) for any open set $O\subset \cO$, $\{\lambda:\, F(\lambda)\cap O\neq \emptyset\}\in \mathscr M$. Then there exists a  measurable function $f$: $(\Lambda,\mathscr M)\rightarrow (\cO,\cB(\cO))$ such that for $\mu$-a.e. $\lambda\in\Lambda$, $f(\lambda)\in F(\lambda)$. 
\end{thm}

{
In the proof of Theorem \ref{thm-existence}, the selector $a(s,x)$ is
constructed from a fixed countable dense subset $\{v_j:j\geq1\}$ of $U$.
With $Q_\phi(s,x,v)=\mathscr L^v\phi(s,x)+f(s,x,v)$, one takes the
first $v_j$ satisfying
\[
Q_\phi(s,x,v_j)\leq\inf_{i\geq1}Q_\phi(s,x,v_i)+\eps.
\]
The countable infimum and the first-index selection are
$\sP\otimes\cB(\bR^d)$-measurable. Continuity in $v$ identifies this
infimum with the infimum over $U$; on the common exceptional set one
assigns the fixed value $v_1$. Evaluating $a(s,\cdot)$ at the lagged
state $X^n_{t_k}$ on each interval $(t_k,t_{k+1}]$ gives the
progressively measurable controls $\theta^n$ used in that proof.
}

{
In the proof of Lemma~\ref{lem-cmp-polynomial-local}, apply the
selection theorem directly at the stopping time $\tau$, with
$(\Lambda,\mathscr M,\mu)=(A,\sF_\tau|_A,\bP|_A)$ and
$\cO=D$. The gap $(u-\phi)(\tau,x)$, or $(\phi-u)(\tau,x)$
for the supersolution assertion, is measurable in $\omega$ and
continuous in $x$. Its maximum is attained on the compact ball $D$.
The argmax multifunction is nonempty and closed; its measurability
follows using countable dense subsets of compact subsets of $D$.
Theorem~\ref{thm-MS} therefore gives an $\sF_\tau$-measurable
maximizer $\xi$ on $A$. Its positive gap and the collar condition
place it in $B_{r_0}$; extend it by zero off $A$.
}

\section{Proof of Proposition~\ref{prop-classical-viscosity}}
\begin{proof}[Proof of Proposition~\ref{prop-classical-viscosity}]
The terminal condition and the required local $\cS^2$ regularity
are part of Definition~\ref{defn-classical}. For $F=u$ or a test
field $F=\phi$, write
\begin{align}
 Q_F(s,x,v)={}&\mathfrak d_sF(s,x)
       +\beta(s,x,v)\cdot DF(s,x)+f(s,x,v)
 +\tfrac12\operatorname{tr}\!\left[
       \sigma\sigma'(s,x,v)D^2F(s,x)\right]
        \nonumber\\
  &
 +\operatorname{tr}\!\left[
       \sigma(s,x,v)D\mathfrak d_\omega F(s,x)\right]{.}
\end{align}
Equation~\eqref{eq:classical-equation} says that $d\bP\otimes ds$-a.e.,
\begin{equation}\label{eq:classical-Q-zero}
 \inf_{v\in U}Q_u(s,x,v)=0,
 \qquad Q_u(s,x,v)\ge0\quad(v\in U), \quad\forall x\in\bR^d.
\end{equation}
 
We use the exclusion formulation in
Definition~\ref{defn-viscosity}. Suppose first that an upper-contact
test $\phi$ violates the subsolution condition. On an event
$A\in\sF_\tau$ of positive probability, the contact holds until a
stopping time $\hat\tau>\tau$, while, for some $r,\eps>0$ and a
stopping time $\rho>\tau$ on $A$,
\[
 R_\phi(s,x)\ge2\eps
 \quad\text{for a.e. }s\in(\tau,\rho),\quad x\in B_r(\xi).
\]
Reduce $A$ so that $|\xi|\le N$ and $\tau\le T-b_0$ there,
for deterministic $N,b_0>0$, and choose $r$ smaller than the
contact radius. Extend $\xi$ by zero outside $A$; all contact and
residual inequalities below are used on $A$. For any control $\theta$, start
$X^\theta=X^{\tau,\xi;\theta}$ and stop it at
\[
 \eta_h^\theta=(\tau+h)\wedge T\wedge\hat\tau\wedge\rho
       \wedge\inf\{s\ge\tau:|X_s^\theta-\xi|\ge r/2\}.
\]
The fourth-moment increment bound in Lemma~\ref{lem-SDE}, Chebyshev's inequality, and
$\hat\tau\wedge\rho>\tau$ on $A$ allow one fixed
$0<h<b_0\wedge1$ and $m>0$ such that
\begin{equation}\label{eq:classical-duration}
 \inf_{\theta\in\cU}E[1_A(\eta_h^\theta-\tau)]\ge m.
\end{equation}
This is the estimate proved in \eqref{eq-existence-duration}; its
constants depend on $N,r,A$, but not on $\theta$.

Put $F=\phi-u$. On $A$, the stopped state stays in a deterministic ball.
Spatial cutoff and the random-initial-time version of
Lemma~\ref{lem-ito-wentzell} therefore apply to both $\phi$ and
$u$ using only their local test-function norms. The stochastic
integral multiplied by $1_A$ is square integrable after stopping. Since
$F(\tau,\xi)=0$ on $A$, for every $\theta$,
\begin{equation}\label{eq:classical-contact-Ito}
 E[1_AF(\eta_h^\theta,X_{\eta_h^\theta}^\theta)]
 =E\left[1_A\int_\tau^{\eta_h^\theta}
       (Q_\phi-Q_u)(s,X_s^\theta,\theta_s)\,ds\right].
\end{equation}
The random-start formula follows by the finite-valued stopping-time
approximation used in the proof of Theorem~\ref{thm-existence}.
At upper contact the left side is nonnegative. By
\eqref{eq:classical-Q-zero}, it follows that
\begin{equation}\label{eq:classical-upper-bound}
 0\le E\left[1_A\int_\tau^{\eta_h^\theta}
                      Q_\phi(s,X_s^\theta,\theta_s)\,ds\right]
 \quad\text{for every }\theta.
\end{equation}

Here is the control-selection estimate used for both contact
directions. For either $Q_\phi$ or $Q_u$, choose a fixed countable
dense set $\{v_j\}\subset U$ and, at each $(s,x)$, take the first
$v_j$ whose value of $Q$ is within $\eps$ of $\inf_{v\in U}Q$.
{Continuity in $v$ and separability of $U$ make this a
predictable selector $a_Q(s,x)$; set $a_Q=v_1$ on the common
exceptional set. For $n\ge1$, put $t_k=(\tau+kh/n)\wedge T$,
$0\le k\le n$, and construct $\theta^n$ and
$X^n=X^{\tau,\xi;\theta^n}$ successively by
\[
 \theta_s^n=a_Q(s,X^n_{t_k}),\qquad Y_s^n=X^n_{t_k},
 \qquad t_k<s\le t_{k+1},\quad 0\le k<n.
\]
Set $Y_\tau^n=\xi$ and use $v_1$ for the control outside
$(\tau,t_n]$. At each step, $X^n_{t_k}$ is
$\sF_{t_k}$-measurable, so the prescribed control is predictable
and Lemma~\ref{lem-SDE}(i) gives the state on the next interval.
For $\Delta_n:=\sup_{\tau\le s\le(\tau+h)\wedge T}|X_s^n-Y_s^n|$,
the fourth-moment estimates in Lemma~\ref{lem-SDE}(ii) give
\[
 E\Delta_n^4
 \le C(h/n)^2\sum_{k=0}^{n-1}E\varpi(X^n_{t_k})^4
 \le Ch^2/n\longrightarrow0, \quad \text{as }n\to\infty.
\]
Fix $D=\overline B_{N+r}$ strictly inside a larger deterministic
ball. On $A$, before $\eta_h^{\theta^n}$, both $X_s^n$ and
$Y_s^n$ lie in $B_{r/2}(\xi)\subset D$.}
Local test-function norms on the larger ball and
$(\cA1)$ give
\[
\begin{aligned}
 K_Q(s)&:=\sup_{\substack{x\in D\\v\in U}}|Q(s,x,v)|\in\cL^2,\\
 \omega_Q(s,a)&:=\sup_{\substack{x,y\in D,\ |x-y|\le a\\v\in U}}
           |Q(s,x,v)-Q(s,y,v)|\longrightarrow0
           \quad\text{a.e. as }a\downarrow0,
\end{aligned}
\]
with $\omega_Q(s,a)\le2K_Q(s)$. Splitting according to
$\{\Delta_n\le a\}$ and then letting
$n\to\infty$ and $a\downarrow0$ yields
\begin{equation}\label{eq:classical-frozen-error}
 E\left[1_A\int_\tau^{\eta_h^{\theta^n}}
 |Q(s,X_s^n,\theta_s^n)-Q(s,Y_s^n,\theta_s^n)|\,ds\right]
 \longrightarrow0.
\end{equation}
This is the same modulus argument as in the proof of
Theorem~\ref{thm-existence}; it uses no regularity of the
selector as a function of $x$.

Since $R_\phi=-\inf_vQ_\phi\ge2\eps$ on the stopped cylinder,
apply this construction to $Q_\phi$. Then
$Q_\phi(s,Y_s^n,\theta_s^n)\le-\eps$ before the stop, and
 \eqref{eq:classical-duration},
\eqref{eq:classical-upper-bound} and \eqref{eq:classical-frozen-error}
yield
\[
 0\le E\left[1_A\int_\tau^{\eta_h^{\theta^n}}
       Q_\phi(s,X_s^n,\theta_s^n)\,ds\right]
 \le-\eps m+o(1).
\]
This is impossible for large $n$. Thus $u$ is a viscosity subsolution.

For a lower-contact test, suppose instead that
$R_\phi\le-2\eps$ on such a stopped cylinder. Then
$Q_\phi(s,x,v)\ge2\eps$ for every $v$ there.
The same localization gives \eqref{eq:classical-duration}; now
\eqref{eq:classical-contact-Ito} has a nonpositive left side.
Select $\theta^n$ using $Q_u$, whose infimum is zero by
\eqref{eq:classical-Q-zero}. Then
$Q_u(s,Y_s^n,\theta_s^n)\le\eps$, and
\eqref{eq:classical-frozen-error} gives
\[
 2\eps E[1_A(\eta_h^{\theta^n}-\tau)]
 \le E\left[1_A\int_\tau^{\eta_h^{\theta^n}}Q_\phi\,ds\right]
 \le E\left[1_A\int_\tau^{\eta_h^{\theta^n}}Q_u\,ds\right]
 \le\eps E[1_A(\eta_h^{\theta^n}-\tau)]+o(1),
\]
contradicting \eqref{eq:classical-duration}. This proves the
supersolution condition.
\end{proof}
 
\end{appendix}

\bibliographystyle{siam}

\begin{thebibliography}{10}

\bibitem{bank2026duality}
{\sc P.~Bank, J.~R. Dause, F.~de~Feo, and P.~K. Friz}, {\em Duality for
  stochastic control with non-markovian random coefficients}, arXiv preprint
  arXiv:2609.05101,  (2026).

\bibitem{bauerle2020portfolio}
{\sc N.~B\"auerle and S.~Desmettre}, {\em Portfolio optimization in fractional
  and rough heston models}, SIAM Journal on Financial Mathematics, 11 (2020),
  pp.~240--273.

\bibitem{bayer2023rough}
{\sc C.~Bayer, P.~K. Friz, M.~Fukasawa, J.~Gatheral, A.~Jacquier, and
  M.~Rosenbaum}, {\em Rough volatility}, SIAM, 2023.

\bibitem{bender2016-first-order}
{\sc C.~Bender and N.~Dokuchaev}, {\em A first-order {BSPDE} for swing option
  pricing}, Mathematical Finance, 26 (2016), pp.~461--491.

\bibitem{bensoussan2013mean}
{\sc A.~Bensoussan, J.~Frehse, P.~Yam, et~al.}, {\em Mean field games and mean
  field type control theory}, vol.~101, Springer, 2013.

\bibitem{bouchard-nutz2016-target}
{\sc B.~Bouchard and M.~Nutz}, {\em Stochastic target games and dynamic
  programming via regularized viscosity solutions}, Mathematics of Operations
  Research, 41 (2016), pp.~109--124.

\bibitem{buckdahn2015pathwise}
{\sc R.~Buckdahn, C.~Keller, J.~Ma, and J.~Zhang}, {\em Pathwise viscosity
  solutions of stochastic {PDEs} and forward path-dependent {PDEs}---a rough
  path view}, arXiv:1501.06978,  (2015).

\bibitem{buckdahn2007pathwise}
{\sc R.~Buckdahn and J.~Ma}, {\em Pathwise stochastic control problems and
  stochastic {HJB} equations}, SIAM J. Control Optim., 45 (2007),
  pp.~2224--2256.

\bibitem{cardaliaguet2015master}
{\sc P.~Cardaliaguet, F.~Delarue, J.-M. Lasry, and P.-L. Lions}, {\em The
  master equation and the convergence problem in mean field games},
  arXiv:1509.02505,  (2015).

\bibitem{cardaliaguet2025mean}
{\sc P.~Cardaliaguet, B.~Seeger, and P.~Souganidis}, {\em Mean field games with
  common noise and degenerate idiosyncratic noise}, The Annals of Applied
  Probability, 35 (2025), pp.~1531--1569.

\bibitem{cont2013functional}
{\sc R.~Cont and D.-A. Fourni{\'e}}, {\em Functional {It{\^o}} calculus and
  stochastic integral representation of martingales}, Ann. Probab., 41 (2013),
  pp.~109--133.

\bibitem{crandall1992user}
{\sc M.~G. Crandall, H.~Ishii, and P.-L. Lions}, {\em User's guide to viscosity
  solutions of second order partial differential equations}, Bull. Amer. Math.
  Soc., 27 (1992), pp.~1--67.

\bibitem{crandall2000lp}
{\sc M.~G. Crandall, M.~Kocan, and A.~{\'S}wiech}, {\em L$^p$-theory for fully
  nonlinear uniformly parabolic equations}, Commun. Partial Differ. Equ., 25
  (2000), pp.~1997--2053.

\bibitem{DiestelUhl1977}
{\sc J.~Diestel and J.~J. Uhl, Jr.}, {\em Vector Measures}, vol.~15 of
  Mathematical Surveys, American Mathematical Society, Providence, R.I., 1977.

\bibitem{DuQiuTang10}
{\sc K.~Du, J.~Qiu, and S.~Tang}, {\em $\textrm{L}^p$ theory for
  super-parabolic backward stochastic partial differential equations in the
  whole space}, Appl. Math. Optim., 65 (2011), pp.~175--219.

\bibitem{DuTangZhang-2013}
{\sc K.~Du, S.~Tang, and Q.~Zhang}, {\em $\textrm{W}^{m,p}$-solution ($p\ge 2$)
  of linear degenerate backward stochastic partial differential equations in
  the whole space}, J. Differ. Equ., 254 (2013), pp.~2877--2904.

\bibitem{Du_Zhang_DegSemilin2012}
{\sc K.~Du and Q.~Zhang}, {\em Semi-linear degenerate backward stochastic
  partial differential equations and associated forward-backward stochastic
  differential equations}, Stoch. Process. Appl., 123 (2013), pp.~1616--1637.

\bibitem{ekren2014viscosity}
{\sc I.~Ekren, C.~Keller, N.~Touzi, and J.~Zhang}, {\em On viscosity solutions
  of path dependent {PDEs}}, Ann. Probab., 42 (2014), pp.~204--236.

\bibitem{ekren2016viscosity-1}
{\sc I.~Ekren, N.~Touzi, and J.~Zhang}, {\em Viscosity solutions of fully
  nonlinear parabolic path dependent {PDEs: Part I}}, Ann. Probab., 44 (2016),
  pp.~1212--1253.

\bibitem{ekren2016pseudo}
{\sc I.~Ekren and J.~Zhang}, {\em {Pseudo-Markovian} viscosity solutions of
  fully nonlinear degenerate {PPDEs}}, Probability, Uncertainty and
  Quantitative Risk, 1 (2016), p.~6.

\bibitem{EnglezosKaratzas09}
{\sc N.~Englezos and I.~Karatzas}, {\em Utility maximization with habit
  formation: Dynamic programming and stochastic \textrm{PDEs}}, \textrm{SIAM}
  J. Control Optim., 48 (2009), pp.~481--520.

\bibitem{El_Karoui-reflec-1997}
{\sc N.~E\textrm{l Karoui}, C.~Kapoudjian, E.~{Pardoux}, S.~Peng, and M.~C.
  Quenez}, {\em Reflected solutions of backward \textrm{SDE's}, and related
  obstacle problems for \textrm{PDE's}}, Ann. Probab., 25 (1997), pp.~702--737.

\bibitem{Karoui_Peng_Quenez}
{\sc N.~E\textrm{l Karoui}, S.~Peng, and M.~C. Quenez}, {\em Backward
  stochastic differential equations in finance}, Math. Finance, 7 (1997),
  pp.~1--71.

\bibitem{El-Karoui-Peng-Quenez-2001}
\leavevmode\vrule height 2pt depth -1.6pt width 23pt, {\em A dynamic maximum
  principle for the optimization of recursive utilities under constraints},
  Ann. Appl. Probab., 11 (2001), pp.~664--693.

\bibitem{Hu_Ma_Yong02}
{\sc Y.~Hu, J.~Ma, and J.~Yong}, {\em On semi-linear degenerate backward
  stochastic partial differential equations}, Probab. Theory Relat. Fields, 123
  (2002), pp.~381--411.

\bibitem{juutinen2001definition}
{\sc P.~Juutinen}, {\em On the definition of viscosity solutions for parabolic
  equations}, Proceedings of the American Mathematical Society, 129 (2001),
  pp.~2907--2911.

\bibitem{krylov1999-policies}
{\sc N.~V. Krylov}, {\em Approximating value functions for controlled
  degenerate diffusion processes by using piece-wise constant policies},
  Electron. J. Probab., 4 (1999), pp.~1--19.

\bibitem{Krylov_09}
{\sc N.~V. Krylov}, {\em On the {It\^o}-{W}entzell formula for
  distribution-valued processes and related topics}, Probab. Theory Relat.
  Fields, 150 (2011), pp.~295--319.

\bibitem{kunita1981some}
{\sc H.~Kunita}, {\em Some extensions of {Ito's} formula}, in S{\'e}minaire de
  Probabilit{\'e}s XV 1979/80, Springer, 1981, pp.~118--141.

\bibitem{Leao-etal-2018}
{\sc D.~Le{\~a}o, A.~Ohashi, and A.~B. Simas}, {\em A weak version of
  path-dependent functional it{\^o} calculus}, The Annals of Probability, 46
  (2018), pp.~3399--3441.

\bibitem{liang2026viscosity}
{\sc D.~Liang and Q.~Meng}, {\em Viscosity solutions of stochastic
  hamilton--jacobi--bellman equations with jumps}, arXiv preprint
  arXiv:2605.20593,  (2026).

\bibitem{LionsSouganidis1998b}
{\sc P.~Lions and P.~Souganidis}, {\em Fully nonlinear stochastic partial
  differential equations: Non-smooth equations and applications}, C.R. Acad.
  Sci. paris, 327 (1998), pp.~735--741.

\bibitem{ma2012non}
{\sc J.~Ma, H.~Yin, and J.~Zhang}, {\em On non-{M}arkovian forward--backward
  {SDEs} and backward stochastic {PDEs}}, Stoch. Process. Appl., 122 (2012),
  pp.~3980--4004.

\bibitem{Malliavin-1978}
{\sc P.~Malliavin}, {\em Stochastic calculus of variation and hypoelliptic
  operators}, in Proc. Intern. Symp. SDE Kyoto 1976, Kinokuniya, Tokyo and
  Wiley, New York, 1978, pp.~195--263.

\bibitem{meng2023optimal}
{\sc Q.~Meng, Y.~Dong, Y.~Shen, and S.~Tang}, {\em Optimal controls of
  stochastic differential equations with jumps and random coefficients:
  stochastic hamilton--jacobi--bellman equations with jumps}, Applied
  Mathematics \& Optimization, 87 (2023), p.~3.

\bibitem{Pardoux1979}
{\sc E.~Pardoux}, {\em Stochastic partial differential equations and filtering
  of diffusion processes}, Stoch.,  (1979), pp.~127--167.

\bibitem{Peng_92}
{\sc S.~Peng}, {\em Stochastic {H}amilton-{J}acobi-{B}ellman equations}, {SIAM}
  J. Control Optim., 30 (1992), pp.~284--304.

\bibitem{Peng-DPP-1997}
{\sc S.~Peng}, {\em Backward stochastic differential equations: Stochastic
  optimization theory and viscosity solution for {HJB} equations}, in Topics on
  Stochastic Analysis (in Chinese), J.-A. Yan, S.~Peng, S.~Fang, and L.~Wu,
  eds., Science Press, Beijing, 1997, pp.~85--138.

\bibitem{peng1999open}
{\sc S.~Peng}, {\em Open problems on backward stochastic differential
  equations}, in Control of Distributed Parameter and Stochastic Systems,
  Springer, 1999, pp.~265--273.

\bibitem{peng2011backward}
\leavevmode\vrule height 2pt depth -1.6pt width 23pt, {\em Backward stochastic
  differential equation, nonlinear expectation and their applications}, in
  Proceedings of the International Congress of Mathematicians, 2010,
  pp.~393--432.

\bibitem{peng2011note}
\leavevmode\vrule height 2pt depth -1.6pt width 23pt, {\em Note on viscosity
  solution of path-dependent {PDE and G-martingales}}, arXiv:1106.1144,
  (2011).

\bibitem{picarelli-reisinger2020-error}
{\sc A.~Picarelli and C.~Reisinger}, {\em Probabilistic error analysis for some
  approximation schemes to optimal control problems}, Systems \& Control Lett.,
  137 (2020), p.~104619.

\bibitem{pP05}
{\sc P.~E. Protter}, {\em Stochastic Integration and Differential Equations},
  Springer, Berlin Heidelberg, 2005.

\bibitem{Qiu2014weak}
{\sc J.~Qiu}, {\em Weak solution for a class of fully nonlinear stochastic
  hamilton--jacobi--bellman equations}, Stoch. Process. Appl., 127 (2017),
  pp.~1926--1959.

\bibitem{Qiu2018viscosity}
\leavevmode\vrule height 2pt depth -1.6pt width 23pt, {\em Viscosity solutions
  of stochastic {Hamilton--Jacobi--Bellman} equations}, SIAM J. Control Optim.,
  56 (2018), pp.~3708--3730.

\bibitem{QiuTangMPBSPDE11}
{\sc J.~Qiu and S.~Tang}, {\em Maximum principles for backward stochastic
  partial differential equations}, J. Funct. Anal., 262 (2012), pp.~2436--2480.

\bibitem{QiuWei-RBSPDE-2013}
{\sc J.~Qiu and W.~Wei}, {\em On the quasi-linear reflected backward stochastic
  partial differential equations}, J. Funct. Anal., 267 (2014), pp.~3598--3656.

\bibitem{qiu2019uniqueness}
{\sc J.~Qiu and W.~Wei}, {\em Uniqueness of viscosity solutions of stochastic
  hamilton-jacobi equations}, Acta Mathematica Scientia, 39 (2019),
  pp.~857--873.

\bibitem{QiuZhang2023games}
{\sc J.~Qiu and J.~Zhang}, {\em Stochastic differential games with random
  coefficients and stochastic hamilton--jacobi--bellman--isaacs equations}, The
  Annals of Applied Probability, 33 (2023), pp.~889--930.

\bibitem{Tang_98}
{\sc S.~Tang}, {\em The maximum principle for partially observed optimal
  control of stochastic differential equations}, \textrm{SIAM} J. Control
  Optim., 36 (1998), pp.~1596--1617.

\bibitem{Tang-Wei-2013}
{\sc S.~Tang and W.~Wei}, {\em On the cauchy problem for backward stochastic
  partial differential equations in {H{\"o}lder} spaces}, Ann. Probab., 44
  (2016), pp.~360--398.

\bibitem{tang2012path}
{\sc S.~Tang and F.~Zhang}, {\em Path-dependent optimal stochastic control and
  viscosity solution of associated {Bellman} equations}, Discrete and
  Continuous Dynamical Systems - Series A, 35 (2015), pp.~5521--5553.

\bibitem{van-neerven-veraar-weis2007}
{\sc J.~M. A.~M. van Neerven, M.~C. Veraar, and L.~Weis}, {\em Stochastic
  integration in {UMD} banach spaces}, Ann. Probab., 35 (2007), pp.~1438--1478.

\bibitem{wagner1977survey}
{\sc D.~H. Wagner}, {\em Survey of measurable selection theorems}, SIAM J.
  Control Optim., 15 (1977), pp.~859--903.

\bibitem{wang1992-II}
{\sc L.~Wang}, {\em On the regularity theory of fully nonlinear parabolic
  equations: {II}}, Commun. Pure Appl. Math., 45 (1992), pp.~141--178.

\bibitem{ye2025fractional}
{\sc Y.~Ye, Y.~Li, and S.~Tang}, {\em Fractional backward stochastic partial
  differential equations with applications to stochastic optimal control of
  partially observed systems driven by l{\'e}vy processes}, SIAM Journal on
  Control and Optimization, 63 (2025), pp.~1314--1347.

\bibitem{yong-zhou1999stochastic}
{\sc J.~Yong and X.~Y. Zhou}, {\em Stochastic controls: {Hamiltonian systems}
  and {HJB} equations}, vol.~43, Springer Science \& Business Media, 1999.

\bibitem{zhou2023viscosity}
{\sc J.~Zhou}, {\em Viscosity solutions to second order path-dependent
  hamilton--jacobi--bellman equations and applications}, The Annals of Applied
  Probability, 33 (2023), pp.~5564--5612.

\bibitem{Zhou_93}
{\sc X.~Zhou}, {\em On the necessary conditions of optimal controls for
  stochastic partial differential equations}, \textrm{SIAM} J. Control Optim.,
  31 (1993), pp.~1462--1478.

\end{thebibliography}

\end{document}